\documentclass[a4paper]{article}
\usepackage[english]{babel}
\usepackage[utf8x]{inputenc}
\usepackage{amsmath}
\usepackage{amssymb}
\usepackage{amsthm}
\usepackage{graphicx}
\usepackage[colorinlistoftodos]{todonotes}
\usepackage{url}
\usepackage[backref=page]{hyperref}
\usepackage{tensor}
\usepackage[export]{adjustbox}
\usepackage{empheq}
\newtheorem{theorem}{Theorem}
\newtheorem{conjecture}{Conjecture}

\title{Generating isospectral but not isomorphic quantum graphs}
\author{Mats-Erik Pistol\\ Solid State Physics\\ Box 118, Lund University, S-221 00 Lund SWEDEN\\ \\ \texttt{mats-erik.pistol@ftf.lth.se}}

\begin{document}
\maketitle
\begin{abstract}
Quantum graphs are defined by having a Laplacian defined on the edges of a metric graph with boundary conditions on each vertex such that the resulting operator, $\mathbf{L}$, is self-adjoint. We primarily use standard boundary conditions. The spectrum of $\mathbf{L}$ does not determine the graph uniquely, that is, there exist non-isomorphic graphs with the same spectra. There are few known examples of pairs of non-isomorphic but isospectral quantum graphs. In this paper we start to correctify this situation by finding hundreds of isospectral sets, using computer algebra.

We have found all sets of isospectral but non-isomorphic equilateral connected quantum graphs with at most nine vertices. This includes thirteen isospectral triplets and one isospectral set of four. One of the isospectral triplets involves a loop where we could prove isospectrality. We also present several different combinatorial methods to generate arbitrarily large sets of isospectral graphs, including infinite graphs in different dimensions. As part of this we have found a method to determine if two vertices have the same Titchmarsh-Weyl $M$-function. We give combinatorial methods to generate sets of graphs with arbitrarily large number of vertices with the same $M$-function. We also find several sets of graphs that are isospectral under more general, permutation invariant, boundary conditions. This necessitates a study of eigenvalue zero where we prove several results.
We discuss the possibilities that our program is incorrect, present our tests and open source it for inspection at \href{http://github.com/meapistol/Spectra-of-graphs}{this url}.

\end{abstract}

\newpage
\tableofcontents

\section{Introduction}
The theory of isospectral manifolds is rich and has a long history \cite{brooks1988constructing, kac1966can,milnor1964eigenvalues,sunada1985riemannian,gordon1992one} where most often the Laplace operator is the relevant operator combined with Dirichlet or Neumann boundary conditions. There are many manifolds which have the same spectrum but are not isometric, which also include subsets of $\mathbb{R}^2$ \cite{gordon1992one}. For quantum graphs it has been shown that if the lengths of the edges are rationally independent, then two graphs having the same spectra must be identical, but if the lengths of the edges are rationally dependent then there exist examples of isospectral, but not isomorphic, quantum graphs \cite{gutkin2001can,kurasov2005inverse,band2006nodal} and the interest in isospectrality is high \cite{Below-isospectral}. We will call such pairs \textit{isospectral pairs}. A problem has been that only very few examples of isospectral pairs have been known, making it difficult to study their properties and to find patterns among them. Band et al. found a method to construct isospectral pairs of quantum graphs \cite{band2009isospectral}. Their examples involved either not only Neumann boundary conditions at terminal vertices or involved disjoint graphs. 

In order to improve the situation we have here searched for isospectral pairs, or more generally isospectral sets, using computer algebra. Our investigations are most often limited to connected equilateral graphs, where all edges have the same length and we leave more general graphs for future study. We initially used standard boundary conditions (defined below) which is the most investigated case. We have found 364 sets of isospectral graphs, which include one isospectral set of four, among all equilateral graphs with at most nine vertices. In addition we have found 51 isospectral pairs among all equilateral tree graphs having at most thirteen vertices. Some of the isospectral sets we found are unusually simple. 

We also present a method to generate isospectral graphs by attaching certain graphs to any compact graph which can generate arbitrarily large sets of isospectral graphs. In order to do so we need to find the Titchmarsh-Weyl $M$-function and this can be done even by hand in some cases. Since the $M$-function is highly important to generate isospectral graphs we have found a method to determine if two vertices have the same $M$-function and implemented it in software. We classify the vertices according to their Titchmarsh-Weyl $M$-function for all isospectral graphs with up to seven vertices. 

Our combinatorial methods allowed us to find an isospectral triplet where one member is the loop graph which is one of the most simple graphs. This method also allowed us to find infinite graphs where periodic and aperiodic graphs have the same spectrum and this can be done in high dimensions. 

There is interest in finding criteria that distinguishes isospectral graphs, such as the number of nodal points \cite{band2006nodal}, scattering properties \cite{Pivovarchik-Mugnolo}, the Titchmarsh-Weyl $M$-function \cite{Gernandt}, or knowing the spectrum under different boundary conditions \cite{Kaliuzhnyi-Verbovetskyi}. We will show that there are plenty of graphs that are isospectral also under different boundary conditions, including $\delta$-type boundary conditions which depend on a parameter. During this investigation we have found all sets of isospectral graphs under Dirichlet boundary conditions at terminal vertices where the graphs have at most eight vertices and at most 13 vertices for trees.

Our results show that there are many interesting isospectral sets and makes it possible to get some general insights about them. The use of our software has allowed us to fairly quickly do experiments with quantum graphs which has led to several of our discoveries and there is likely more to be found.

The increasing interest in isospectral graphs is witnessed by the appearance of three manuscripts after the first version of this manuscript \cite{kurasov.muller,post}, where isospectral graphs are constructed, in one case using the magnetic Laplacian, appearing after version 12 of this manuscript \cite{Fabila-Carrasco.magnetic.2023}.

\section{Laplacians on graphs and their spectra}

We consider only finite compact metric graphs, $ \Gamma $, formed by joining together a set of edges, $ E_n $, at a set of
vertices, $ V_m $. Each edge, $ E_n $, has a certain length and can be seen as the interval $ [x_{2n-1}, x_{2n}] $ on the real line.
On each edge we define the Laplace operator $ \mathbf{L} = - \frac{d^2}{dx^2} $ which has solutions given by a linear combination of $e^{i k x}$ and $e^{-i k x}$. We impose standard boundary conditions (unless otherwise stated) 

\begin{equation} \label{sc1}
\left\{
\begin{array}{ll}
\displaystyle  f(x_i)=f(x_j), & x_i, x_j \in V_m, \\[3mm]
\displaystyle \sum_{x_i{ \in V_m} }\partial_{n}f(x_i)=0. &
\end{array} \right. 
\end{equation}
at each vertex $ V_m$ where the $x_i$'s are the endpoints of the edges that meet at the vertex. In words, the eigenfunctions are required to be continuous at the vertex and the sum of their (outward) normal derivatives, $\partial_{n}f(x_i)$, at the vertex is zero. At terminal (valence one) vertices we call these boundary conditions Neumann boundary conditions. With these boundary conditions the Laplace operator is self-adjoint \cite{berkolaiko2013introduction,kostrykin1999kirchhoff} and has a spectrum which is discrete and formed by a sequence of eigenvalues tending to $+ \infty$. We will denote this self-adjoint operator $\mathbf{L}(\Gamma)$ or just $\mathbf{L}$ if there is no confusion about the domain.
We note that $\lambda_0 = 0$ is an eigenvalue with the eigenfunction $\psi_0 (x) = 1$. This eigenfunction is unique, apart from normalisation, provided $ \Gamma $ is connected. 
Imposing the boundary conditions on the eigenfunctions gives a certain secular equation (or secular determinant), $\Sigma(k)$, that has to be zero in order for $k$ to be a root (which we will often call an \textit{eigenfrequency}), such that $ \lambda = k^2 $ where $\lambda$ is an eigenvalue. How to obtain $\Sigma(k)$, which is not unique, has been described many times before \cite{gutkin2001can, kurasov2005inverse,berkolaiko2013introduction,berkolaiko2017elementary} and we will not repeat it here. We will use the term \textit{boundary condition(s)} to mean any boundary condition(s), not necessarily standard, that make the operator $\mathbf{L}$ self-adjoint. If we say Neumann (or Dirichlet) boundary conditions we mean at terminal vertices and do not always spell out "at pendant edges" or "at terminal vertices".

\section{Computing the eigenvalues}

In order to find the eigenvalues of a graph we wrote a computer program that constructs a $\Sigma(k)$ as a function of the graph. Two graphs are typically isospectral if they have the same $\Sigma(k)$ apart from possible factors that do not contribute any roots. The program optionally solves the equation $\Sigma(k) = 0$ if the graph has rationally dependent edge lengths, and the solutions are given in a symbolic form. If the graph has a pair of rationally independent edges then we seldom get any solutions. Our program is written in Mathematica \cite{Mathematica}.

As an example of the usefulness of our programs we have found that the complete graph of $V=F+1$ vertices has the secular equation:

\begin{equation}
\Sigma(k)=\left(F e^{2 i k}+F+2 e^{i k}\right)^{F} \left(-1+e^{i k}\right)^{P+2} \left(1+e^{i k}\right)^P
\end{equation}

where $P=\frac{1}{2} \left(V^2-3 V \right)$ and where the length of each edge is one. This was tested for up to 15 vertices. Such a result can be found by hand but it will take some effort. More such results will be given below and most of them were obtained by computer.

\section{Testing}

Since our results are highly dependent on our program being correct we have performed a set of tests, often against known results. Specifically:

\begin{itemize}
\item The program gives the correct eigenvalues with the correct multiplicities for the path graph, the loop graph, the lasso graph for different lengths of the pendant edge, and the star graph with $ n $ leaves of equal length \cite{berkolaiko2017elementary, berkolaiko2017simplicity, kennedy2015spectral,Malenova2013}.
\item The program gives the correct eigenvalues with the correct multiplicities for the star graph having three leaves with different lengths \cite{berkolaiko2017simplicity}. 
\item The program gives the correct eigenvalues with the correct multiplicities for the flower graph with two petals with different lengths \cite{berkolaiko2017simplicity,Blixt2015}. 
\item The program gives the correct second eigenvalue (i. e. the spectral gap) for the complete graph with $n$ vertices as well as for the pumpkin graph with $n$ edges of equal length \cite{kennedy2015spectral}.
\item The program correctly gives the same eigenvalues for the two isospectral, but not isomorphic, graphs given by Gutkin and Smilanzky \cite{gutkin2001can}. 
\item The program changes the eigenvalues correctly when scaling the length of the graphs \cite{kennedy2015spectral}. 
\item Consider a graph having a set of vertices with valence two. Let us create a second graph by removing some of these vertices. These two graphs have the same spectrum and this is confirmed by our program in all tested cases.
\item For a few isospectral graphs found by the program we have been able to verify the isospectrality by hand. Details are provided below.
\item We wrote a separate program which constructs the secular equation, $\Sigma(k)$, using bond and edge scattering matrices as described by Berkolaiko \cite{berkolaiko2017elementary}. We tested that our two programs give the same secular equation for one thousand equilateral graphs (and a few that depend on a parameter) and they do.
\end{itemize}

Despite this it is possible that the program is still not correct in all cases and we are not competent to formally prove that it is correct. We open source our programs \cite{pistol2021graphroots} including our test examples. We also include notebooks such that our figures can be reproduced, and independent minds can thus inspect our results and do more tests. We here follow a trend in modern mathematics to use computers to either prove or make results highly likely \cite{hales2012dense,isaksen2020stable,buzzard2020formalising}.

\section{Testing graphs for isospectrality}

We downloaded the one graph with two vertices, the two graphs with three vertices, the six graphs with four vertices, the 21 graphs with five vertices, the 112 graphs with six vertices, the 853 graphs with seven vertices, the 11117 graphs with eight vertices and the 261080 graphs with nine vertices \cite{McKay}. All of the downloaded graphs are connected equilateral graphs. We then computed the spectra for all these graphs and checked for possible isospectral pairs amongst all graphs with at most seven vertices. We only compared graphs having equal number of vertices, meaning that a loop with six vertices is not considered to have a loop with seven vertices as an isospectral partner. Since our program does not give the multiplicities of the eigenvalues explicitly we checked the relevant determinants, $\Sigma(k)$, by hand in order to finally isolate all isospectral pairs. Our program is too slow to directly search for isospectral pairs among the graphs with eight and nine vertices and for them we resorted to a trick. We have empirically found that isospectral graphs have the same characteristic polynomial. We thus preselected graphs with eight and nine vertices by their characteristic polynomial which can be done very fast and subsequently tested for isospectrality on the reduced sets. There are several characteristic polynomials in the literature and the one we used is $C(x) = |T x - A|$ where $A$ is the adjacency matrix and $T$ is a matrix where the diagonal entries are the valencies of the vertices. With $|A|$ we mean the determinant of $A$. If there are isospectral graphs with different characteristic polynomials, then we have likely missed some isospectral sets, but only for graphs having eight or nine vertices. 

Fig. \ref{fig:all7} shows the result for six and seven vertices. We find one isospectral pair with six vertices and five isospectral pairs having seven vertices. Note that in all cases one graph in the pair has a terminal vertex and the other graph does not have a pendant edge. We have found that the isospectral pair in Fig. \ref{fig:all7}f) to be special, in the sense that the members of the other isospectral pairs with seven vertices are subgraphs of one of the members of this pair. These patterns break down for eight vertices and nine vertices as shown in Fig. \ref{fig:afirsttriplets}, where we show the three isospectral triplets with eight vertices and the single isospectral set of four with nine vertices. The first isospectral pair of tree graphs occur with nine vertices, see Fig. \ref{fig:atrees}. In the Appendix we show all isospectral sets of three having nine vertices. The full set of isospectral pairs with eight and nine vertices is too large to display in this manuscript but can be found in Ref. \cite{pistol2021graphroots}. Mathematica sometimes plots graphs with overlapping edges so caution is adviced when checking the figures in Ref. \cite{pistol2021graphroots} and the Mathematica files should be inspected. In Table  \ref{tab:Numberofisospectralpairs} we summarize the number of isospectral sets as a function of the number of vertices of the graphs. We find that isospectral trees are very rare and that the fraction of graphs that belong to an isospectral set decreases with the number of vertices. 

In all honesty, in a moment of (temporary) insanity we also downloaded the 12005168 graphs with ten vertices, but could not test them all for isospectrality due to computer limitations. We doubt we can ever test them all. However, we could test subsets of them for isospectrality and we give the results in Ref. \cite{pistol2021graphroots}. There are many fascinating sets of isospectral graphs with ten vertices and we give a few examples below.

\begin{figure}
\centering
\includegraphics[width=1.0\textwidth]{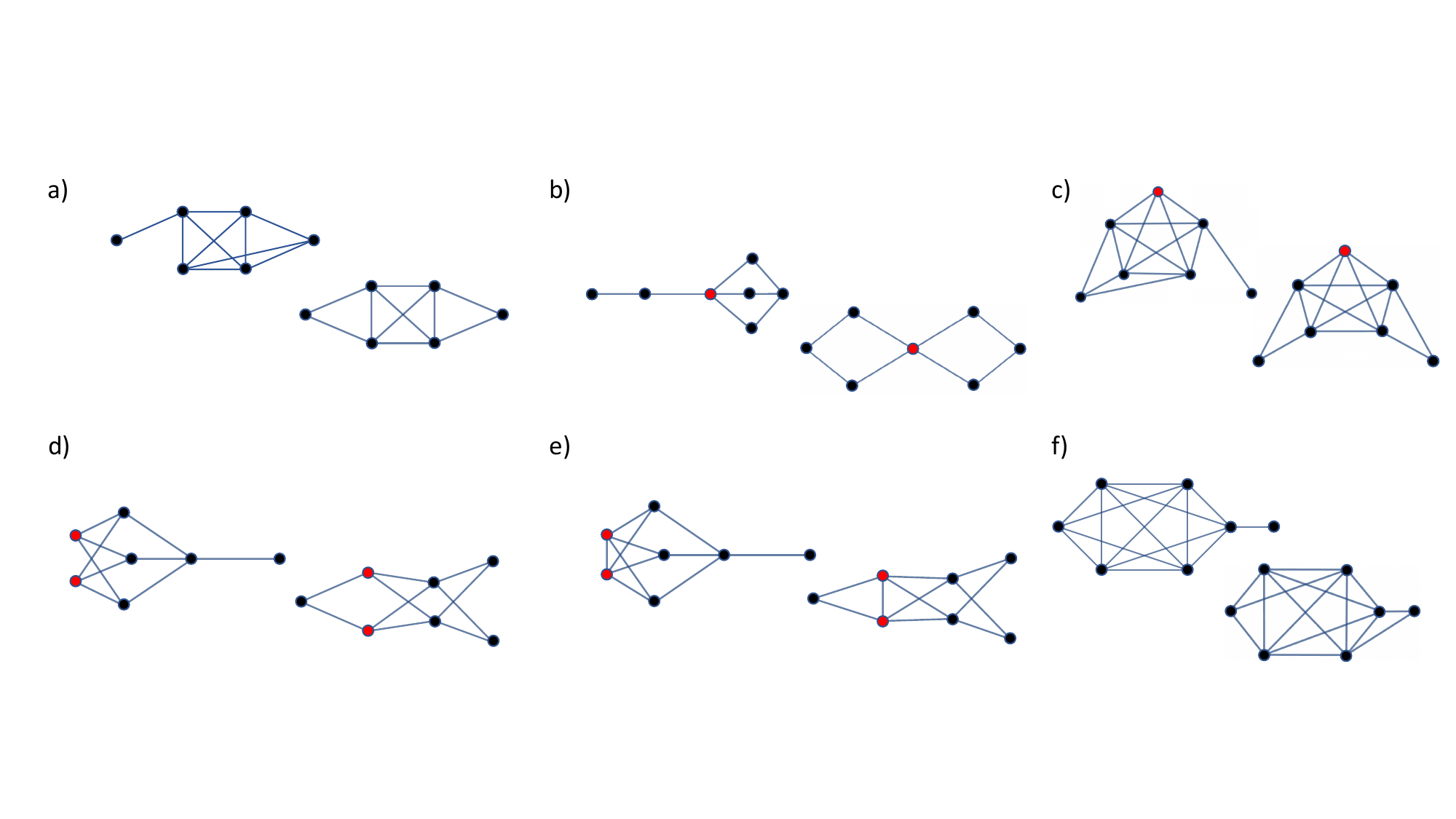}
\caption{ All isospectral pairs with at most seven vertices. The edge length is one for all graphs. a) The one isospectral pair with six vertices. b) - f) The five isospectral pairs with seven vertices. The pairs in e) have one extra edge compared with those in d). In all cases one member of the pair has a terminal vertex and the other not. All graphs in b) - e) are subgraphs of a member of f). Isospectral graphs sometimes have vertices with the same $M$-functions. Such vertices are red. $M$-functions are described in the main text.}
\label{fig:all7}
\end{figure}

We also studied tree graphs, hereafter called \textit{trees}, in more detail and generated all trees with at most 13 vertices. Checking for isospectral pairs we found 51 examples, where the first example had nine vertices. In Fig. \ref{fig:atrees} we show the one isospectral pair having nine vertices and the two having ten vertices. The isospectral pair from Ref. 6, where the trees have eight vertices, was not detected as an isospectral pair by our program since the edge lengths are not equal. We found five isospectral pairs with 11 vertices and six isospectral pairs with 12 vertices, shown in the Appendix. The remaining 37 isospectral pairs with 13 vertices are given in Ref. \cite{pistol2021graphroots}. 

Our results agree with those of Chernyshenko and Pivovarchik who did not find isospectral pairs for equilateral graphs having at most five vertices and equilateral trees having at most eight vertices \cite{chernyshenko2020recovering}. 

\begin{figure}[ht]
\centering
\includegraphics[width=1.0\textwidth]{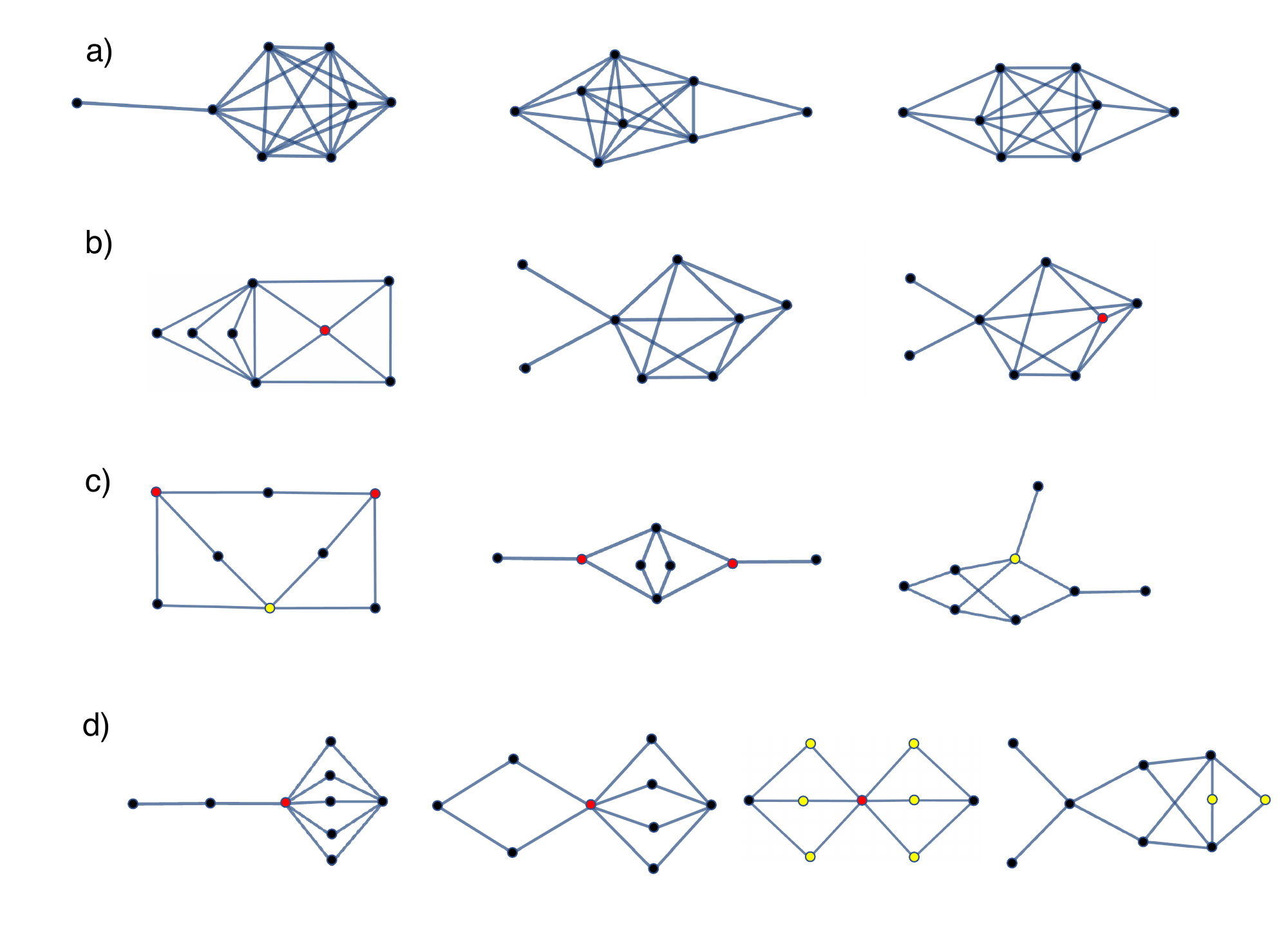}
\caption{a)-c) The three isospectral triplets with eight vertices. Two vertices among the graphs in b) have the same $M$-function, red. The red vertices in c) have the same $M$-function say $M_1$, and so do the yellow ones, say $M_2$. $M_1 \neq M_2$. d) The single isospectral set of four with nine vertices. Also here we have two sets of vertices having the same $M$-function if they have the same colour (except black). There are graphs with eight vertices belonging to isospectral sets that are not a subgraph of any member of a). The graphs have been checked such that there are no overlapping edges.}
\label{fig:afirsttriplets}
\end{figure}

\begin{figure}
\centering
\includegraphics[width=1.0\textwidth]{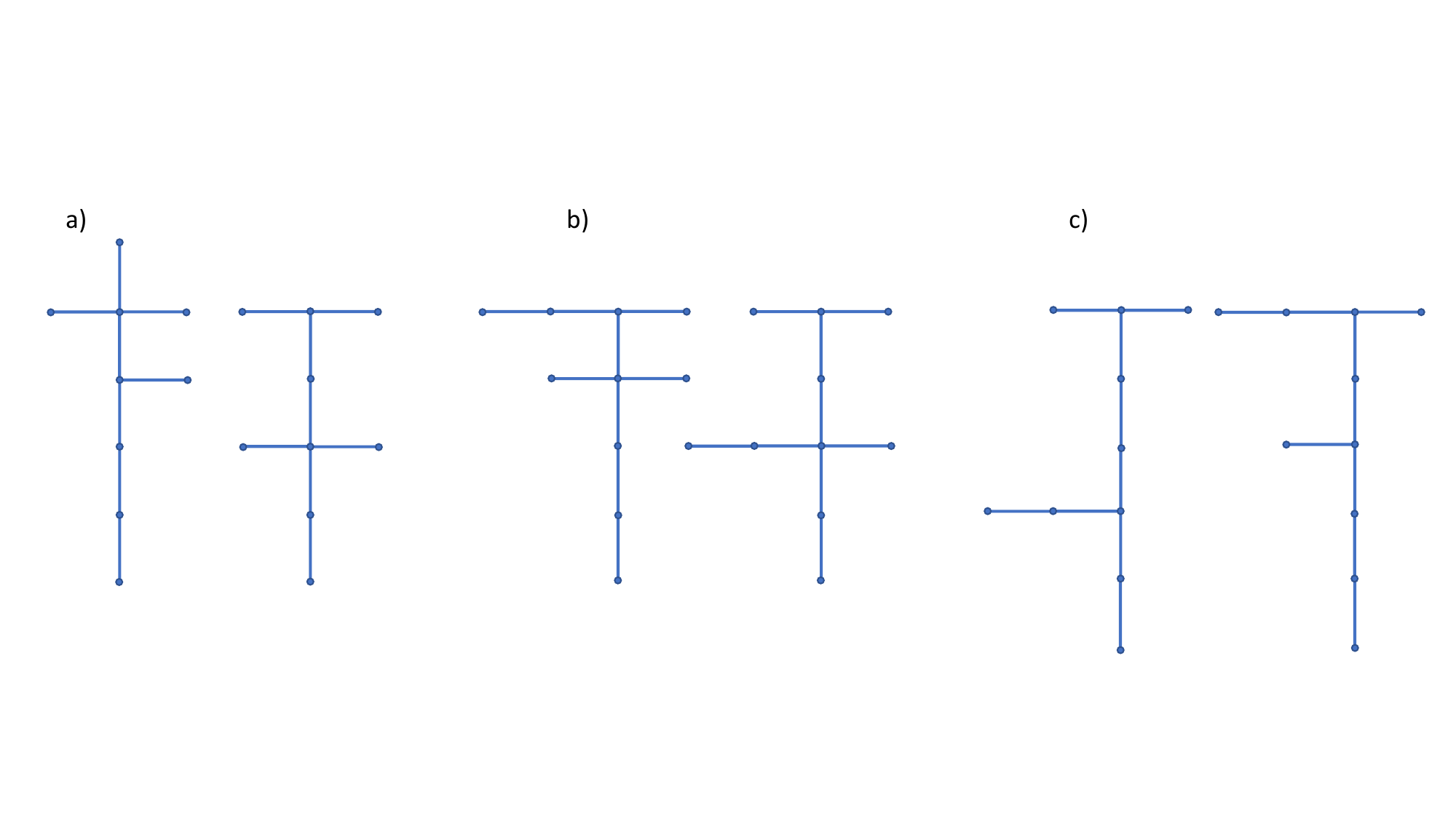}
\caption{All isospectral pairs of trees with at most ten vertices. a) The isospectral pair with nine vertices. b), c) The two isospectral pairs having ten vertices. These graphs have no two vertices with the same $M$-function within each pair.}
\label{fig:atrees}
\end{figure}

\begin{table}[]
\centering
\caption{\label{Table 1:all} The number of equilateral graphs, isospectral pairs (which includes trees), isospectral triplets, isospectral sets of four and isospectral pairs of trees as a function of the number of vertices. An empty field means we don't have data.
The fraction of graphs that have at least one isospectral partner decreases from 0.018 to 0.0025 as the number of vertices goes from 6 to 9.}
\begin{tabular}{llllll}
         &            &       &          &              &       \\ \hline
Vertices & Graphs     & Pairs & Triplets & Sets of four & Pairs of trees \\ \hline
6        & 112        & 1     & 0        & 0            & 0     \\
7        & 853        & 5     & 0        & 0            & 0     \\
8        & 11117      & 39    & 3        & 0            & 0     \\
9        & 261080     & 304   & 10       & 1            & 1     \\
10       & 12005168   &       &          &              & 2     \\
11       & 1018997864 &       &          &              & 5     \\
12       &            &       &          &              & 6     \\
13       &            &       &          &              & 37 \\  \hline
\end{tabular}
 \label{tab:Numberofisospectralpairs}
\end{table}

\begin{figure}[ht]
\centering
\includegraphics[width=1.0\textwidth]{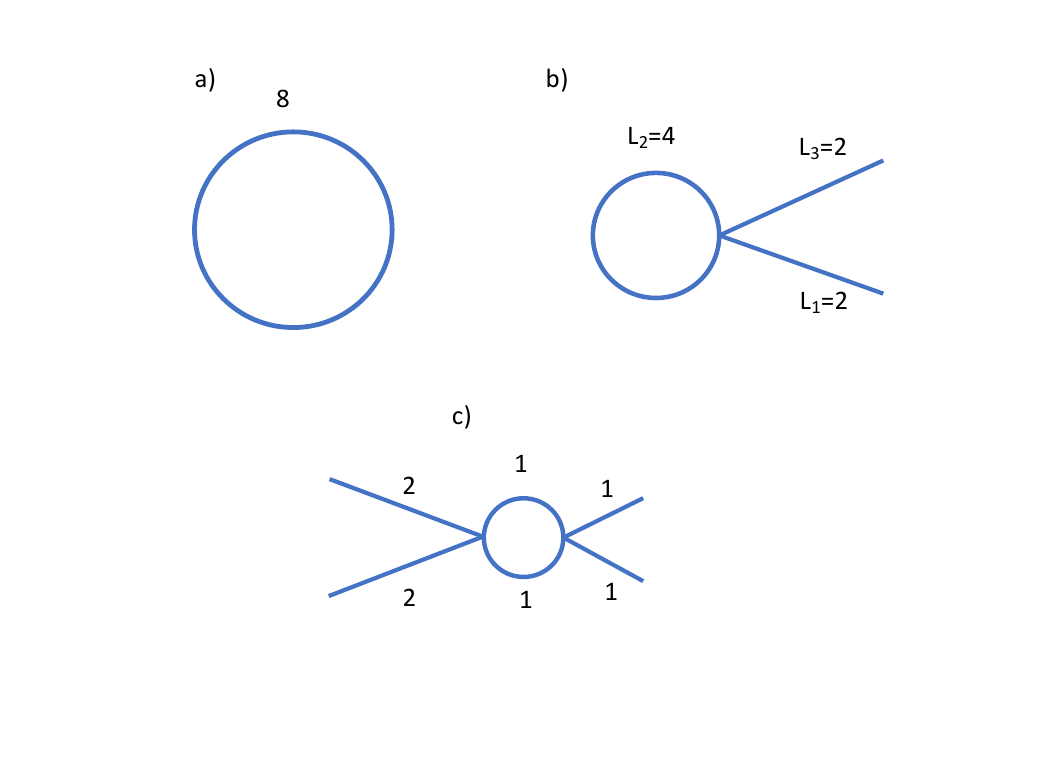}
\caption{Three very simple isospectral graphs. a) A loop graph having length eight. b) The first isospectral partner of the loop graph when $L_1=L_3=2$ and $L_2=4$ such that the total length of the graph is 8. c) The second isospectral partner of the loop.}
\label{fig:aloop} 
\end{figure}

Encouraged by these examples of isospectral pairs we then generated graphs consisting of a loop with four vertices which we decorated with terminal vertices or pendant trees such that the total number of vertices was at most ten. We found 22 isospectral pairs and one isospectral triplet. Some of these isospectral sets are very simple, in particular the isospectral triplet in Fig. \ref{fig:aloop}. This triplet involves a loop which is very well known. The other two members is a loop decorated with two pendant edges and a loop decorated with four pendant edges. The graph in Fig. \ref{fig:aloop}c) was not found during the search but using the combinatorial method to be described below. We also generated graphs consisting of a loop with six vertices which we decorated with pendant edges or pendant trees such that the full graph has at most 12 vertices and the isospectral sets are given in the Appendix.

We decided to compute the eigenfrequencies of the graphs in Fig. \ref{fig:aloop} by hand in order to give our program an extra check as well as to prove that the graphs really are isospectral. We have the following theorem. \\

\begin{theorem}
The loop has (at least) two isospectral partners. They are shown in Fig. \ref{fig:aloop}.
\end{theorem}

\begin{proof} 
In order to prove this we followed the method explained in detail by Berkolaiko in Ref. \cite{berkolaiko2017elementary}, which involves a bond scattering matrix, $S_v$, and an edge scattering matrix, $S_e(k)$. These matrices are as follows for the graph in Fig. \ref{fig:aloop}b):

\begin{equation}
S_v =
\left(
\begin{array}{cccccc}
 0 & -\frac{1}{2} & \frac{1}{2} & \frac{1}{2} & 0 & \frac{1}{2} \\
 1 & 0 & 0 & 0 & 0 & 0 \\
 0 & \frac{1}{2} & \frac{1}{2} & -\frac{1}{2} & 0 & \frac{1}{2} \\
 0 & \frac{1}{2} & -\frac{1}{2} & \frac{1}{2} & 0 & \frac{1}{2} \\
 0 & \frac{1}{2} & \frac{1}{2} & \frac{1}{2} & 0 & -\frac{1}{2} \\
 0 & 0 & 0 & 0 & 1 & 0 \\
\end{array}
\right) 
\end{equation}

\begin{equation}
S_e(k) = \left(
\begin{array}{cccccc}
 e^{i k L_1} & 0 & 0 & 0 & 0 & 0 \\
 0 & e^{i k L_1} & 0 & 0 & 0 & 0 \\
 0 & 0 & e^{i k L_2} & 0 & 0 & 0 \\
 0 & 0 & 0 & e^{i k L_2} & 0 & 0 \\
 0 & 0 & 0 & 0 & e^{i k L_3} & 0 \\
 0 & 0 & 0 & 0 & 0 & e^{i k L_3} \\
\end{array}
\right)
\end{equation}
 
The edges were ordered - one pendant edge, the loop, and the other pendant edge. Each edge labels two rows and two columns, since directed edges are used in the construction of $S_v$.
The eigenfrequencies are found by solving $\Sigma (k)=\text{Det}(I-S_v S_e(k)) = 0$. $\Sigma (k)$ is the secular equation. If we set $L_1=L_3=1/4$ and $L_2 = 1/2$ we get $\Sigma (k)=\left(e^{i k}-1\right)^2=0$ which has solutions $k = 2 \pi N$ with multiplicity two and where $N$ is a positive integer. These are precisely the non-zero eigenfrequencies with the correct multiplicities for a loop with length one. The last graph in Fig. \ref{fig:aloop} was also checked by hand and was confirmed to have the same secular equation as a loop with length eight. The details are given in Ref. \cite{pistol2021graphroots}. 
\end{proof}

We found some isospectral pairs involving fairly simple graphs, and three examples are given in Fig. \ref{fig:alooptriplets1}. For the pair in Figs. \ref{fig:alooptriplets1}c) and \ref{fig:alooptriplets1}d) we checked the eigenfrequencies once again by hand and found that the graphs have the same secular equation. The calculation is given in the Appendix. We also found two other sets of isospectral triplets and these sets are given in Fig. \ref{fig:alooptriplets2} and Fig. \ref{fig:alooptriplets3}.

\begin{figure}[ht]
\centering
\includegraphics[width=0.6\textwidth]{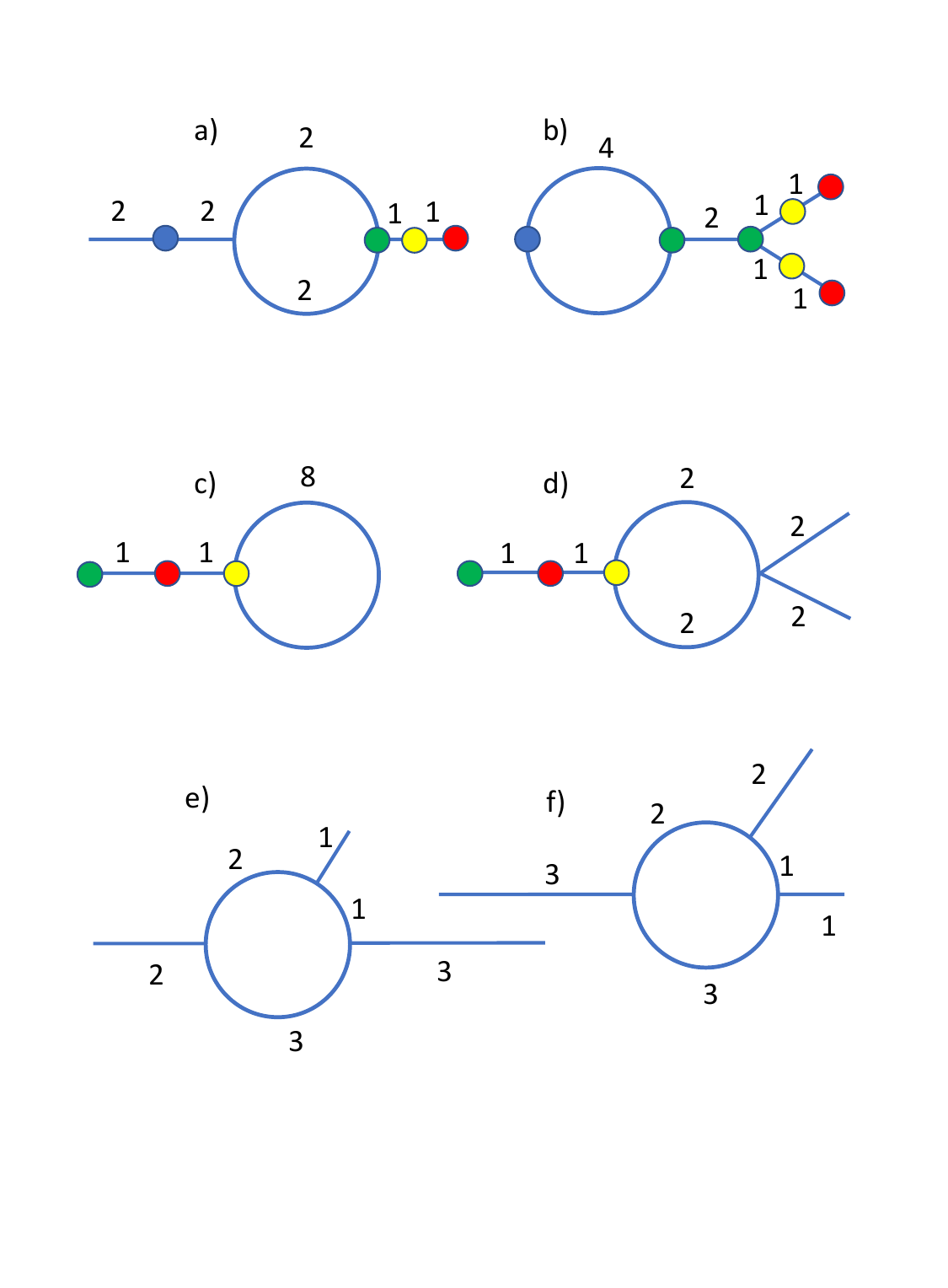}
\caption{Three simple isospectral pairs where a) is isospectral with b), c) is isospectral with d) and e) is isospectral with f). The length of the edges connecting any two vertices are indicated. Vertices with the same colour have the same $M$-function within each isospectral pair. The isospectral pair in e-f) does not have any hot vertices seen as equilateral graphs with 12 vertices.}
\label{fig:alooptriplets1} 
\end{figure}

In the Appendix we give a larger set of isospectral pairs including graphs with up to twelve vertices and even more examples are given in Ref. \cite{pistol2021graphroots}. Most of the graphs in Figs. \ref{fig:alooptriplets1}-\ref{fig:alooptriplets3} were created from corresponding graphs in Figs. \ref{fig:A1}-\ref{fig:A2} where vertices with valence (or degree) two were removed.

\begin{figure}[ht]
\centering
\includegraphics[width=0.8\textwidth]{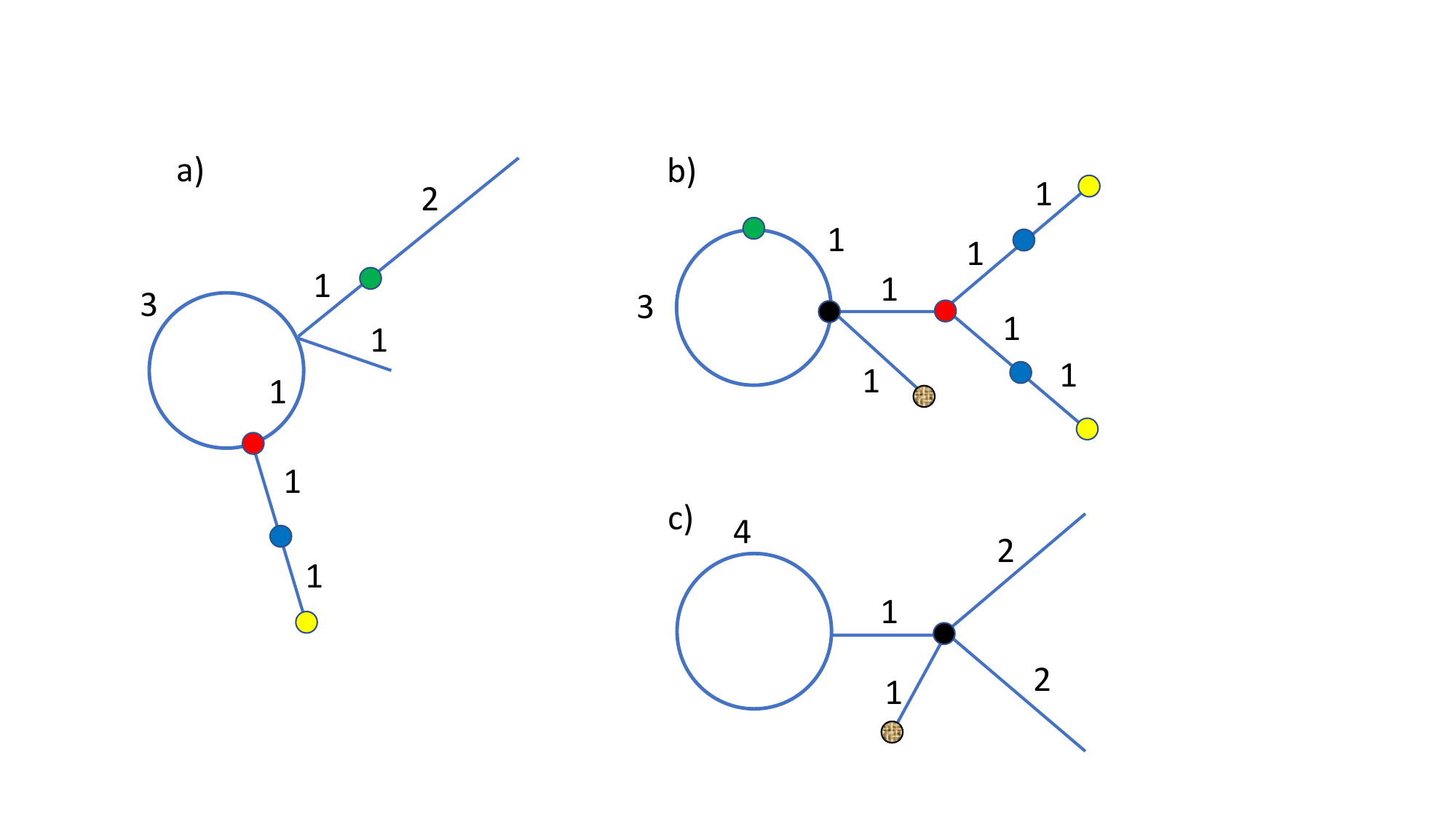}
\caption{ A set of three isospectral graphs. All the loops have a total length of four and the two vertices on the loop for the graph in a) are separated by length one. Vertices with the same colour have the same $M$-function.}
\label{fig:alooptriplets2}
\end{figure}

\begin{figure}[ht]
\centering
\includegraphics[width=0.8\textwidth]{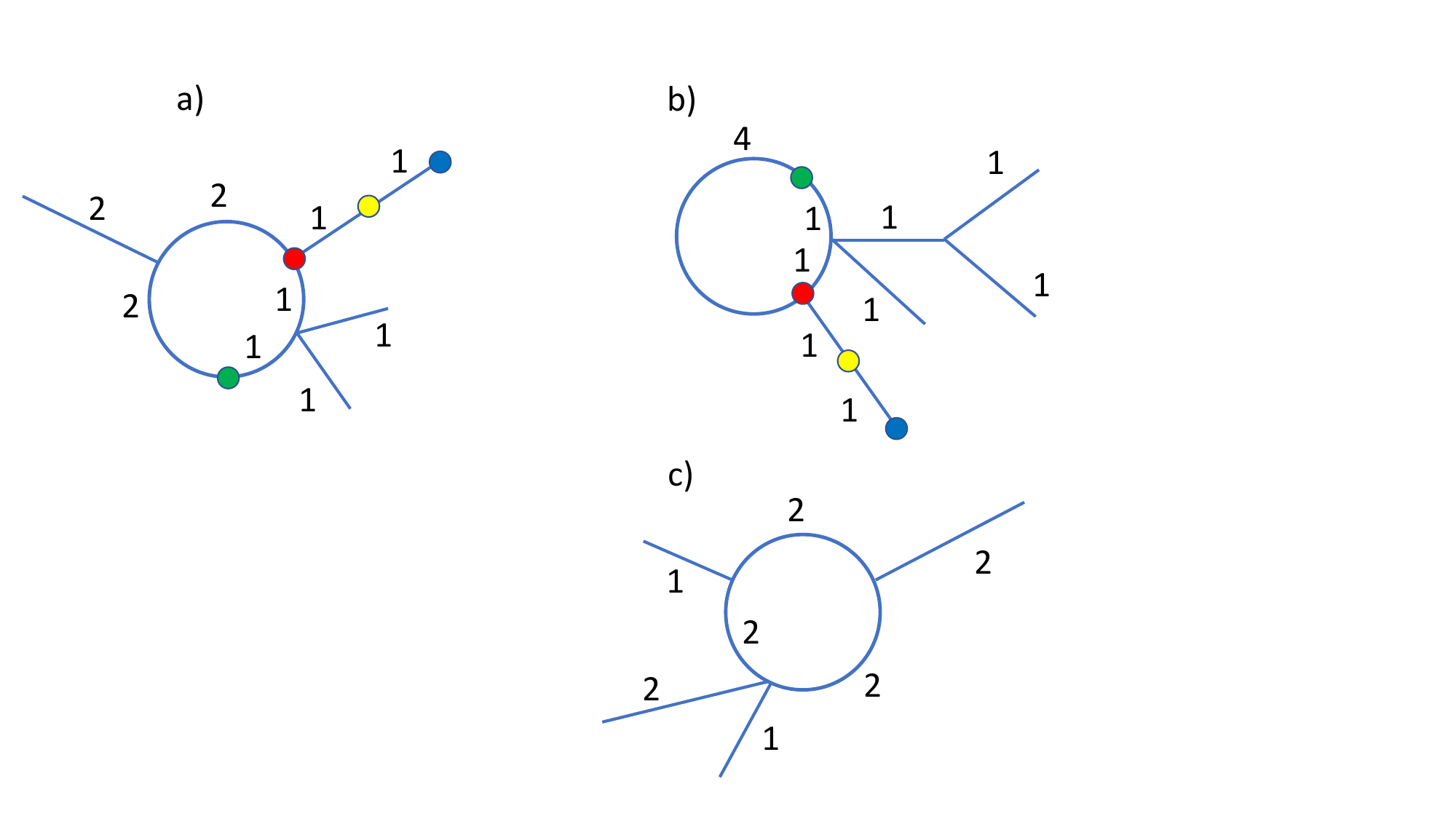}
\caption{A set of three isospectral graphs. The loops all have total length six. Vertices with the same colour have the same $M$-function. The graph in c) does not have any hot vertices seen as an equilateral graph with 12 vertices }
\label{fig:alooptriplets3} 
\end{figure}

\section{Combinatorial construction of isospectral\\ graphs}

There is a combinatorial method to easily generate large numbers of isospectral pairs and, more generally, isospectral sets of graphs containing many members. 

This is done by attaching any compact graph to specific vertices on graphs belonging to some particular isospectral set. These vertices have the same Titchmarsh-Weyl $M$-functions, to be described below. Fig. \ref{fig:aattach} illustrates the method. This isospectral pair is the same as in Figs. \ref{fig:aloop}a) and b) but we set their length to four for graphic simplicity. We add a vertex of valence two at the indicated positions. If we now attach any compact graph, having any boundary conditions, at these positions we find that the resulting graphs will form an isospectral pair. The attachment is done by identifying one vertex of $\Gamma$ with the indicated vertex, creating a pendant graph, as exemplified in Fig. \ref{fig:aattach}c-d). The attachment must be done identically to each member of the pair. We here see that although two non-isomorphic graphs with rationally independent edge lengths must have different spectra, they can have isospectral partners which do not have rationally independent edge lengths, as explained in the figure caption. The two shown graphs are the members of a generating set that can generate an infinite family of isospectral pairs. We will sometimes call the indicated vertex a \textit{hot vertex}. The vertex of $\Gamma$ can have any valence. 

The attached graph can have infinite length and the frequency-dependent backscattering from two vertices with the same $M$-function will be the same. It is thus possible to experimentally check that two vertices have the same $M$-function which connects to experimental work on quantum graphs \cite{Hul}. Mugnolo and Pivovarchik have shown that it is possible to distinguish certain isospectral graphs by attaching infinite leads at different vertices and investigate the scattering properties of the resulting graphs \cite{Pivovarchik-Mugnolo}. This method will not work if the vertices have the same $M$-function. 

The graph in Fig. \ref{fig:aloop}c) does not seem to have a hot vertex. We denote that two graphs $\Gamma_i$ and $\Gamma_j$ are isospectral by: $\Gamma_i \simeq_{is} \Gamma_j$. 

\begin{figure}
\centering
\includegraphics[width=0.8\textwidth]{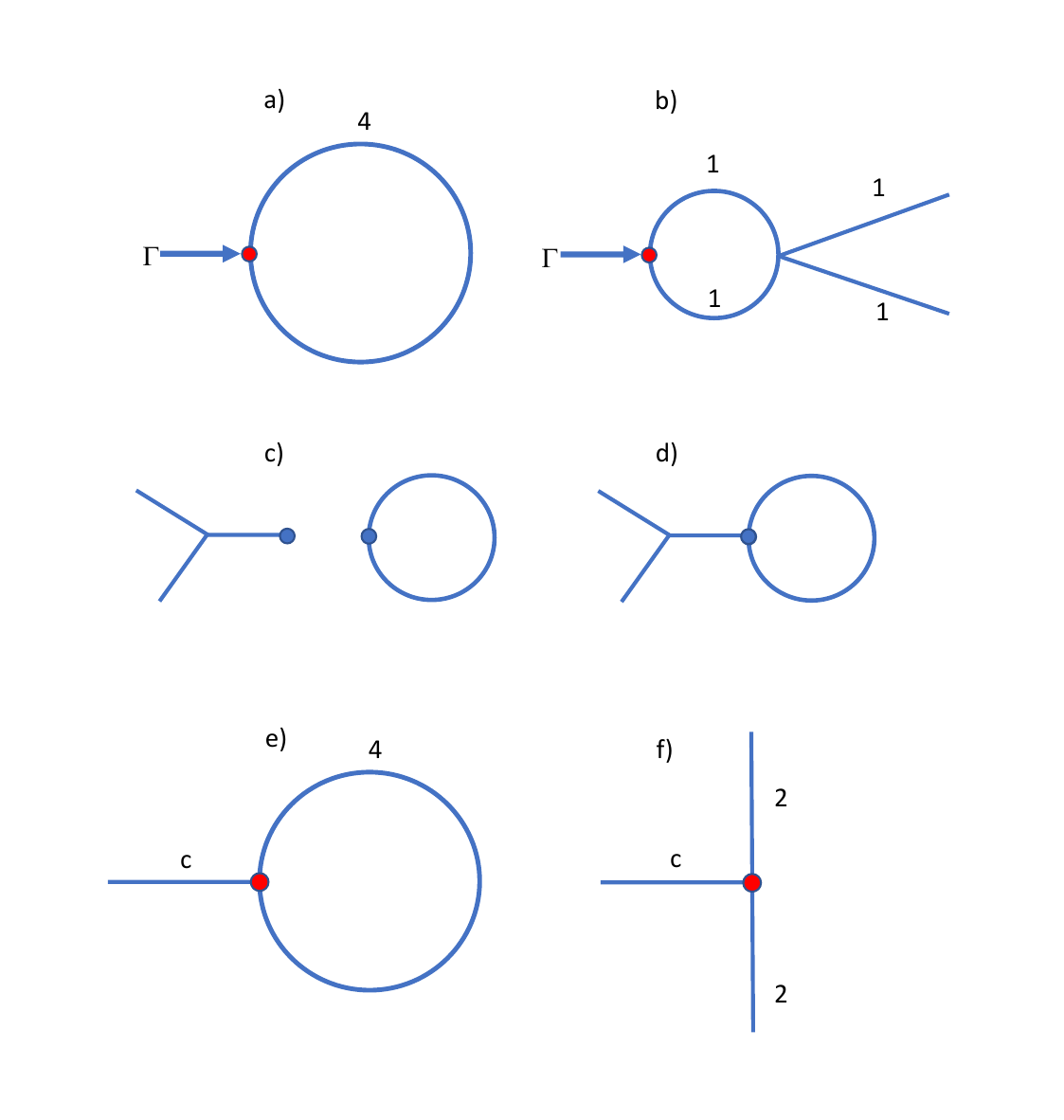}
\caption{Two isospectral graphs from Fig. \ref{fig:aloop}a) and b) with a vertex of valence two added. If any compact graph $\Gamma$ is attached to the valence two vertex (as indicated by the arrows) the resulting graphs will form an isospectral pair. These two graphs form a generating set for an infinite number of isospectral pairs. c) and d) shows how attachment of two graphs is done. The two highlighted vertices are identified with each other. Note that if $\Gamma$ is an interval of length, say $\pi$, and thus graph a) has rationally independent edges, it still has an isospectral partner b) which has rationally dependent edges. e) and f) shows a loop and an interval where we have attached an interval of length c. These graphs are used to illustrate how we can determine that the attachment vertices have the same $M$-function, explained in the main text.}
\label{fig:aattach} 
\end{figure}

\begin{figure}
\centering
\includegraphics[width=0.8\textwidth]{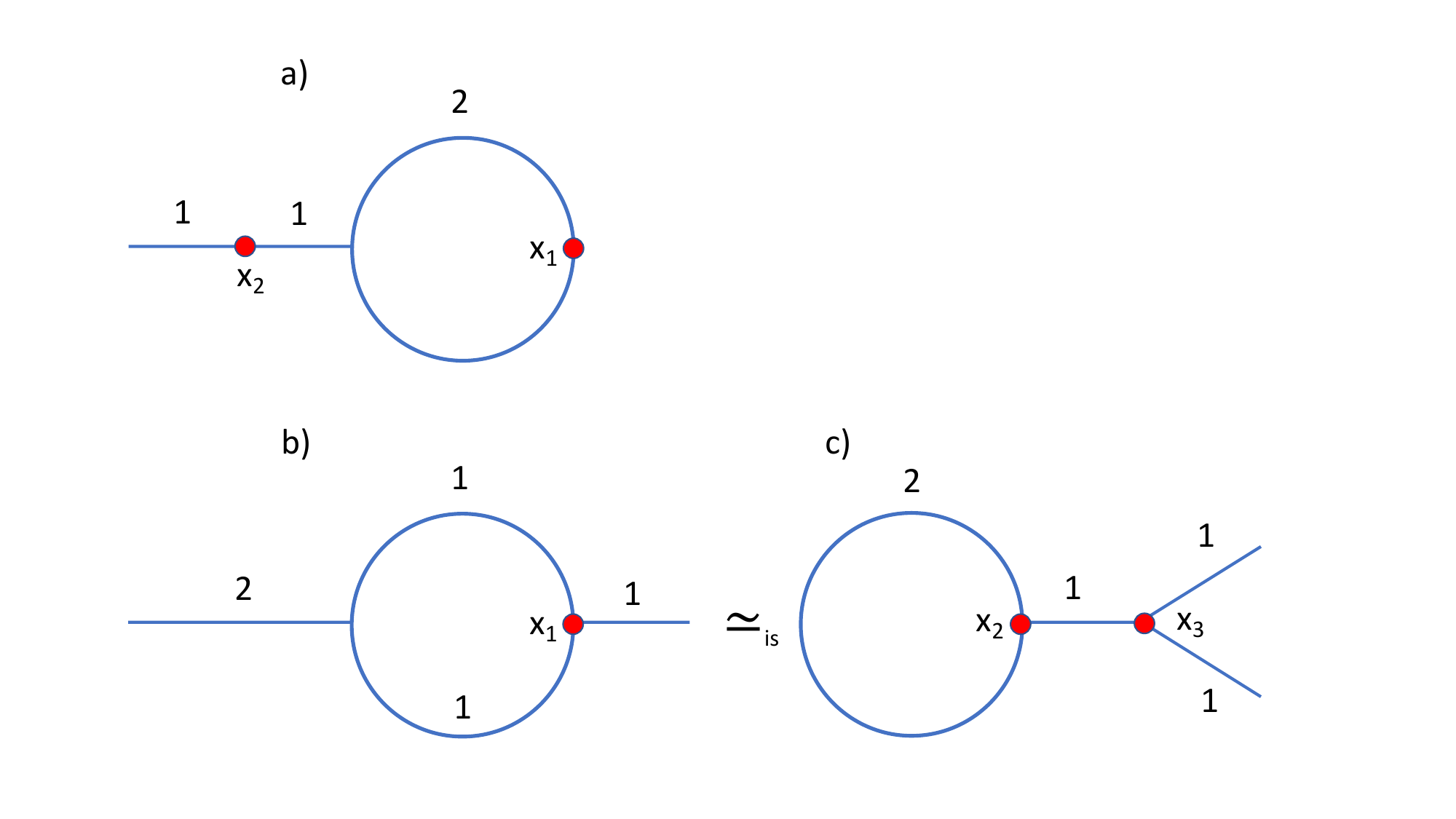}
\caption{a) For this graph we can attach any compact graph to either $x_1$ or $x_2$ and the generated graphs will be isospectral. b-c) Attaching an interval of length one to the graph in a) at $x_1$ or $x_2$ will generate these two isospectral graphs. These two graphs have three hot vertices in total. It is thus easy to generate isospectral triplets by attaching any compact graph to either $x_1$, $x_2$, or $x_3$.}
\label{fig:bhotvertexdemo}
\end{figure}

We remark that any graph containing a loop attached to a single vertex has an isospectral partner since the loop can be replaced with the isospectral partner of the loop shown in Fig. \ref{fig:aattach}b). 

\subsection{Titchmarsh-Weyl $M$-functions}

Our construction can be understood by Titchmarsh-Weyl $M$-functions (also known as Dirichlet to Neumann maps and often called $M$-functions) which are described in \cite{berkolaiko2013introduction} and \cite{kurasov.muller}. Let us consider a graph, $\Gamma$ with one special boundary (or contact) vertex, $\partial\Gamma$. The other vertices are called {\textit{interior vertices}}. Let us also consider a function $u(k)$ which satisfies the differential $\bf{L}$ and is thus a linear combination of $e^{i k x}$ and $e^{-i k x}$ at each edge. Let $u(k)$ further satisfy some boundary conditions, such as standard boundary conditions, on the interior vertices and be continuous on $\partial\Gamma$. $u(k)$ will in general not satisfy the standard or Neumann boundary condition at $\partial\Gamma$, unless $k$ is an eigenfrequency of $\Gamma$. Instead there will be a function, $M^\Gamma(k)$, which relates the value of $u(k)$ at $\partial\Gamma$, which we denote $u(k,\partial\Gamma)$, with the sum of its (outward) derivatives at $\partial\Gamma$, which we denote $u'(k,\partial\Gamma)$:

\begin{equation}
M^\Gamma(k) u(k,\partial\Gamma) = u'(k,\partial\Gamma) 
\end{equation}

If $u(k,\partial\Gamma)=0$ then $M^\Gamma(k)$ is not defined, and if $M^\Gamma(k)=0$ then $k$ is an eigenfrequency of $\mathbf{L}(\Gamma)$, if we demand that the boundary conditions at $\partial\Gamma$ are standard. There might be other eigenfrequencies that are not detected by $M^\Gamma(k)$, for instance because the corresponding eigenfunction is localised away from $\partial\Gamma$. We will give examples of such eigenfunctions below. It is usual to define the $M$-function when there are several boundary vertices, in which case the $M$-function is a matrix-valued function. But for us it suffices to consider only one boundary vertex. It is also possible to use boundary conditions at $\partial\Gamma$ which are not standard, but we will not consider this possibility either in this paper.

It was noted explicitly by Kurasov and Muller \cite{kurasov.muller} and implicitly by Berkolaiko and Kuchment \cite{berkolaiko2013introduction} that if we have a given graph, $\Gamma$, and a pair of isospectral graphs, $\Gamma_1$ and $\Gamma_2$, where $\Gamma_1$ and $\Gamma_2$ have the same $M$-functions at some boundary vertices, $\partial\Gamma_1$ and $\partial\Gamma_2$, then we can attach $\Gamma$ to either $\Gamma_1$ or $\Gamma_2$ at their respective boundary vertex and get an isospectral pair. 

Let the attachment vertex of $\Gamma$ be $\partial\Gamma$ and let us call the resulting graph after attachment $\Gamma \sqcup \Gamma_1$, if we have attached to $\Gamma_1$ and $\Gamma \sqcup \Gamma_2$, if we have attached to $\Gamma_2$. Let us call the identified vertex $\partial(\Gamma \sqcup \Gamma_1)$ or $\partial(\Gamma \sqcup \Gamma_2)$ depending on which graph was attached. It is clear that:
\begin{equation}
M^{\Gamma \sqcup \Gamma_1}(k) = M^\Gamma(k)+M^{\Gamma_1}(k)= M^\Gamma(k)+M^{\Gamma_2}(k) = M^{\Gamma \sqcup \Gamma_2}(k)
\end{equation}

The eigenfrequencies of $\Gamma \sqcup \Gamma_1$ and $\Gamma \sqcup \Gamma_2$ are thus the same, since they are determined by $M^{\Gamma \sqcup \Gamma_1}(k) =M^{\Gamma \sqcup \Gamma_2}(k) = 0$. Any eigenfrequency with an eigenfunction which is not detected by the $M$-function will be unaffected by the attachment.

If $\Gamma_1$ and $\Gamma_2$ are not isospectral then we will not necessarily get an isospectral pair if we attach only one graph to $\Gamma$. This is because it is possible that an eigenfunction has support only inside $\Gamma_1$ or $\Gamma_2$. We need to attach both $\Gamma_1$ and $\Gamma_2$ to $\Gamma$ at two vertices in the two different ways possible. The two resulting graphs are then isospectral. We illustrate this in a general way in Fig. \ref{fig:bdecoration}.

\subsection{Determining $M$-functions}

We will now describe a method to find vertices that have the same $M$-functions. We do this by attaching the end of an interval of length c, $\Gamma_c$, to the relevant vertex of the investigated graph, $\Gamma$. The $M$-function at $\partial\Gamma_c$ is $M^{\Gamma_c}(k)=k \text{Tan}(c k)$ and if $k$ is an eigenfrequency of $\Gamma \sqcup \Gamma_c$ we have $M^{\Gamma_c}(k)+M^{\Gamma}(k)=0$, which allows us to deduce $M^{\Gamma}(k)$ at these $k$-values. By varying $c$ we can get $M^{\Gamma}(k)$ in some interval on ${\Bbb R}$. Since $M^{\Gamma}(k)$ is an analytic function except at its poles we can analytically continue $M^{\Gamma}(k)$ to the whole complex plane, except at the poles. Furthermore, if two vertices (on the same or different graphs) have the same $M$-function in some interval, then the $M$-function is the same for these vertices. To calculate the $M$-function in this way is still difficult and Mathematica does not give explicit solutions for most $c$. However our interest is to determine if two vertices have the same $M$-function and this we can always do by inspecting the secular equation for $\Gamma \sqcup \Gamma_c$ for the two vertices and confirm that the $c$-dependent part is the same. The secular equation may have roots that do not depend on $c$, meaning that the corresponding eigenfunctions are not detectable by the $M$-function.

We illustrate our method with an example. Let us attach an interval, $\Gamma_c$, to a loop and to the center of an interval, as illustrated in Fig. \ref{fig:aattach}e-f). The secular equations are:

\begin{equation}
\Sigma (k)=\left(-1+e^{4 i k}\right) \left(-e^{2 i \text{c} k}+3 e^{2 i (\text{c}+2) k}+e^{4 i k}-3\right)
\end{equation}
for the graph in Fig. \ref{fig:aattach}e) and
\begin{equation}
\Sigma (k)=\left(1+e^{4 i k}\right) \left(-e^{2 i \text{c} k}+3 e^{2 i (\text{c}+2) k}+e^{4 i k}-3\right)
\end{equation}
for the graph in Fig. \ref{fig:aattach}f). The first factors in the secular equations do not depend on $c$ and correspond to eigenfunctions which have a node at the attachment vertices. These eigenfunctions are not detectable by the $M$-function and the corresponding eigenvalues are different for the loop and the interval. The second factor depends on $c$ and is the same for the two graphs. The $M$-function at the attachment vertices are thus the same for the loop and for the interval despite them not being isospectral. This we can verify by hand calculations.

For most of our graphs we have tested for vertices with the same $M$-function, but we have typically not introduced vertices of valence two on the edges. We may thus have missed some valence two vertices that have the same $M$-function. Graphs that are not isospectral have seldomly been tested for vertices with the same $M$-function. Thus, to be overly clear about the figures - vertices with the same colour have the same $M$-function only if they belong to one graph or to isospectral graphs. There are some exceptions, such as the loop and the interval, and those are always noted. Furthermore there are vertices with the same $M$-function due to obvious symmetries and those are seldom plotted. Doing so would overload the figures with meaningless information.

\section{More generating sets}

We have found that the tadpole graph shown in Fig. \ref{fig:bhotvertexdemo}a) has two hot attachment vertices. If the attached graph is an interval with unit length we get the two isospectral graphs shown in Fig. \ref{fig:bhotvertexdemo}b-c). These two graphs have three hot vertices, such that if we attach any compact graph to $x_1$, $x_2$ or $x_3$ the resulting graphs have the same spectrum, Why this is the case is explained below (Fig. \ref{fig:bmirror}). If we attach a an interval to the graph in Fig. \ref{fig:bhotvertexdemo}a) at the two hot vertices we will generate an isospectral pair. But this pair will in general only have two hot vertices ($x_1$ and $x_3$). We thus have two generating sets in Fig. \ref{fig:bhotvertexdemo}, one consisting of a single graph and the second consisting of two graphs. We will call the set in Fig. \ref{fig:bhotvertexdemo}b-c) $Q$ and the graph in Fig. \ref{fig:bhotvertexdemo}a) we denote $R$ and we will see it as a set with two members, differing only in the attachment vertex. The set thus contains two "pointed graphs", where a "pointed graph" is the set consisting of a graph and its (single) attachment vertex. It is also somewhat more convenient to consider the set $Q$ to have three pointed graphs as members. 

%We here note that although two non-isomorphic graphs with rationally independent edge lengths must have different spectra, they can have isospectral partners which do not have rationally independent edge lengths.

So far we have discussed three generating sets, shown in Fig. \ref{fig:aattach} and in Fig. \ref{fig:bhotvertexdemo} where the members are isospectral. We have found a very flexible generating set where the members are not all isospectral to each other. Fig. \ref{fig:btadpools} shows this set, which contains a loop, an interval, tadpool graphs and chains of loops. The chain of loops can contain any set of loops as long as the total length is one. These graphs have one attachment vertex associated with them, except the chain of loops which has two. The attachment vertices have all the same $M$-function. We will call the set $P(L)$, where $L$ is the length of the graphs. Fig. \ref{fig:btadpools}a) shows a subset of $P(12)$ and Fig. \ref{fig:btadpools}b) shows the members of $P(1)$. $P(1)$ contains the parametrised tadpole graphs along with the loop, which we sometimes call $\Gamma_L$ and the interval which we sometimes call $\Gamma_I$. One example of a chain of loops is also shown. The tadpole graphs can be seen to "interpolate" between the loop and the interval. Unless otherwise stated we choose the length of the graphs to be one, in which case the length parameter might be omitted. The generating set in Fig. \ref{fig:aattach} is a subset of $P(4)$. 

Decorating graphs with members of $P$, $Q$ or $R$ gives a very powerful way to generate isospectral graphs and we will illustrate some results. In Fig. \ref{fig:bdecoration} we show some examples of graphs which has been decorated with members of $P$, $Q$ and $R$. Note that we can also decorate a graph using two sets of graphs e. g. those coming from $P$ and $R$ using different attachment vertices. We illustrate this in Fig. \ref{fig:bdecoration}f-g). That is, among one set of vertices we decorate with members of $P$ in different permutations as described above, and among a different set of vertices we decorate with members of $Q$ or $R$. 
\begin{figure}
\centering
\includegraphics[width=0.8\textwidth]{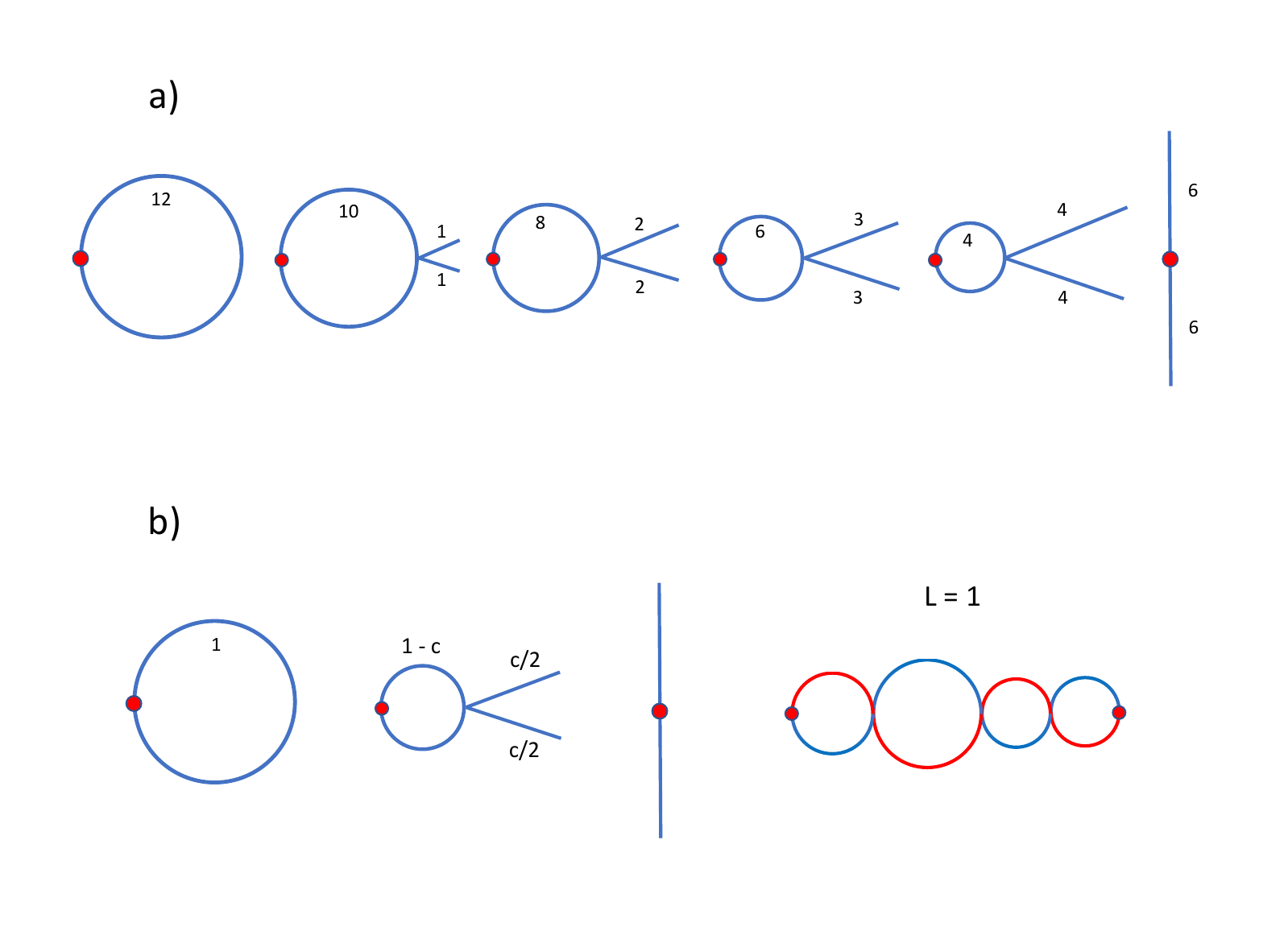}
\caption{a) Members of the generating set P(12), with their attachment vertices. b) Members of the generating set P(1) are of the shown parametrised form, and include the loop, the interval, tadpole graphs and one example of a chain of loops. The length is one but a different generating set can be obtained by changing the length. Red vertices have the same $M$-function. The chain of loops have two paths (branches) connecting the attachment vertices, shown in red and blue. Any chain of loops have the same $M$-function at the end vertices, as long as the total length is kept constant.}
\label{fig:btadpools}
\end{figure}

\begin{figure}
\centering
\includegraphics[width=1.0\textwidth]{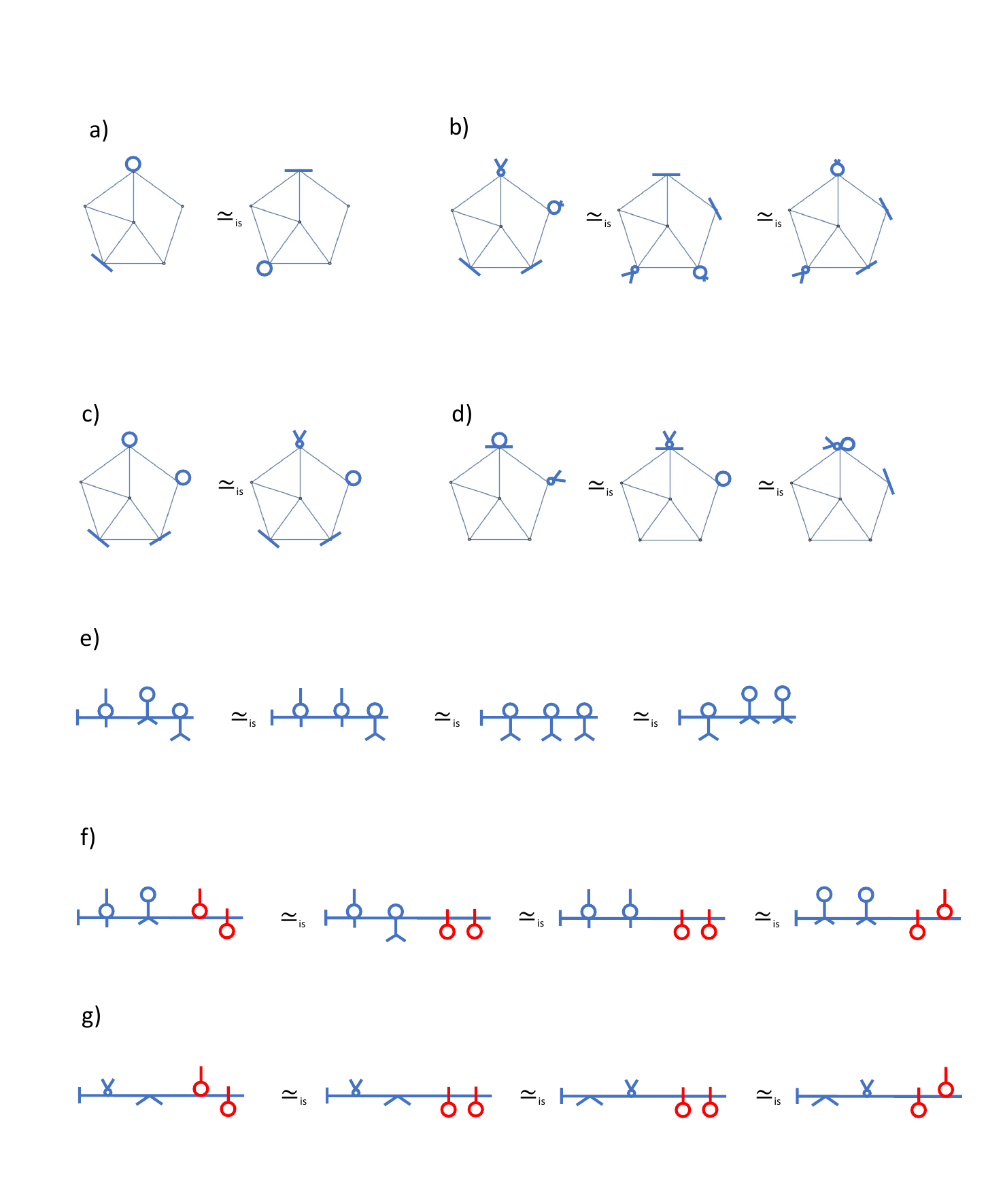}
\caption{ Attaching members of the generating set $P(1)$, shown in Fig. \ref{fig:btadpools}, to a graph in different permutations will generate isospectral graphs a) - d). There is freedom to exchange two attached graphs that are isospectral as shown in c). It is possible to attach two members of $P(1)$ to the same vertex and different permutations will still produce isospectral graphs as illustrated in d). In e) we show some possibilities of attaching members of the generating set $Q$, shown in Fig. \ref{fig:bhotvertexdemo}, to a simple graph, generating isospectral graphs. In f) we have attached members of $Q$ (blue) and $R$ (red) among different vertices in different combinations generating isospectral graphs. In g) we have done the same using members of $P(1)$ (blue) and $R$ (red). More examples could easily be given, also using all three generating sets, but the general idea should be clear.}
\label{fig:bdecoration}
\end{figure}

In Fig. \ref{fig:bmirror} we show that the graph in Fig. \ref{fig:bhotvertexdemo}c) is in fact an interval, $\Gamma_1$, which has been decorated with $\Gamma_L$ and $\Gamma_I$. We can decorate any graph containing $\Gamma_1$ as a subgraph, in two different ways by $\Gamma_L$ and $\Gamma_I$, illustrated in Fig. \ref{fig:bmirror}, and get isospectral graphs. This explains why the graph in Fig. \ref{fig:bhotvertexdemo}c) has two hot vertices. If the interval, $\Gamma_1$, has length 1/2, then the graph in Fig. \ref{fig:bmirror} has the isospectral partner shown in Fig. \ref{fig:bhotvertexdemo}b). We have not yet found an isospectral partner if the length of $\Gamma_1$ is not 1/2 except if it is zero and then the loop is an isospectral partner.

\begin{figure}
\centering
\includegraphics[width=1.0\textwidth]{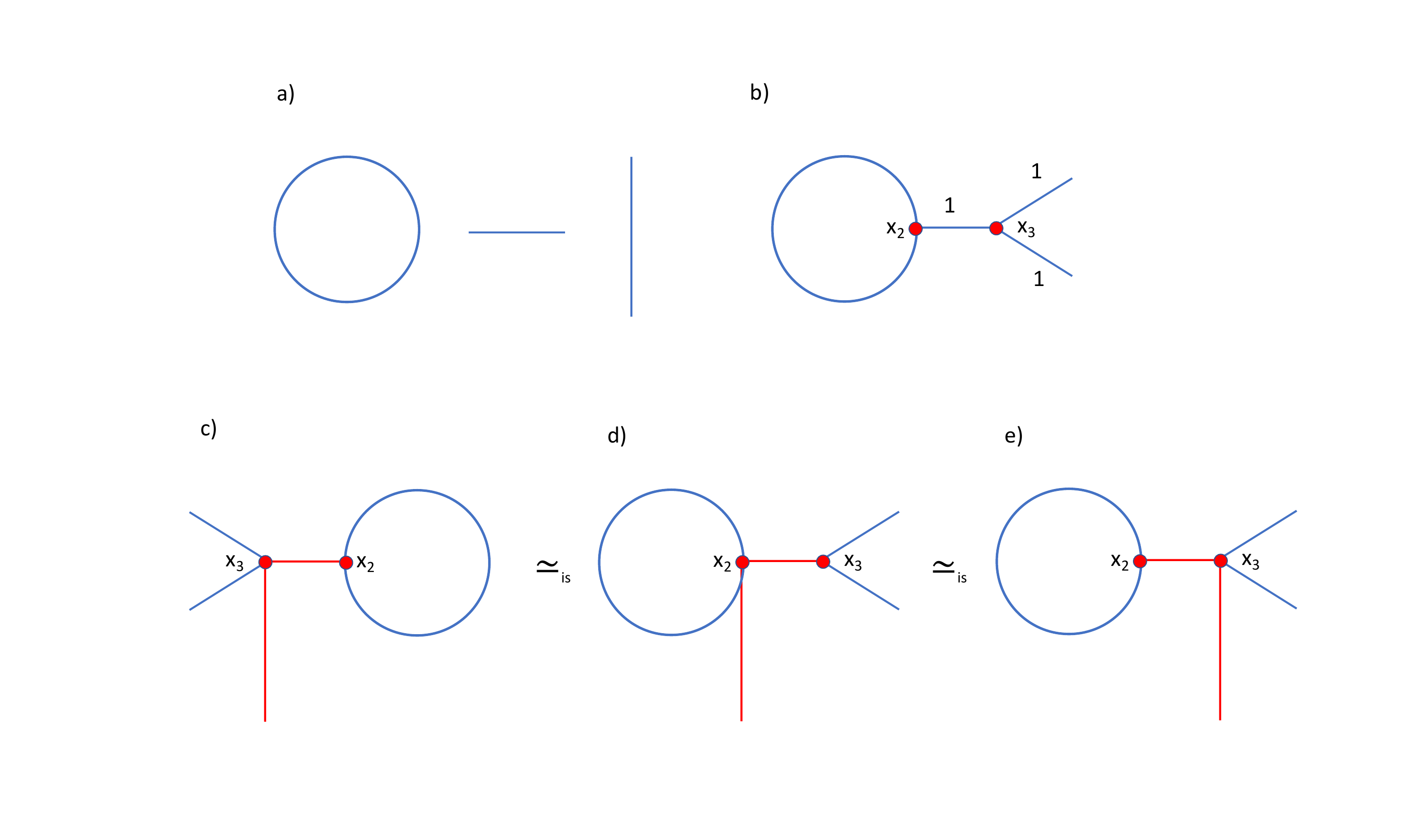}
\caption{a) - b) Attaching $\Gamma_{L}$ and $\Gamma_{I}$ $\in$ $P(1)$ to an interval will generate the graph in Fig. \ref{fig:bhotvertexdemo}b) for a suitable length of the decorated interval. c) - e) Attaching 
$\Gamma_{L}$ and $\Gamma_{I}$ in two ways to the graph in red and reflecting one of the graphs shows that vertices $x_2$ and $x_3$ have the same $M$-function.}
\label{fig:bmirror}
\end{figure}

\begin{figure}
\centering
\includegraphics[width=0.85\textwidth]{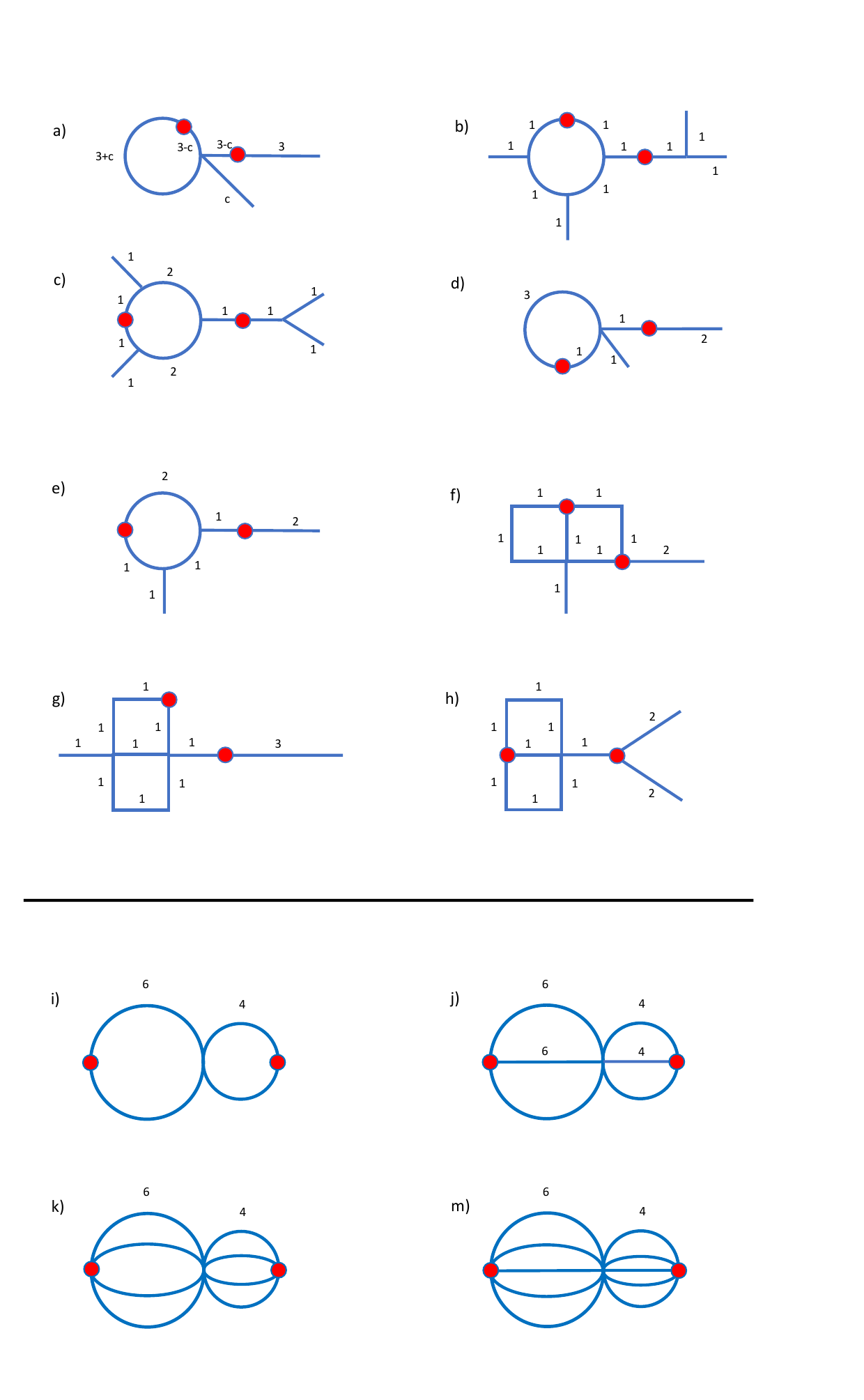}
\caption{ a) A graph having two hot vertices, which depend on a parameter c. b) - h) Different graphs having two hot vertices. The loop in c) has length 6. The loops in b) and d) - e) have length 4. Some of the graphs have three hot vertices where two are related by symmetry. We only display hot vertices not related by obvious symmetries. i-m) The first four members of an infinite family of graphs, each having two hot vertices.}
\label{fig:btwohotvertices}
\end{figure}

\section{Generating sets with one member and two hot vertices}
We have found several graphs which have two hot vertices, and thus form generating sets. These graphs are shown in Fig. \ref{fig:btwohotvertices}. The graph in Fig. \ref{fig:btwohotvertices}a) depends on a parameter c, and thus form an uncountable family of graphs with two hot vertices. These graphs are related to the tadpole graphs shown in Fig. \ref{fig:btadpools}. 

We can attach any compact graph, $\Gamma$, to the two hot vertices in two different ways as illustrated in Fig. \ref{fig:ctwohotmirror}. That is, we use two vertices of $\Gamma$ and attach them in the two ways possible to the hot vertices. The resulting graphs are then isospectral. We can also attach two different graphs to the two hot vertices in two different ways as illustrated in Fig. \ref{fig:ctwohotmirror}. The resulting graphs are isospectral. $\Gamma$, $\Gamma_1$ and $\Gamma_2$ can have any boundary conditions, as long as they make $\mathbf{L}$ selfadjoint.

What we do find interesting is that we can make an infinite set of generating sets having two hot vertices (checked by computer), as shown in Fig. \ref{fig:btwohotvertices}i,j,k,m). These graphs consists of two connected "pumpkin" graphs. We will call the number of edges in a pumpkin graph the degree of the pumpkin graph, and the degree of a loop is two. Since we can not only change the degree of the pumpkin graphs but also the length of the edges, we have a very flexible set of graphs having two hot vertices. 

\begin{figure}

\centering
\includegraphics[width=1.0\textwidth]{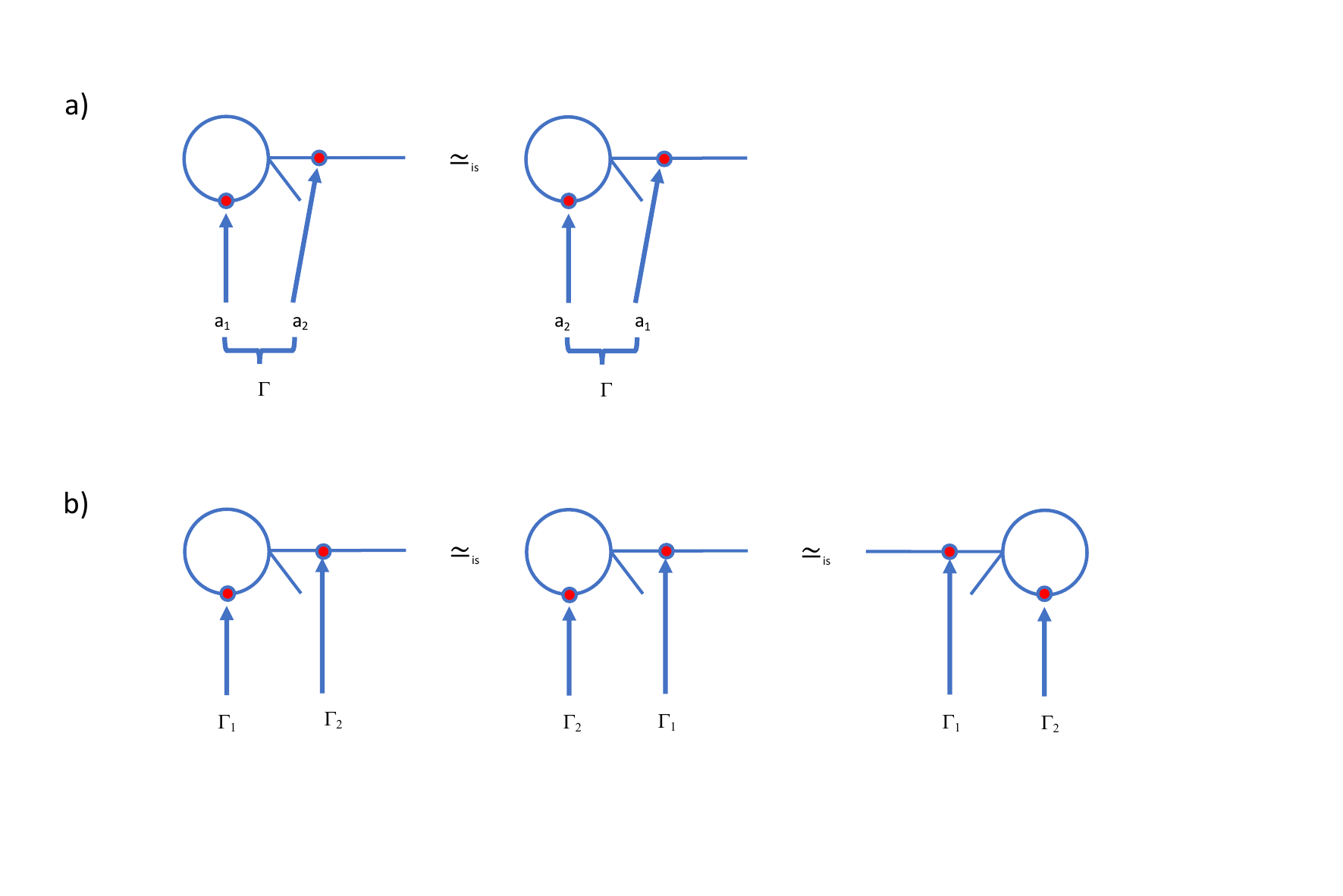}
\caption{a) We here show the graph in Fig. \ref{fig:btwohotvertices}d) where we have attached a graph, $\Gamma$, to two hot vertices using two vertices $a_1$ and $a_2$ of $\Gamma$. This can be done by attaching the vertices $a_1$ and $a_2$ to the hot vertices in two ways. The resulting graphs are isospectral. We can also attach two different graphs in two different ways as illustrated in b) which can be seen as a flipping of the top graph. The resulting graphs are isospectral. Note that $\Gamma$, $\Gamma_1$ and $\Gamma_2$ can have any boundary conditions, as long as they make $\mathbf{L}$ selfadjoint.}
\label{fig:ctwohotmirror}
\end{figure}

\section{Superlattices and aperiodic graphs that have the same spectra}

Using graphs with two hot vertices it is possible to create long chains of isospectral graphs. Especially interesting are chains of pumpkin graphs, which has been noted before, although in a different context \cite{berkolaiko2019surgery}. A pumpkin chain, where the pumpkin graphs have the same degree (but not necessarily the same length) does not change its spectrum if we change the order of the pumpkin graphs. We illustrate this for chains of loops in Fig. \ref{fig:cchainofloops}. It can be seen that very simple isospectral graphs, consisting of only three loops can be constructed. Such chains also have two hot vertices at their ends which indeed do have the same $M$-functions, proved below. It is also possible to have pumpkin chains where the pumpkin graphs have different degrees. In such a chain one is free to change the order of the graphs in any subchain consisting of pumpkin graphs having the same degree. This we illustrate in Fig. \ref{fig:cchainloopspumpkins}. We have the following theorem.

\begin{figure}
\centering
\includegraphics[width=0.8\textwidth]{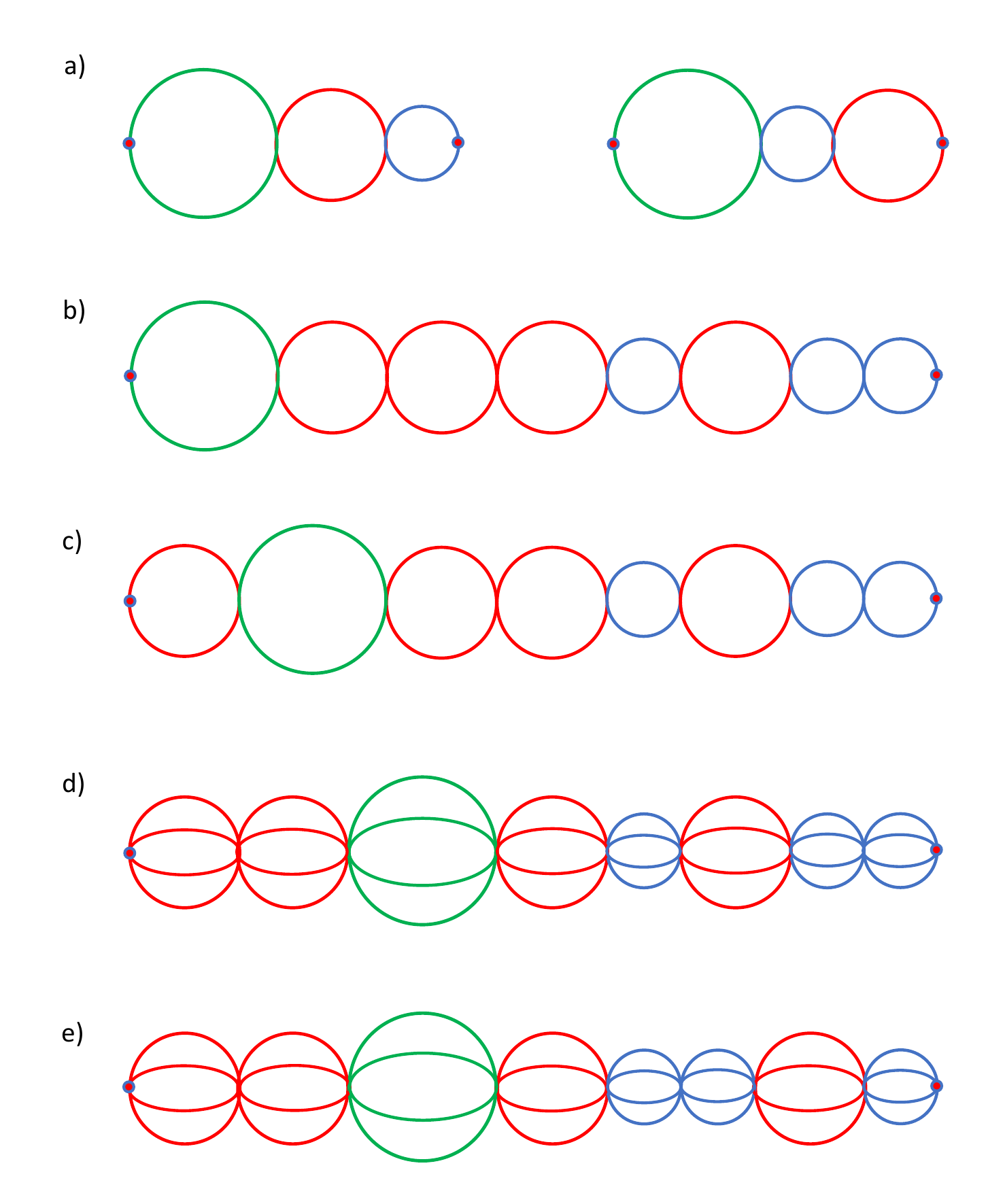}
\caption{ a) Two simple isospectral graphs. b-e) chains of pumpkin graphs. b) is isospectral to c) and d) is isospectral to e). All these graphs have hot vertices which are shown.}
\label{fig:cchainofloops}
\end{figure}

\begin{figure}
\centering
\includegraphics[width=0.8\textwidth]{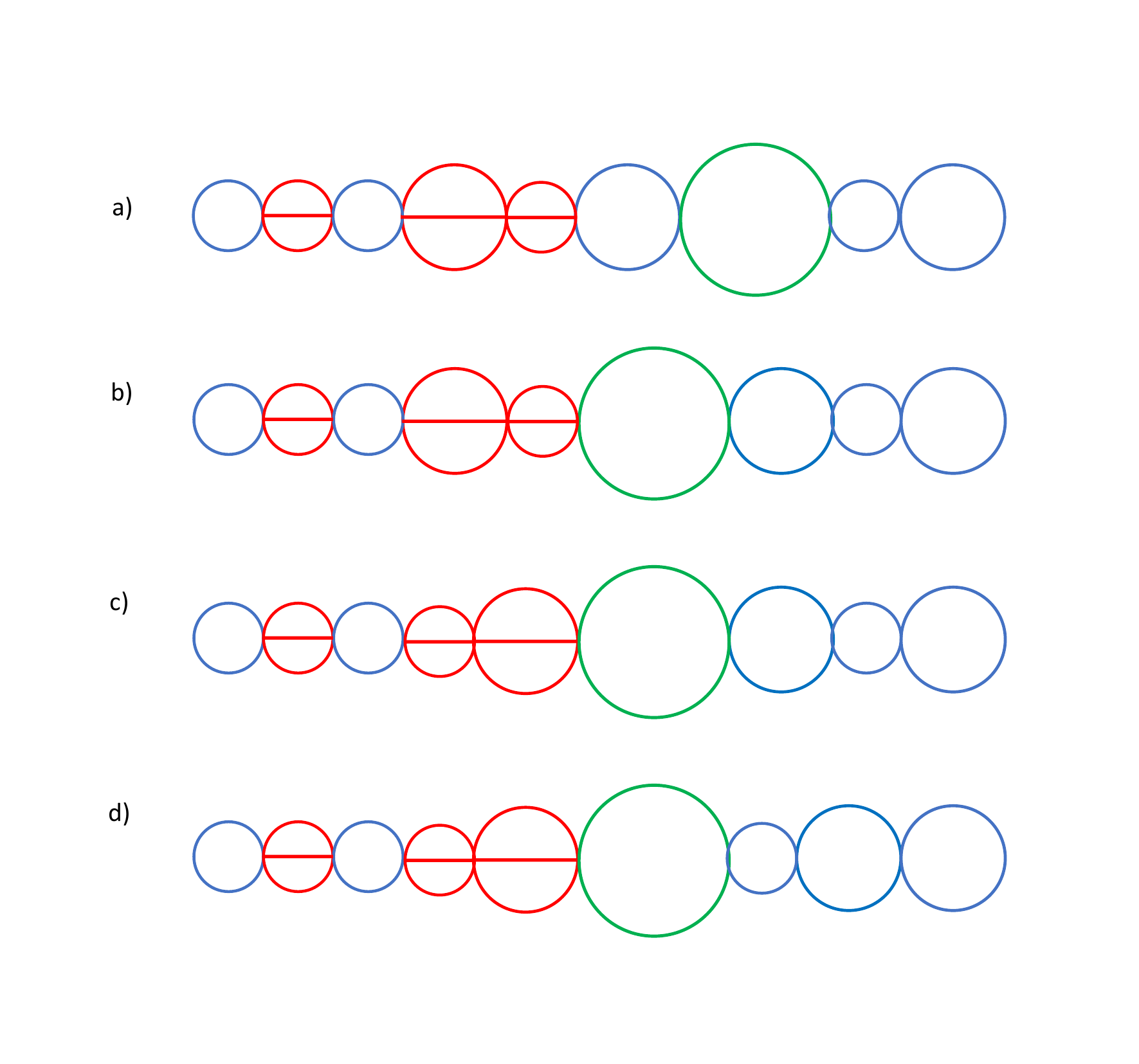}
\caption{ Four isospectral graphs consisting of chains of pumpkin graphs with degree two and three. Any set of adjacent pumpkin graphs with the same degree can have any order. These graphs do not seem to have any hot vertices.}\label{fig:cchainloopspumpkins}
\end{figure}

\begin{figure}
\centering
\includegraphics[width=0.8\textwidth]{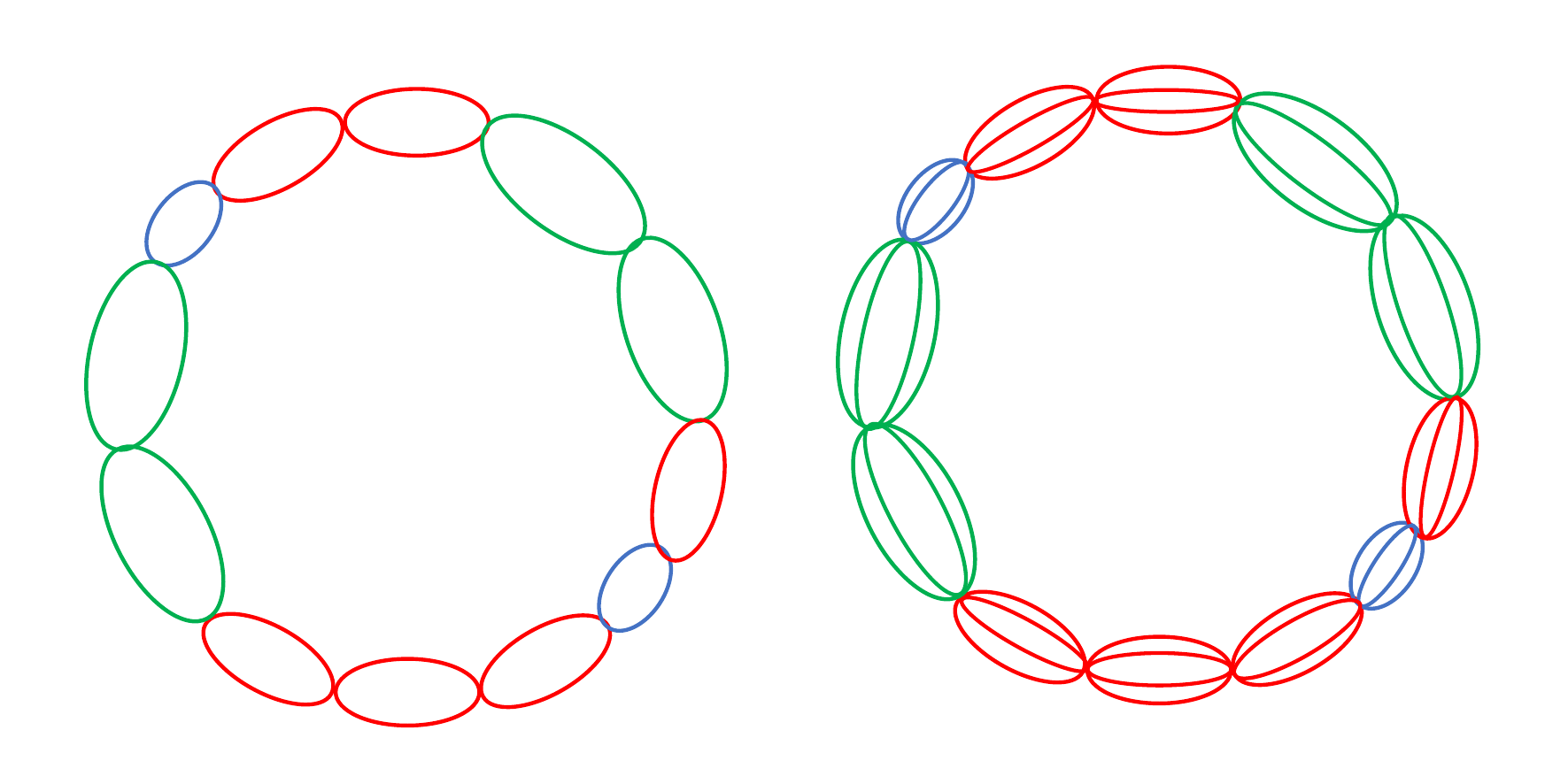}
\caption{ a) Two rings constructed from loops and pumpkin graphs. The order of the loops or the pumpkin graphs does not influence the spectrum.}
\label{fig:cringofloops}
\end{figure}

\begin{theorem}
The spectrum of a chain of pumpkin graphs, possibly with different lengths, but with the same degree, does not depend on the order of the pumpkin graphs, not even if the ends of the chains are attached to two compact graphs or two vertices of a compact graph. Each end of such a chain is a hot vertex. We illustrate such chains of pumpkin graphs in Fig. \ref{fig:cchainofloops}.
\end{theorem}

\begin{proof} 
The proof is very simple. We can reach any permutation of a chain of pumpkin graphs by performing a set of transpositions of two adjacent pumpkin graphs \cite{Rotman,Zeilinger}. Two connected pumpkin graphs have two hot vertices each having the same $M$-function. Thus transposing two pumpkin graphs (even if the hot vertices are attached to compact graphs) does not change the spectrum, see Fig. \ref{fig:ctwohotmirror}b) and we are done.
\end{proof}

This theorem is simple and thus maybe of not much interest. Nevertheless we find it intriguing that a periodic structure can be transformed into a periodic structure with a different periodicity or even into an aperiodic structure without changing the spectrum. An infinite graph with a periodic structure obeys Floquets theorem and band-structure emerges \cite{berkolaiko2013introduction}. An infinite graph which has an aperiodic structure is expected to behave differently (at least with respect to the eigenfunctions) and e. g. Anderson localisation can occur \cite{Anderson}. However, any chain of e. g. loops have permutation invariant spectra also if their length is very long. We have directly checked by computer that the order of a set of loops does not change the spectrum.

If we have a set of loops in a chain with lengths ($L_1,L_2,L_3,...,L_m$) where the number of loops of length $L_i$ is $n_i$ and the total length of the chain is $L$ then we have the following secular equation:
\begin{equation}
\Sigma (k)=\left(-1+e^{i k L}\right)\left(-1+e^{L_1 i k}\right)^{n_1} \left(-1+e^{L_2 i k}\right)^{n_2}  ... \left(-1+e^{L_m i k}\right)^{n_m} 
\end{equation}

with eigenfrequencies: 

\begin{itemize}
\item 0 with multiplicity one.
\item $2 \pi N/L$ with multiplicity one. 
\item $2 \pi N/L_1$ with multiplicity $n_1$.
\item $2 \pi N/L_2$ with multiplicity $n_2$.\\
...
\item $2 \pi N/L_m$  with multiplicity $n_m$. 
\end{itemize}

where $N$ is a positive integer. %$0$ with multiplicity one, $2 \pi N/L$ with multiplicity one, and a set of eigenfrequencies $2 \pi N/L_i$ with multiplicity $n_i$. $N$ is a positive integer. 

We see that there are eigenfrequencies that do not depend on the length of the graph, $L$ but instead on the length of the individual loops $L_n$. There is also a set of eigenfrequencies, $2 \pi N/L$, which become denser and denser over the whole non-negative real axis when $L$ increases. There are no spectral gaps, in the sense that any interval will contain eigenfrequencies if $L$ is larger than some constant. In 1D it is expected that spectral gaps should be prevalent for periodic graphs \cite{berkolaiko2013introduction} and Avron et al. have shown that periodic chains of pumpkin graphs can have large spectral gaps \cite{Avron1994}.

 It is not difficult to prove that the above secular equation is correct which will be done in the following section.

We can get some information about the $M$-function for a chain of loops.

\begin{theorem}
The $M$-function at the end vertices of a chain of loops, see Fig. \ref{fig:btadpools}, depends only on the length of the chain, $L$, and is the same as the $M$-function for a loop of length $L$.
\end{theorem}

\begin{proof}
Let's call $k$-dependent functions which satisfy the $\bold{L}$-differential - \textit{k-functions}. Such $k$-functions can be chosen to be either symmetric or anti-symmetric under the exchange of branches, where branches are defined in Fig. \ref{fig:btadpools}. One way to see this is that they obey either Dirichlet or Neumann boundary conditions at the interior end vertex. If they are anti-symmetric they will not be continuous at all vertices unless $k$ is an eigenfrequency in which case they will have value zero at the boundary vertex. Such functions do not contribute to the $M$-function. Symmetric $k$-functions can be written $\textnormal{Cos}(k x)$ for some $k$. $x=0$ at one end vertex, follows one branch and returns on other branch. This is precisely the same $k$-functions that determine the $M$-function for a loop and we are done. 
\end{proof} 
We have checked that this theorem is true by computer, for extra safety.

A set of loops in a ring, as shown in Fig. \ref{fig:cringofloops} have the secular equation: 
\begin{equation}
\Sigma (k)=\left(-1+e^{i k L/2}\right)^2\left(-1+e^{L_1 i k}\right)^{n_1} \left(-1+e^{L_2 i k}\right)^{n_2} ... \left(-1+e^{L_m i k}\right)^{n_m} 
\end{equation}
and also here the secular equation tells us that the order of the loops does not matter.

\section{Generating graphs having many vertices with the same $M$-function}
It can be seen that replacing all the edges in a graph with chains of loops corresponds to taking two copies of a graph, and then attaching them to each other at every vertex. We introduce vertices of valence two on each edge to define the loops and we illustrate this in Fig. \ref{fig:cloopconstruction}. 
Let's call the initial graphs $\Gamma_1$ and $\Gamma_2$ and the resulting graph $\Gamma_1 \sqcup \Gamma_2$. We want to find the secular equation for $\Gamma_1 \sqcup \Gamma_2$ and we do this by finding all eigenfunctions. 

Any edge originates from $\Gamma_1$ or $\Gamma_2$. Let's call the set of edges that originates from $\Gamma_1$ or $\Gamma_2$ a \textit{branch} and we have two branches. Any eigenfunction of $\Gamma_1 \sqcup \Gamma_2$ can be chosen to be either symmetric or anti-symmetric under exchange of the two branches. If the eigenfunction is symmetric under exchange of the branches it can be written $f_i+g_i$ where $f_i$ and $g_i$ have support on one branch each. They will be eigenfunctions of $\Gamma_1$ and $\Gamma_2$, respectively. If the eigenfunction is anti-symmetric under exchange of branches then it can be written as a linear combination of eigenfunctions of the loops. Each loop has a set of anti-symmetric eigenfunctions with support on the loop alone. We have proven: 
\begin{theorem}
The secular equation for $\Gamma_1 \sqcup \Gamma_2=\Gamma \sqcup \Gamma$ is:
\begin{equation}
\Sigma^{\Gamma \sqcup \Gamma} (k)=\Sigma^\Gamma(k)\left(-1+e^{L_1 i k}\right)^{n_1} \left(-1+e^{L_2 i k}\right)^{n_2}  ... \left(-1+e^{L_m i k}\right)^{n_m} 
\end{equation}
where $\Sigma^\Gamma(k)$ is the secular equation of graph $\Gamma$ and $\Gamma \sqcup \Gamma$ contains $m$ sets of loops with length $L_i$, each repeated $n_i$ times. \end{theorem}
We can see that the secular equation for a chain of loops and a ring with loops given above are of the correct form.
We can iterate this process and get pumpkin graphs instead of loops in the construction.

By replacing edges with loops we have a method to generate non-isomorphic graphs which have many hot vertices having the same $M$-function. We can also generate non-isomorphic graphs having different classes of hot vertices.
We start with a graph having a set of vertices with the same $M$-function, due to symmetry. This is illustrated in Fig. \ref{fig:cloopconstruction}. We then replace edges with chains of loops and break the symmetry. The hot vertices will still have the same $M$-functions, albeit likely different from the initial $M$-function. We then also generate several non-isomorphic graphs each having a set of hot vertices with the same $M$-function.

Having such a set of graphs we can make infinite graphs in several dimensions, where the unit cell is repeated, not isomorphically but isospectrally. This we illustrate in the two-dimensional case in Fig. \ref{fig:csuperlattice2D}. It is clear that many interesting infinite graphs can be made this way.

\section{Star graphs with pumpkin graphs as leaves}
Since pumpkin graphs seem to be important we decided to investigate pumpkin graphs for isospectrality. We found that two connected pumpkin graphs are isospectral, as long as the edges have the same length. The allocation of edges between the two constituent pumpkin graphs does not matter as long as they have at least one edge. Such a connected pumpkin graph is shown in Fig. \ref{fig:cdumbbells}a). The central vertex is always a hot vertex. Fig. \ref{fig:cdumbbells}b) shows a double loop (a special case of the graph in a)) which has the isospectral partner shown. These graphs have a hot vertex each and if we attach an interval as shown to these hot vertices we generate two isospectral graphs which are part of an isospectral triplet c). In Fig. \ref{fig:cdumbbells}d) we show three isospectral graphs which are members of the isospectral set of four, shown in Fig. \ref{fig:afirsttriplets}. Using these connected pumpkin graphs it is easy to construct sets containing many isospectral graphs having hot vertices. The secular equation of the connected pumpkin graphs is:

\begin{equation}
\Sigma (k)=\left(e^{2 i k}-1\right)^{K-1} \left(e^{2 i k}+1\right),
\end{equation}

where $K$ is the number of edges and the length of the graph is $K$. The eigenfrequencies are 0 with multiplicity one, $N \pi$ with multiplicity $K-1$ and $N \pi + \pi/2$ with multiplicity one. $N$ is a positive integer.

\begin{figure}
\centering
\includegraphics[width=1.0\textwidth]{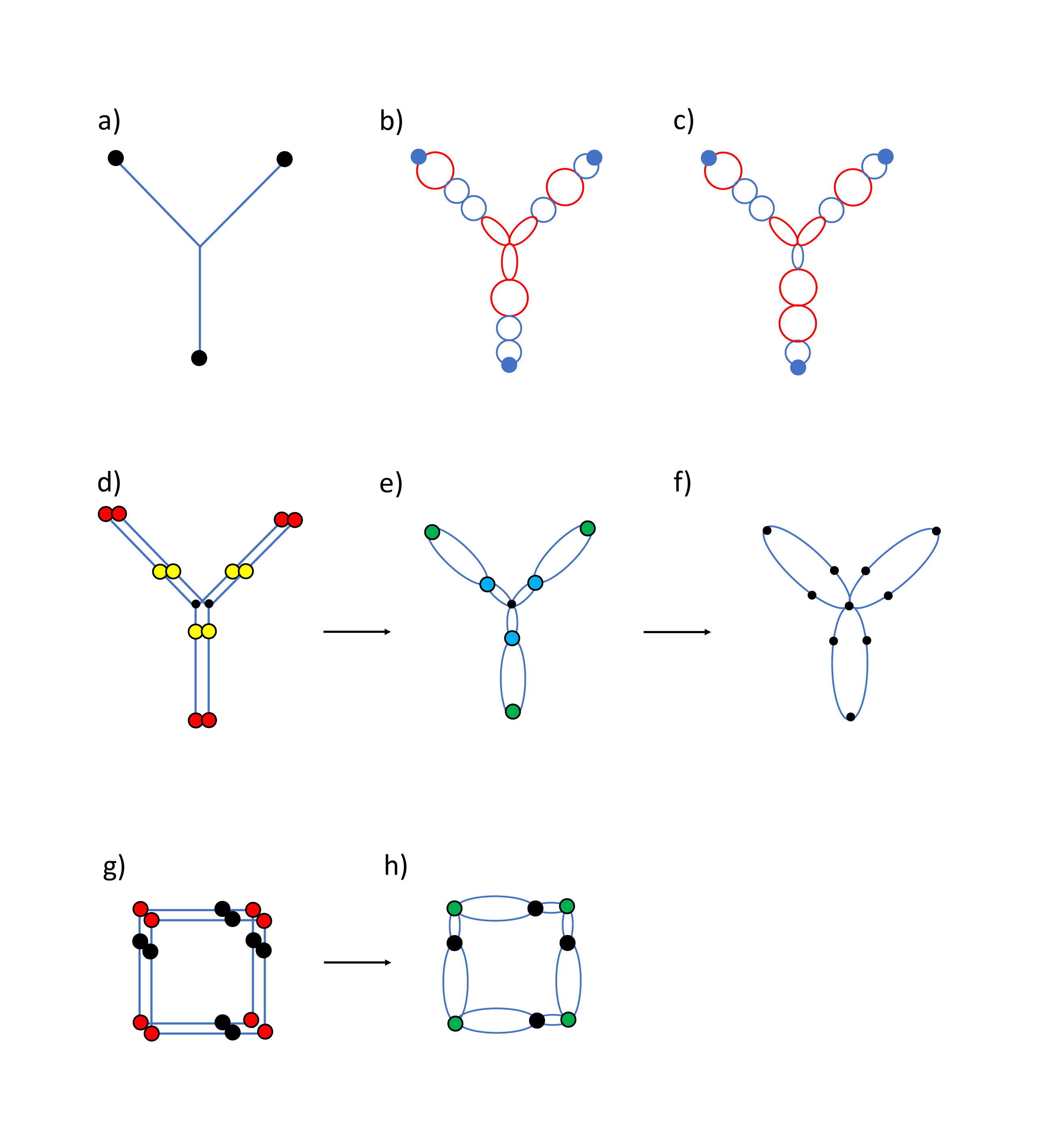}
\caption{ a) A star graph having three hot vertices. b-c) Two star graphs where the leaves consist of loops having three hot vertices each. These graphs are isospectral. d) An illustration that attaching two identical graphs to each other will create a graph where the edges are replaced with loops. The initial graphs d) have two sets of hot vertices (red and yellow), related by symmetry. e) After attaching the graphs to each other we obtain new graphs having two sets of hot vertices, blue and green. f) unrolling the loops makes it easier to visualise eigenfunctions. It is still necessary to keep the boundary conditions right. g-h) Attaching a square to itself generates a graph with four hot vertices, shown in green CHECK. Vertices with the same colour have the same $M$-function.}\label{fig:cloopconstruction}
\end{figure}

\begin{figure}
\centering
\includegraphics[width=1.0\textwidth]{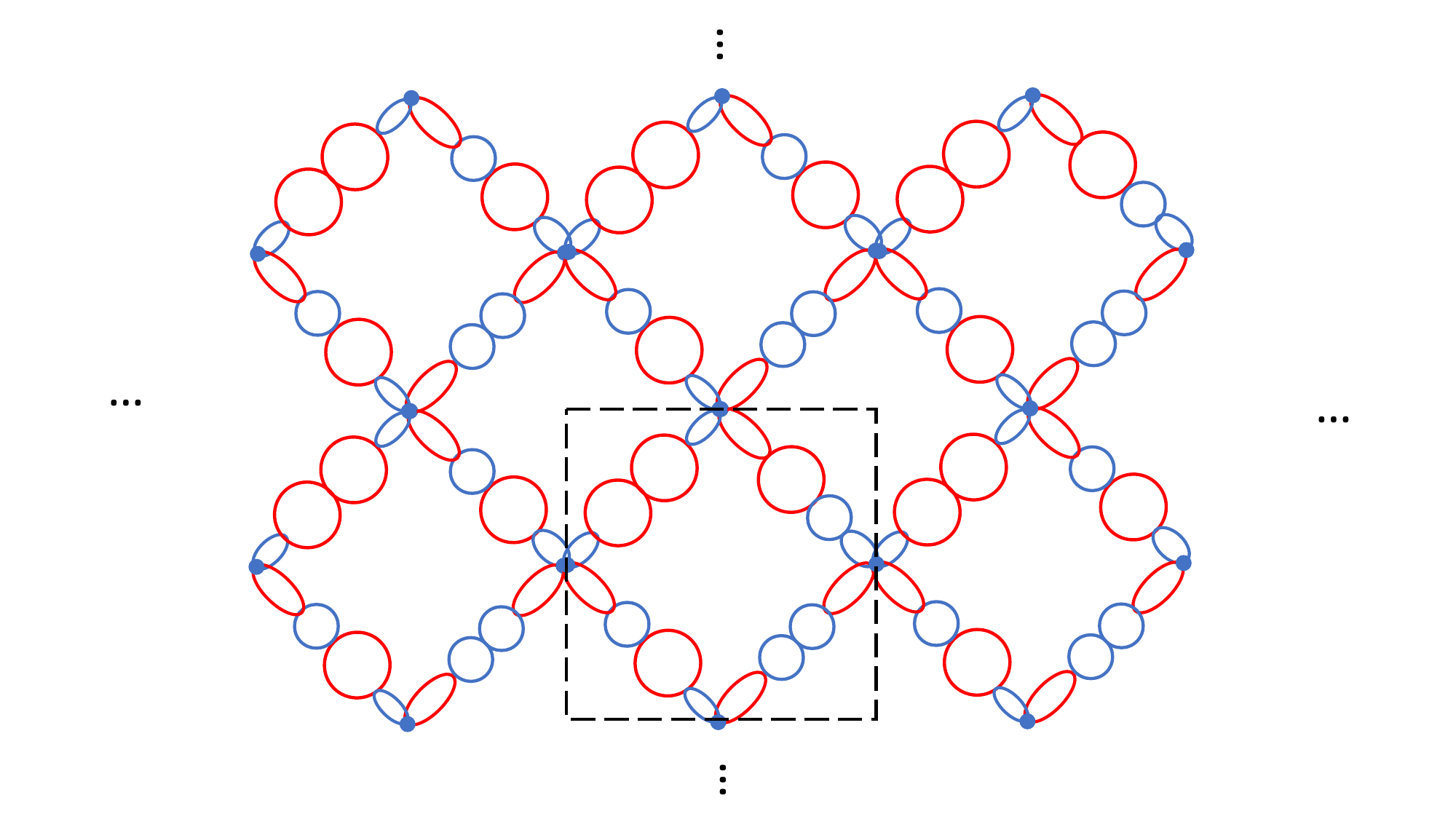}
\caption{ Part of an infinite graph. A unit cell, indicated by the black dashed box, is repeated with isospectral copies instead of isomorphic copies, forming a square lattice. All red loops have the same length and all blue loops have the same length.}
\label{fig:csuperlattice2D}
\end{figure}

\begin{figure}
\centering
\includegraphics[width=1.2\textwidth]{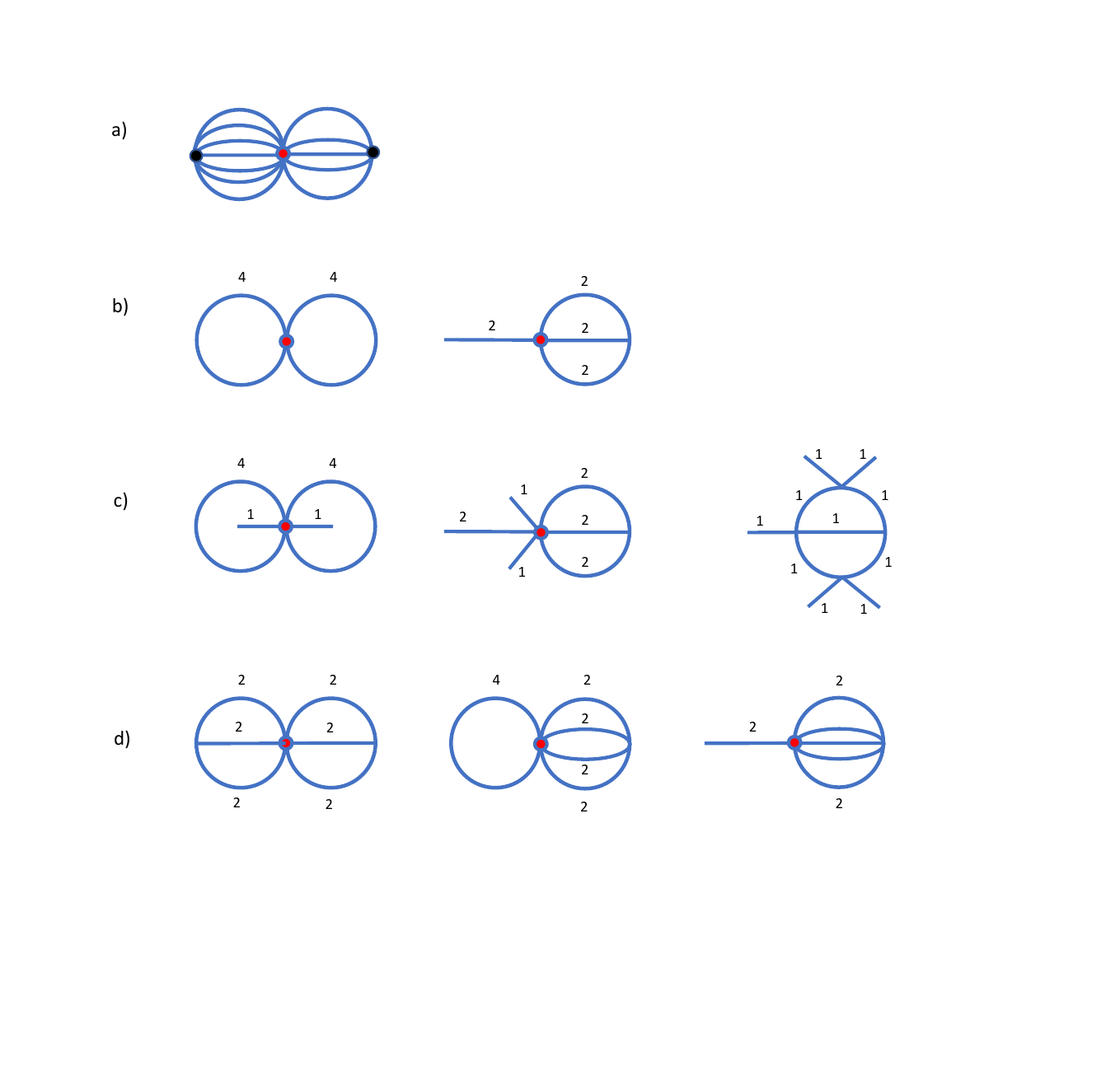}
\caption{ a) Two connected pumpkin graphs with a total of 12 edges. All such graphs with the same number of edges are isospectral. The central, red, vertex is a hot vertex. The black vertices are not hot. b) Two connected loops have an isospectral partner where one subgraph consists of an interval. We consider a loop to have two edges for consistency. c) Attaching an interval to the central vertices in b) creates two isospectral graphs which are part of the shown isospectral triplet. d) Three members of the isospectral set of four shown in Fig. \ref{fig:afirsttriplets}. They can all be seen to be two connected pumpkin graphs with six edges in total. Red vertices within each isospectral set have the same $M$-function.}
\label{fig:cdumbbells}
\end{figure}

\begin{figure}
\centering
\includegraphics[width=1.0\textwidth]{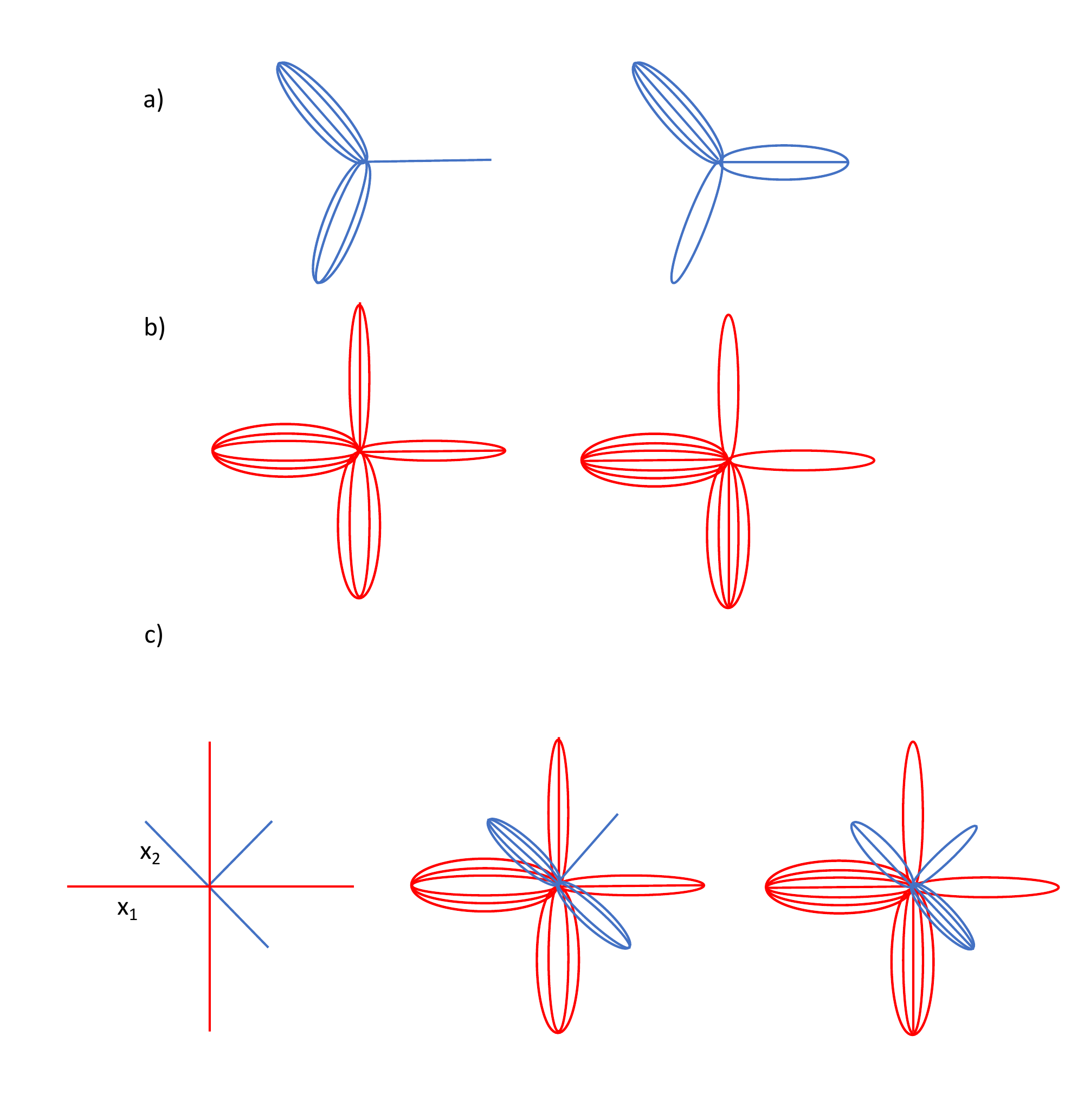}
\caption{ a) Two isospectral graphs consisting of a star graph where the leaves are pumpkin graphs. These graphs have three leaves and ten edges. b) A different example of two isospectral graphs having four leaves and 16 edges. c) A star graph having two set of leaves having different lengths, $x_1$ (red edges) and $x_2$ (blue edges). We can replace the leaves with pumpkin graphs where the edges have length $x_1$ and $x_2$. If the number of edges with length $x_1$ is, say $E_1$, and the number of edges with length $x_2$ is, say $E_2$, we will get isospectral graphs independently of their allocation between leaves. We show two such isospectral graphs in the figure.}
\label{fig:cpumpkinstar}
\end{figure}

\begin{figure}
\centering
\includegraphics[width=0.8\textwidth]{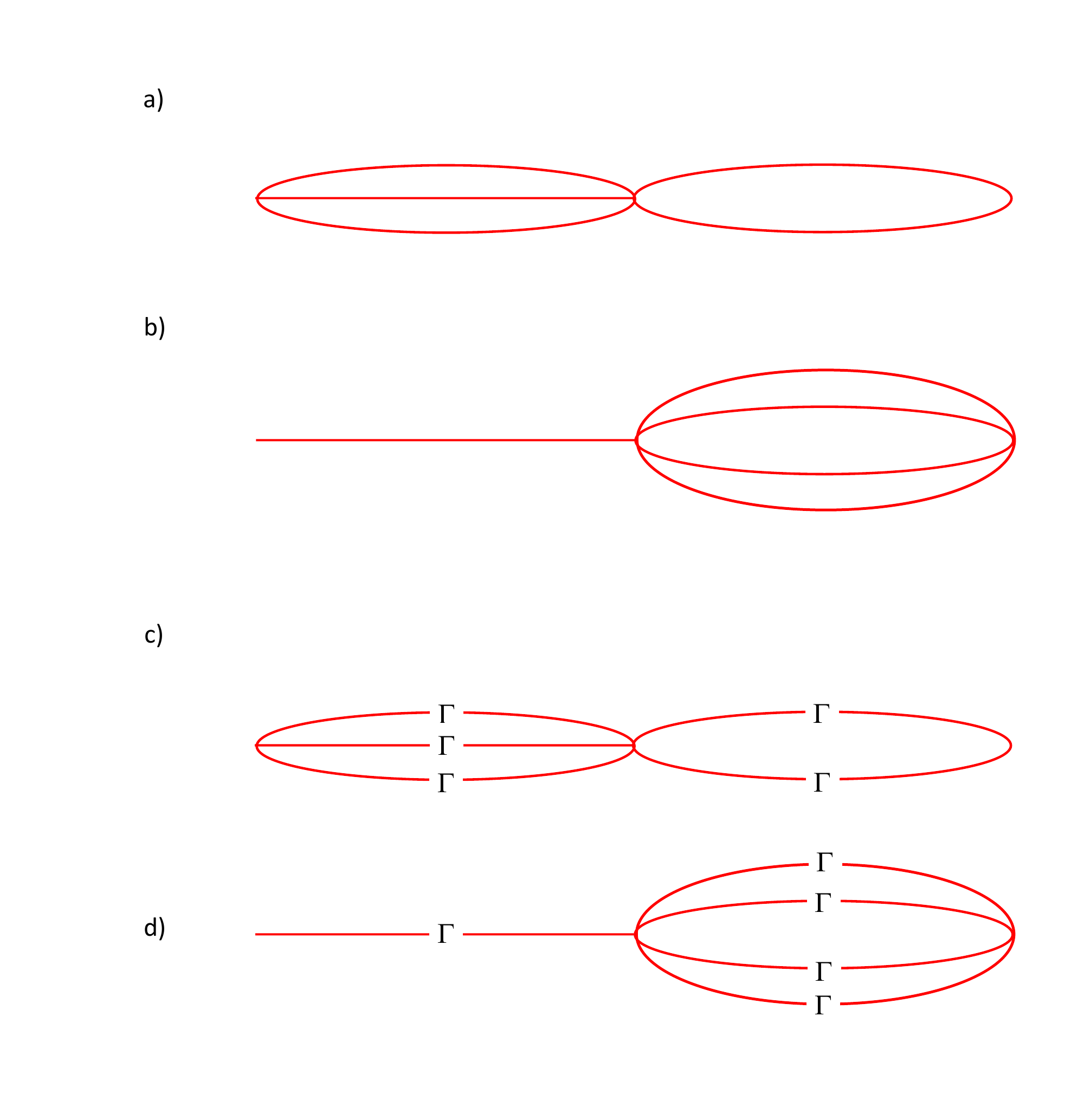}
\caption{ a-b) Two isospectral star graphs with pumpkin leaves having two leaves and a total of five edges. c-d) Replacing the edges of the graphs in a-b) with a compact graph $\Gamma$ will give these isospectral graphs.}
\label{fig:cpumpkinstargraph}
\end{figure}

We have also found that we can replace the leaves of a star graph with pumpkin graphs as shown in Fig. \ref{fig:cpumpkinstar}. Let us call such graphs - \textit{star graphs with pumpkin leaves}. If the total number of edges and leaves is constant then we have isospectral graphs provided the length of the edges is the same. In other words, we can freely allocate the edges between the pumpkin graphs while keeping the total number of edges constant. The central vertex is a hot vertex. We can also start with a star graph having two sets of leaves of different length, $x_1$ and $x_2$ and replace each leaf with a pumpkin graph having edges with the same length as the leaf. The resulting graphs will be isospectral, provided the number of edges with length $x_1$ and $x_2$ is kept invariant. Naturally we can start with star graphs having even more sets of leaves with different lengths and replace them with pumpkin graphs in the same way and generate isospectral graphs. The connected pumpkin graphs in Fig. \ref{fig:cdumbbells} are a special case of the star graphs with pumpkin leaves. 

These results are not surprising since Butler and Grout showed that star graphs with pumpkin leaves are isospectral if they have the same number of edges, with respect to the normalized Laplacian \cite{Butler2012}. It then follows that they are isospectral also as quantum graphs using the transference principle by von Below \cite{vonBelow1985}.

We also found that we can replace the edges with graphs and still get isospectral graphs as illustrated in Fig. \ref{fig:cpumpkinstargraph}. That is, we choose a compact graph $\Gamma$ and replace every edge in an isospectral pair with $\Gamma$ and the resulting graphs will be isospectral.

We find that the secular equation for star graphs with pumpkin leaves is:

\begin{equation}
\Sigma (k)=\left(e^{2 i k}-1\right)^{K-S+1} \left(e^{2 i k}+1\right)^{S-1} 
\end{equation}
where $S$ is the number of leaves, $K$ is the number of edges and the length of the graph is $K$.

The eigenfrequencies are, with $N$ a positive integer:
\begin{itemize}
\item 0 with multiplicity one.
\item $N \pi$ with multiplicity $K-S+1$. 
\item $N \pi + \pi/2$ with multiplicity $S-1$. 
\end{itemize}

The secular equation reduces to $\left(e^{2 i k}-1\right)^{K}$ for pumpkin graphs of length $K$, to $\left(e^{2 i k}-1\right) \left(e^{2 i k}+1\right)^{S-1}$ for star graphs of length $S$ and for flower graphs it becomes
$\left(e^{i k}-1\right)^{S+1} \left(e^{i k}+1\right)^{S-1}$ (where each leaf is a loop \cite{Blixt2015}) of length $2 S$. The secular equation for the star graphs with pumpkin leaves is related to the secular equation of a single edge.  
  
Also for the more complicated graphs shown in Fig. \ref{fig:cpumpkinstargraph}c-d) there is a relation between the spectrum of $\Gamma$ and the spectrum of the full graph which we have not elucidated yet. We strongly believe that if we attach each $\Gamma$ to the neighboring $\Gamma$ such that the symmetry of the pumpkins is still the cyclic group $C_n$, where $n$ is the number of $\Gamma$'s in the pumpkin, we will also get isospectral graphs. We have not tested this since the programming is quite involved.
  
\section{Generating sets with more than one type of hot vertices}
Although we can easily generate graphs having many hot vertices, it is still interesting to find more irregular graphs having different types of hot vertices. We have found a few. In Fig. \ref{fig:dmultihot} we show graphs having two different types of hot vertices. Fig. \ref{fig:dmultihot}a) shows a pair of isospectral graphs having two hot vertices each. Fig. \ref{fig:dmultihot}b) can be obtained from the graphs in Fig. \ref{fig:dmultihot}a) by attaching an interval to the two hot vertices. The resulting isospectral pair has two sets of hot vertices, indicated by red and green vertices. These sets of hot vertices are not equivalent, the green vertices have one $M$-function and the red vertices have another $M$-function. We can attach a compact graph to a red vertex on graph c) and to a red vertex on graph d) and we get isospectral graphs. However if we attach a compact graph to a red vertex on graph c) and to a green vertex on graph d) the resulting graphs will in general not be isospectral. In Fig. \ref{fig:dmultihot}e-f) we show another set of isospectral graphs having two inequivalent sets of hot vertices. Fig. \ref{fig:dmultihot}g) shows an isospectral triplet having three sets of inequivalent hot vertices.\\
Fig. \ref{fig:dmultihottriplets} shows yet another example of two isospectral triplets having two sets of inequivalent hot vertices. One triplet has eight vertices and one has ten vertices. These two triplets are clearly related to each other. We did not find any isospectral triplet having nine vertices which is similar to these two triplets.

We then decided to investigate graphs containing many loops and found some examples of isospectral sets with very many different classes of hot vertices. The set of four isospectral graphs in Fig. \ref{fig:dhotvertices}d) has 21 sets of vertices with the same $M$-function within each set. We have only plotted the two sets that have a vertex on every graph. The isospectral set of two graphs in Fig. \ref{fig:dhotvertices}c) has four sets of vertices each with their unique $M$-function. In Fig. \ref{fig:dhotvertices}a-b) we show different behaviours of hot vertices and one graph has none. These graphs are not isospectral.

\subsection{Generating graphs and isospectral sets having many hot vertices by attachments}
We have empirically found a way to generate graphs having many vertices with the same $M$-function by attachments. This we will exemplify using the graphs in Fig. \ref{fig:dhotvertices}d). If we attach all four graphs to each other at their green vertices, we will have a new large graph that has four hot vertices and we can generate twelve such graphs which are all isospectral to each other. These vertices originate from the red vertices in Fig. \ref{fig:dhotvertices}d). This behaviour will also hold for other graphs but we have very few examples of graphs with many hot vertices. It is possible to use the graphs constructed using loops described in Section 10, but we want graphs without any symmetry.

\begin{figure}
\centering
\includegraphics[width=1.0\textwidth]{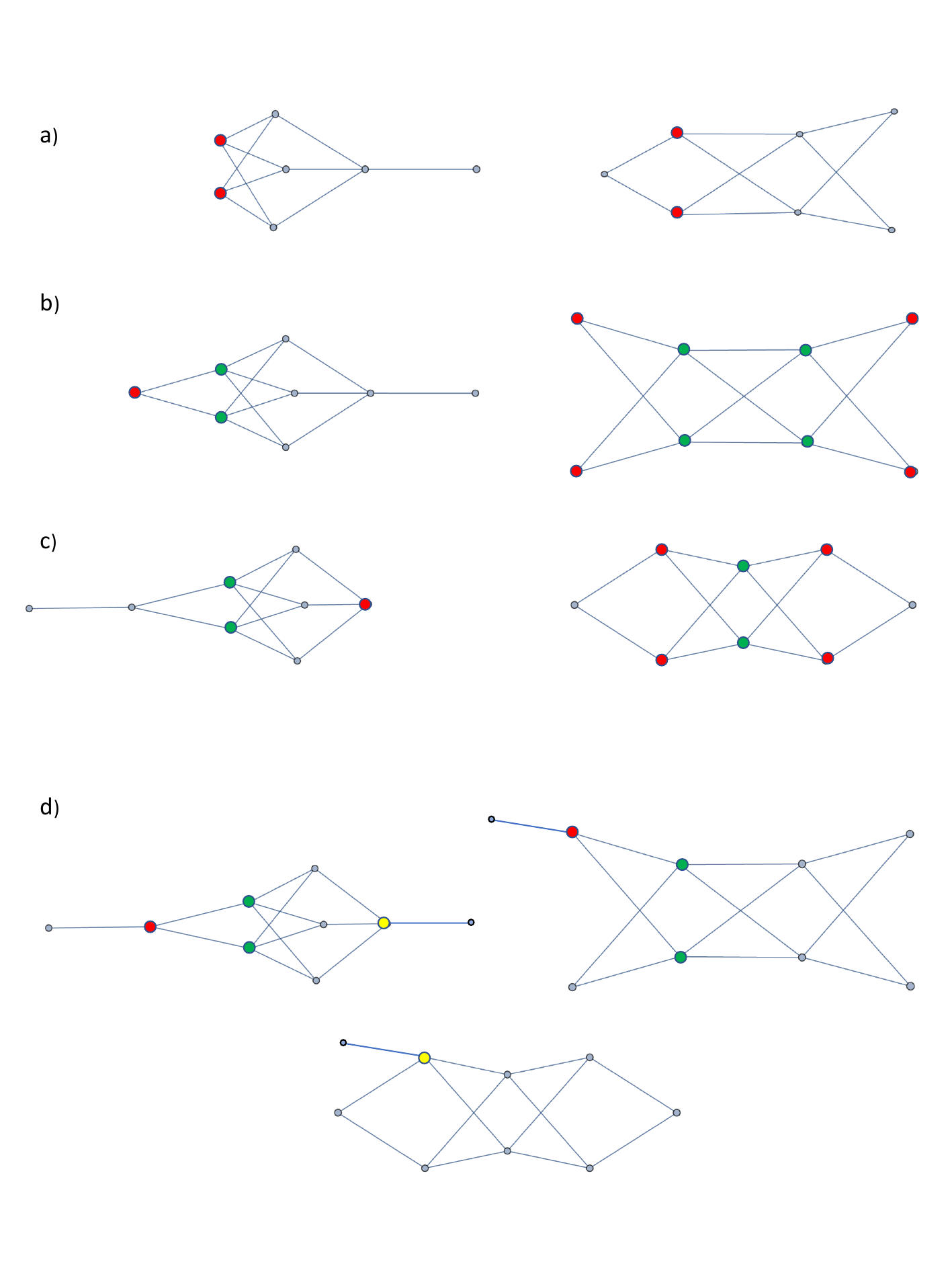}
\caption{ a-b) Two isospectral graphs having two hot vertices each, indicated in red. c-d) Two isospectral graphs having two sets of hot vertices each, indicated by red and green. e-f) Two isospectral graphs having two sets of hot vertices each, indicated by red and green. g) An isospectral triplet, having three sets of inequivalent hot vertices.}
\label{fig:dmultihot}
\end{figure}

\section{Interesting isospectral pairs}
We have found some isospectral pairs that we think are interesting, which we show in Fig. \ref{fig:dinterestingones}. This includes pairs involving asymmetric star graphs. We have not found a way to generalise the graphs shown in Fig. \ref{fig:dinterestingones}a-b). Fig. \ref{fig:dinterestingones}c) shows an isospectral pair without pendant edges. Fig. \ref{fig:dinterestingones}d) shows that a graph with only triangles can have an isospectral partner. Fig. \ref{fig:dinterestingones}e) shows a planar isospectral pair without pendant edges. Since Neumann boundary conditions are difficult to realize at terminal vertices in real physical systems we wanted to find planar graphs without pendant edges for future experiments. We also wanted both them to be planar since we suspect three-dimensionality will affect the spectrum. It is possible that flutes made out of these graphs would sound the same. Numerical studies unfortunately show that isospectral quantum graphs made of silver do not have the same vibrational spectrum even if they both can be embedded in the plane \cite{martin}.

\begin{figure}
\centering
\includegraphics[width=0.9\textwidth]{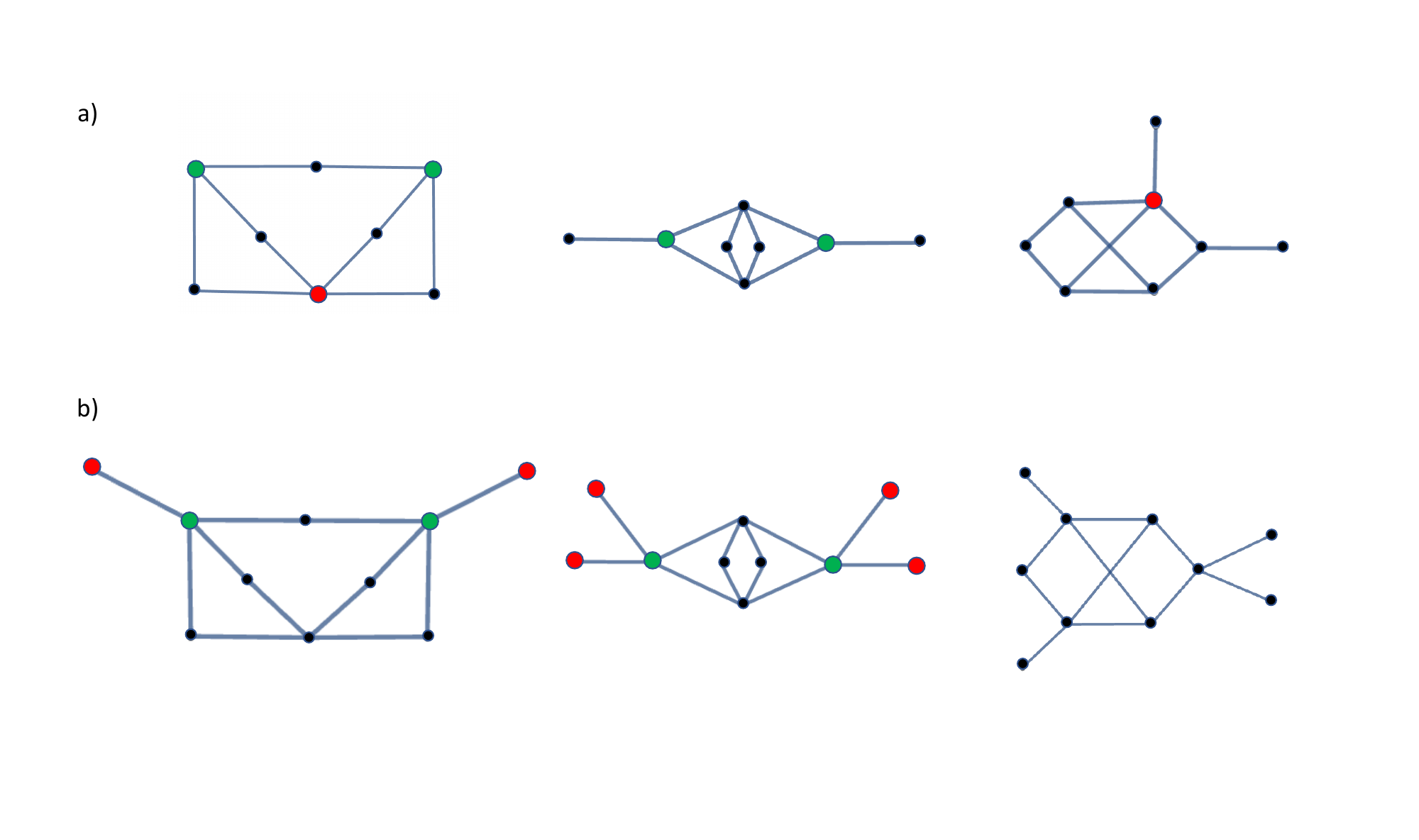}
\caption{ a) An isospectral triplet which has two sets of hot vertices, indicated by green and red vertices. b) Another isospectral triplet, closely related to the one in a). Also this triplet has two sets of hot vertices.}
\label{fig:dmultihottriplets}
\end{figure}

\begin{figure}
\centering
\includegraphics[width=0.9\textwidth]{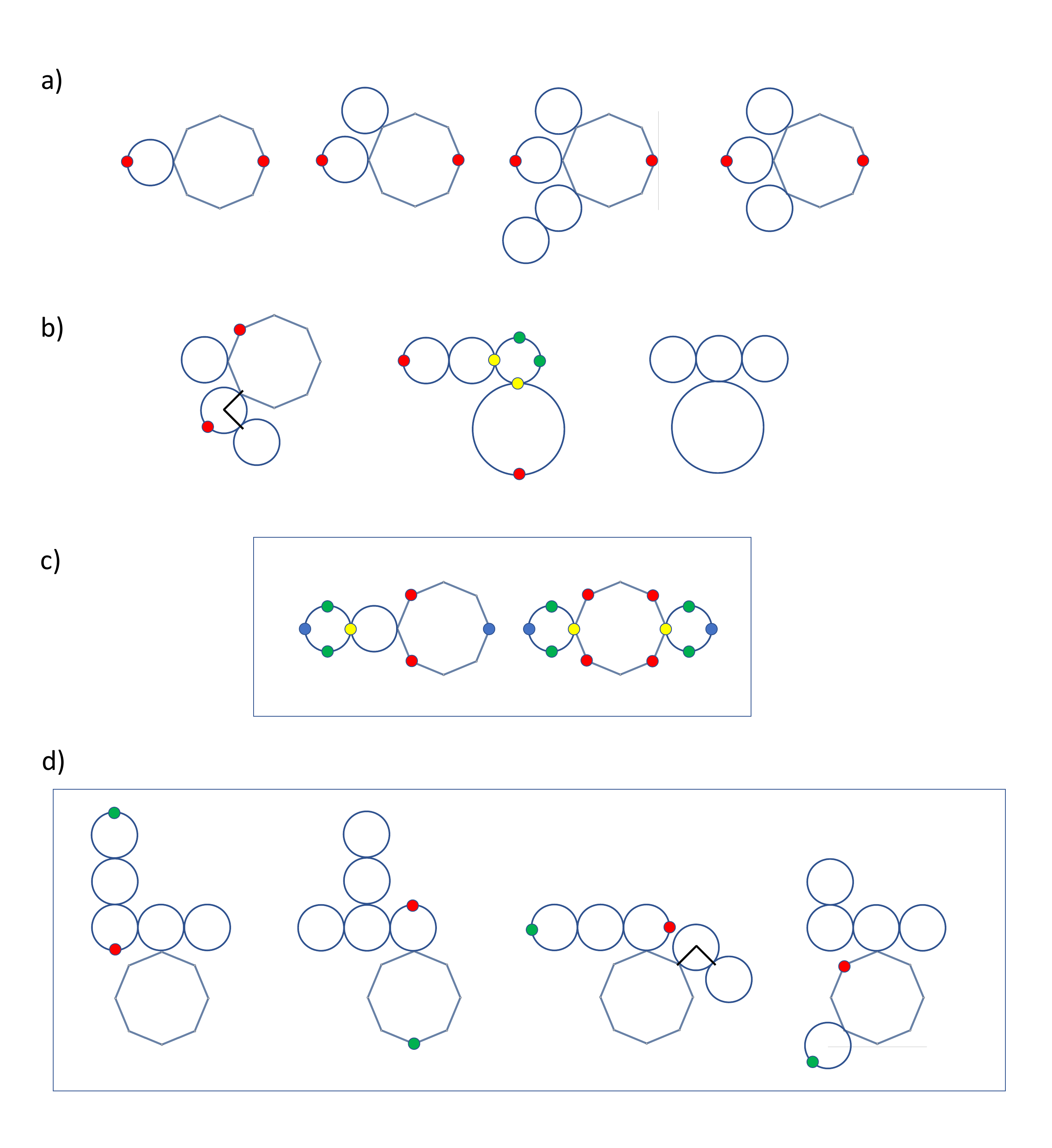}
\caption{ a-b) Individual graphs with their hot vertices. c) an isospectral pair and d) an isospectral set of four. In d) we only show hot vertices that exist on all graphs. There are 19 more sets of hot vertices that exist on two or three graphs. All vertices with the same colour have the same $M$-function if they belong to one graph or a set of isospectral graphs. Large loops have length eight and small loops have length four. We have not plotted hot vertices which are related by obvious symmetries.}
\label{fig:dhotvertices}
\end{figure}

\begin{figure}
\centering
\includegraphics[width=0.70\textwidth]{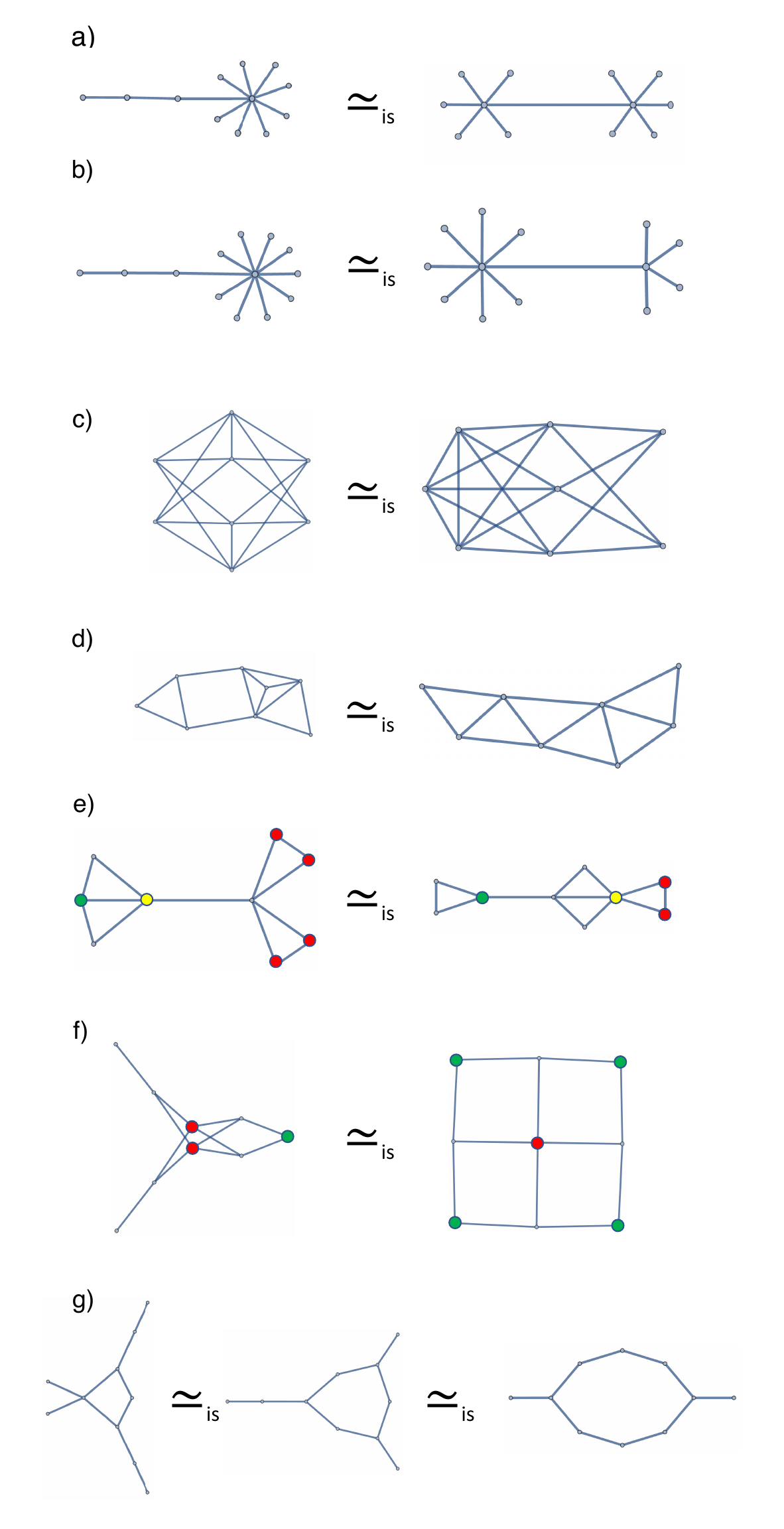}
\caption{ a) An isospectral pair where one member is a star graph with 9 leaves and the isospectral partner is a symmetric dumbbell graph. b) An isospectral pair where one member is a star graph with ten leaves and the partner is an asymmetric dumbbell graph. c) A complicated isospectral pair. d) An isospectral pair where one member forms a triangulation. e) An isospectral pair without pendant edges that can be embedded in the plane. f) An isospectral pair where one member forms a square grid. g) An isospectral triplet consisting of decorated loops. Vertices with the same colour have the same $M$-function within each isospectral pair. Some graphs do not have hot vertices.}
\label{fig:dinterestingones}
\end{figure}

\begin{figure}
\centering
\includegraphics[width=0.65\textwidth]{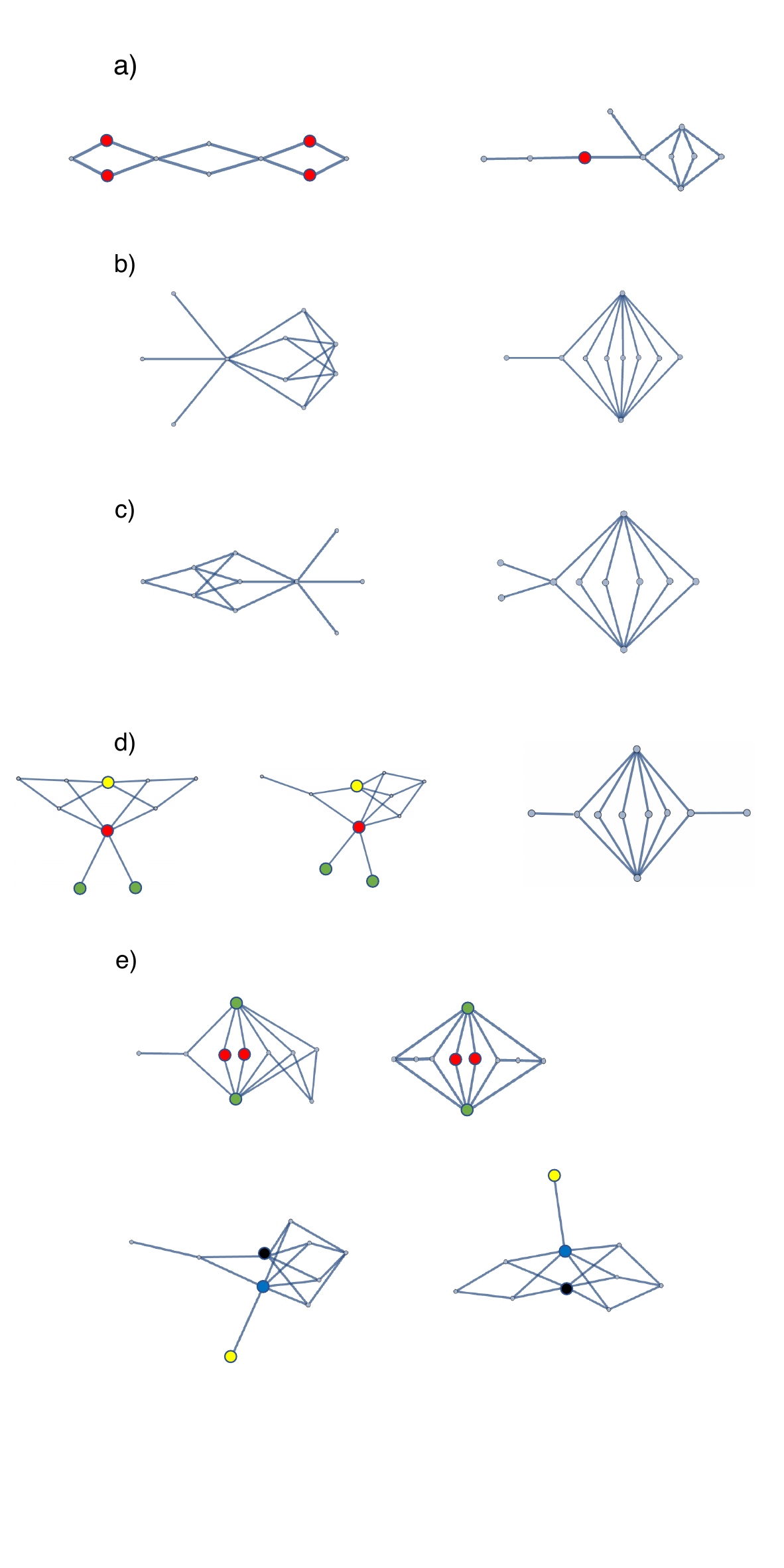}
\caption{ Examples of isospectral sets where the graphs have ten vertices. All the sets contain graphs which are related to pumpkin graphs, such as being decorated pumpkin graphs. a) An isospectral pair where one member consists of three connected loops. b-c) Two isospectral pairs where one member is a lightly decorated pumpkin graph. d) An isospectral triplet. e) A set of four isospectral graphs. All vertex sets with the same $M$-function are shown with the same colour. All vertices on all graphs were tested with respect to their $M$-functions.}
\label{fig:dtenvertices}
\end{figure}

\begin{figure}
\centering
\includegraphics[width=0.75\textwidth]{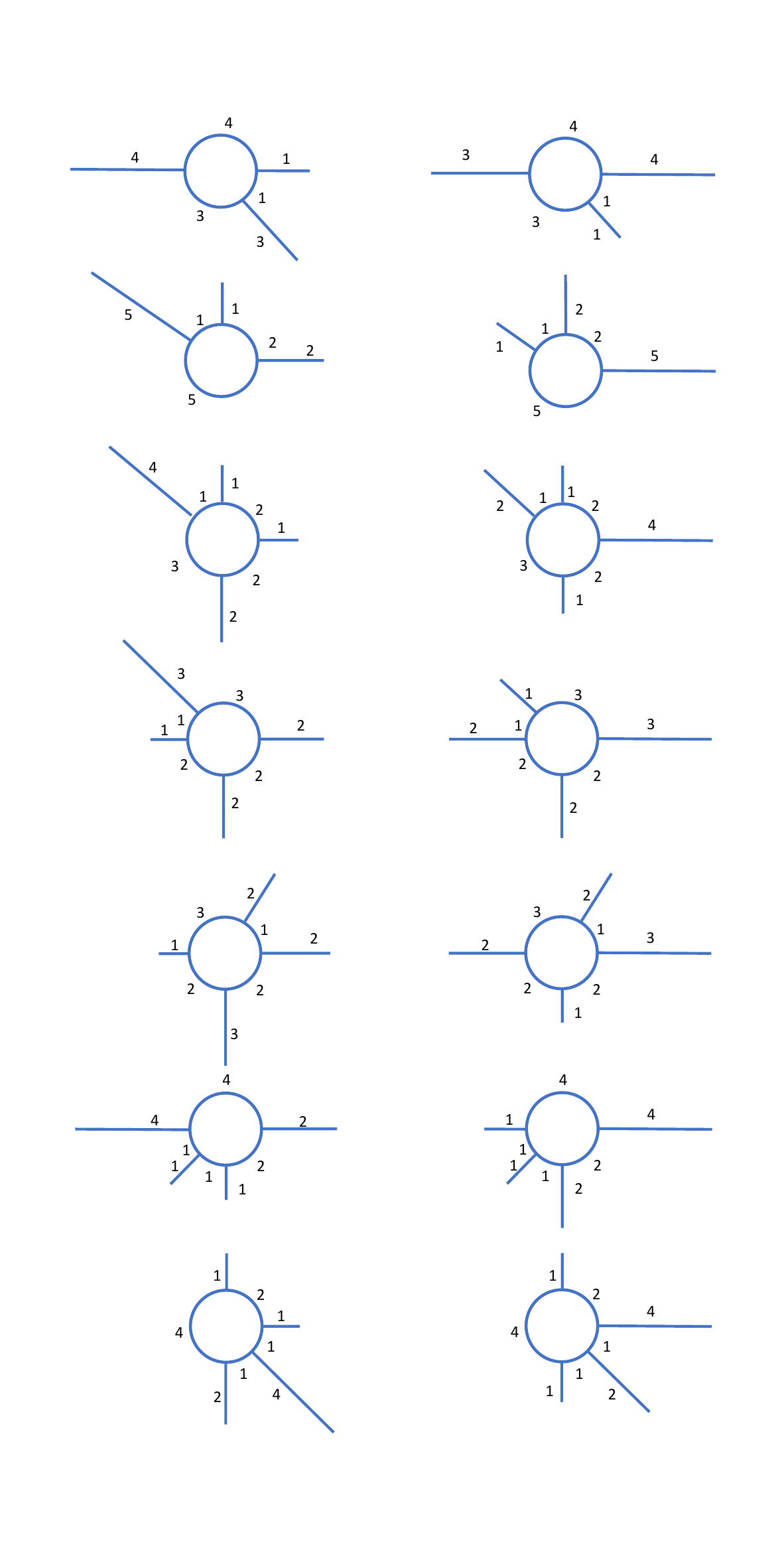}
\caption{ Examples of isospectral pairs consisting of loops decorated by edges. These graphs do not have any hot vertices, seen as equilateral graphs with 16 vertices.}
\label{fig:ddecoratedloops}
\end{figure}

\begin{figure}
\centering
\includegraphics[width=0.8\textwidth]{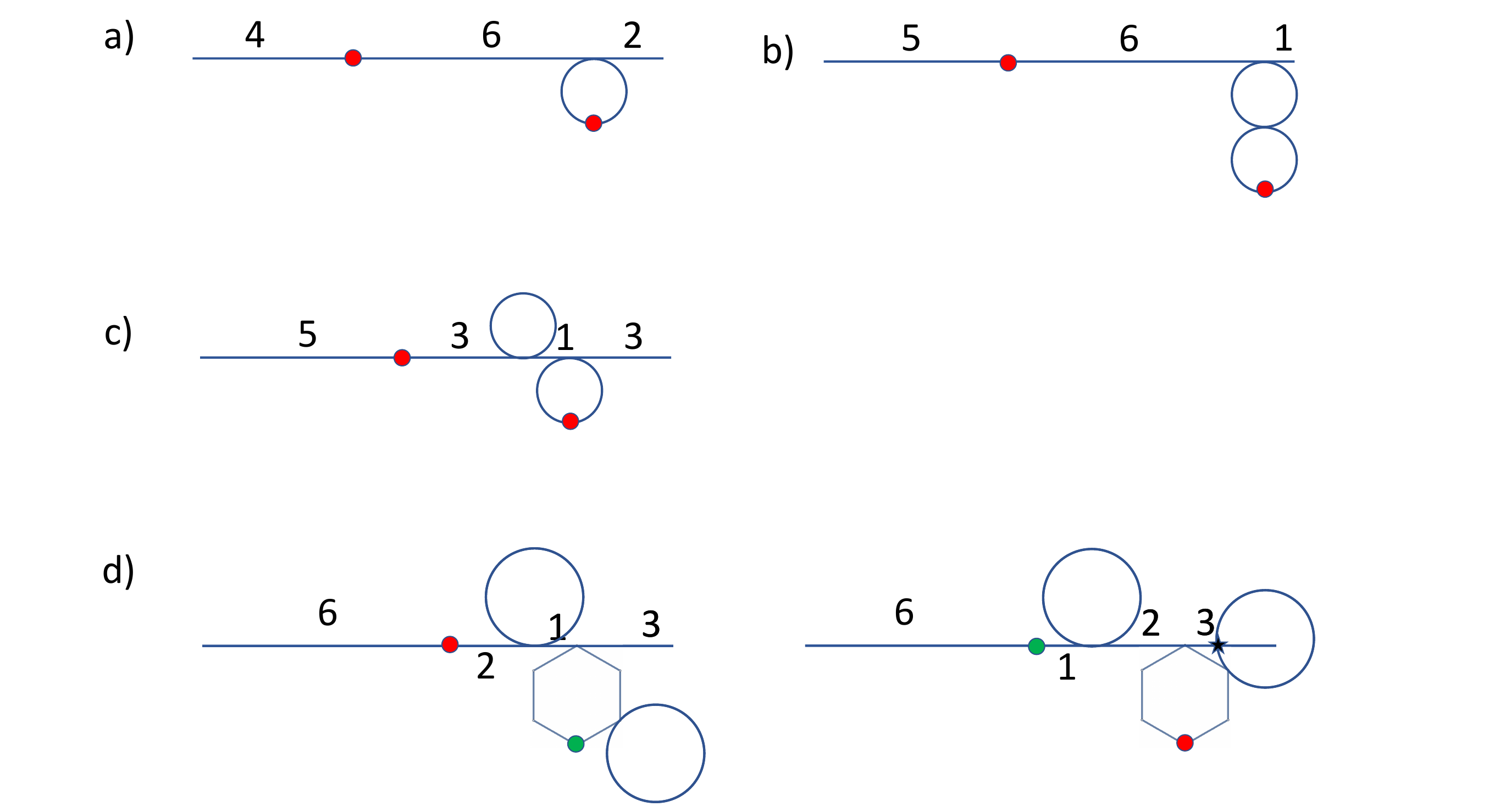}
\caption{a-c) Graphs with vertices having the same $M$-function indicated by red. d) An isospectral pair having two sets of vertices with the same $M$-function. Small loops have length four and large loops length six. The black star in d) indicates that there is no common vertex at this position.}
\label{fig:eintervalwithloops}
\end{figure}

\begin{figure}
\centering
\includegraphics[width=1.0\textwidth]{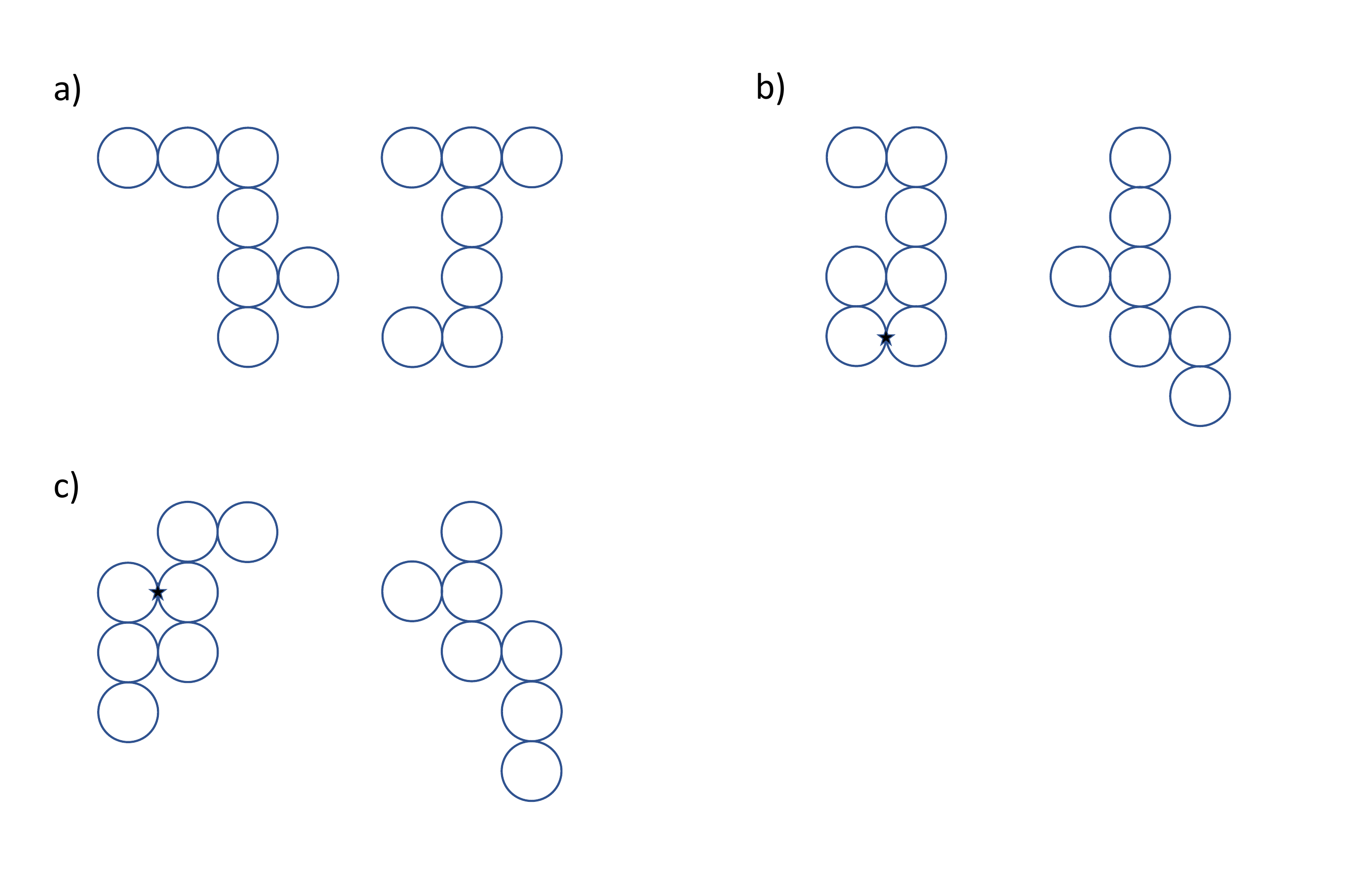}
\caption{ Three examples of isospectral pairs consisting of loops attached to loops. The loops have all the same length. A black star indicates that the two adjacent loops do not have a common vertex at this position. These graphs have no hot vertices seen as equilateral graphs with loops of length four.}
\label{fig:esamesizedloops}
\end{figure}

In Fig. \ref{fig:dtenvertices} we show examples of isospectral graphs having ten vertices. We here find that a graph consisting of three loops of equal length has an isospectral partner without much symmetry. Previously we have shown that single and double loops have isospectral partners. It is possible that chains of $n$ loops of equal length have isospectral partners for all $n$ where the isospectral partner has little or no symmetry. That is, the isospectral partner is not the simple substitution of an end loop with its isospectral partner (Fig. \ref{fig:aloop}). 

We also find one set of four isospectral graphs which does not seem to have much structure, in contrast to the set of four shown in Fig. \ref{fig:afirsttriplets}. 

The isospectral set in Fig. \ref{fig:dinterestingones}g) inspired us to find more examples of simple isospectral pairs involving loops. We thus generated loops decorated with edges having at most 16 vertices where we found several isospectral pairs. In Fig. \ref{fig:ddecoratedloops} we show examples of loops decorated with edges, such that any vertex has maximum valency three. 

Decorating loops with edges gives interesting isospectral pairs and we then decided to decorate an interval with loops. We then found that the resulting isospectral pairs could be obtained by decorating the graphs in Fig. \ref{fig:eintervalwithloops} with loops, and Fig. \ref{fig:eintervalwithloops} shows the pairs of vertices which have the same $M$-functions. Unfortunately it is computationally difficult to generate large sets of these fascinating isospectral pairs. We then generated loops with the same length attached to each other and Fig. \ref{fig:esamesizedloops} shows isospectral pairs generated in this way. The graphs in Fig. \ref{fig:esamesizedloops} do not have any two vertices with the same $M$-function. We only checked the vertices if the graphs are seen as equilateral graphs where each loop has length four.

There are many more interesting examples of isospectral sets, and what is interesting depends on the reader.\\
\clearpage

\section{Dirichlet and more general, $\delta$-type, boundary conditions.}

So far we have investigated quantum graphs having standard and Neumann boundary conditions. However there is considerable interest in quantum graphs having other types of boundary conditions such as Dirichlet boundary conditions at terminal vertices and other boundary conditions elsewhere. Dirichlet boundary conditions at vertices with higher valency than one is not particularly interesting since the relevant edge then essentially decouples as pointed out by e. g. Berkolaiko \cite{berkolaiko2017elementary} and we will only consider Dirichlet boundary conditions at terminal vertices, not elsewhere. The situation with Dirichlet boundary conditions at terminal vertices and standard otherwise corresponds well to quantum mechanics where the wave function is indeed zero at the boundary of a finite system, \cite{Messiah}. Nanowires can be constructed and the electronic spectrum can be investigated by optical means, as one example of a quite accessible fairly one-dimensional quantum mechanical system \cite{Vainorius.2016}.

There are more general $\delta(\alpha)$-type boundary conditions:

\begin{subequations} \label{eq:delta}
\begin{empheq}[left=\empheqlbrace]{align}
\displaystyle  f(x_i)=f(x_j)&=f_0,   \qquad              x_i, x_j \in V_m, \label{eq:delta1} \\[3mm]  
\displaystyle\sum_{x_i \in V_m}\partial_{n}f(x_i)&=\alpha f_0 \label{eq:delta2}
\end{empheq}
\end{subequations}

%\begin{equation} \label{delta}
%\left\{
%\begin{array}{ll}
%\displaystyle  f(x_i)=f(x_j)=f_0, & x_i, x_j \in V_m, \\[3mm]
%\displaystyle \sum_{x_i{ \in V_m} }\partial_{n}f(x_i)=\alpha f_0 &
%\end{array} \right. 
%\end{equation}

which depend on a parameter $\alpha$ where we use the same notation as in Eqn. 1. If $\alpha = 0$ we have the standard boundary conditions. The connection between $\alpha$ and the vertex scattering amplitudes has been given by e. g. by Band and Gnutzmann \cite{bandgnutzmann} and is given in Table \ref{tab:Scatteringamplitudes}.

We also have "dual" $\delta'_s(\beta)$-type boundary conditions:
 
\begin{subequations} \label{eq:delta'_s}
\begin{empheq}[left=\empheqlbrace]{align}
\displaystyle  \partial_{n}f(x_i)=\partial_{n}f(x_j)&=\partial_{n}f_0,   \qquad x_i, x_j \in V_m, \label{eq:delta'_s1}\\[3mm]
\displaystyle \sum_{x_i{ \in V_m} }f(x_i)&=\beta \:  \partial_{n}f_0 \label{eq:delta'_s2}
\end{empheq}
\end{subequations}
 
which depend on a parameter $\beta$. Such boundary conditions have been investigated by Exner and Turek \cite{ExnerTurek} and the vertex scattering amplitudes are given in \cite{ExnerPRL,ExnerSeba}.

Another set of boundary conditions are the $\delta'(\gamma)$-type boundary conditions:

\begin{subequations} \label{eq:delta'}
\begin{empheq}[left=\empheqlbrace]{align}
\displaystyle  f(x_i)-f(x_j) &=\gamma (\partial_{n}f(x_i)-\partial_{n}f(x_j)), \qquad x_i, x_j \in V_m, \label{eq:delta'1}\\[3mm]
\displaystyle \sum_{x_i{ \in V_m} }\partial_{n}f(x_i) &=0  \label{eq:delta'2}
\end{empheq}
\end{subequations}

once again with vertex scattering amplitudes given in \cite{ExnerPRL,ExnerSeba}. 

Also for these boundary conditions there is a dual, $\delta_p(\epsilon)$-type:

\begin{subequations} \label{eq:delta_p}
\begin{empheq}[left=\empheqlbrace]{align}
\displaystyle  \partial_{n}f(x_i)-\partial_{n}f(x_j)&=\epsilon (f(x_i)-f(x_j)),  \qquad x_i, x_j \in V_m, \label{eq:delta_p1}\\[3mm]
\displaystyle \sum_{x_i{ \in V_m} }f(x_i)&=0   \label{eq:delta_p2}
\end{empheq}
\end{subequations}

\begin{table}
\caption{\label{Table 2:all} Vertex scattering amplitudes for different boundary conditions. $d$ is the valency of the vertex.}
    \begin{tabular}{lll} 
     &           &             \\ \hline 
    Boundary condition        & Backscattering amplitude                                                  & Forward scattering amplitude                                            \\ \hline
      &           &             \\
    $\delta(\alpha)$  & $ \frac{2 k}{d k+i \alpha} -1$ & $ \frac{2 k}{d k+i \alpha} $ \\
      &           &             \\
    $\delta'_s(\beta)$ & $\frac{2}{\beta i k - d} + 1$                     & $\frac{2}{\beta i k - d}$                       \\ 
      &           &             \\ 
      $\delta'(\gamma)$ & $ \frac{\gamma i k + 1}{\gamma i k - 1} +\frac{2 }{d (1- \gamma i k)} $                     & $\frac{2}{d (1- \gamma i k)}$                       \\ 
      &           &             \\ 
      $\delta_p(\epsilon)$ & $ \frac{ k-i \epsilon }{ k+i \epsilon }-\frac{2 k}{d (k+i \epsilon)}$                     & $-\frac{2 k}{d (k+i \epsilon)}$                       \\ 
      &           &             \\
      $\delta(0)$, standard                   & $ \frac{2}{d} -1$     \newline                                    & $ \frac{2}{d} $                                          \\ 
      &           &             \\
      $\delta'_s(0)$ & $-\frac{2}{d} + 1$                     & $-\frac{2}{d}$                       \\ 
      &           &             \\ 
       $\delta'(0)$ & $\frac{2 }{d } -1$                     & $\frac{2}{d}$                       \\ 
      &           &             \\ 
      $\delta_p(0)$ & $ -\frac{2}{d }+1$                     & $-\frac{2}{d}$                       \\ 
      &           &             \\ \hline
    \end{tabular}
    \label{tab:Scatteringamplitudes}
\end{table}

We have implemented these boundary conditions (Dirichlet or Neumann at terminal vertices and $\delta$-, $\delta'_s$-, $\delta'$ or $\delta_p$-type otherwise) in one of our programs. The program uses the method of edge and vertex scattering matrices \cite{berkolaiko2017elementary}, and has been tested against the few known results in literature. In a few cases we calculated the secular determinant by hand. However the testing is far less than for the two programs that only implement standard and Neumann bondary conditions. Thus we urge the reader to be cautious and doublecheck our results before relying on them. 
We need some notation in the following:

\begin{itemize}
\item $\delta(\alpha)$: All vertices with valence greater than two have $\delta$-type boundary conditions with parameter $\alpha$. 
\item $\delta$-type: All vertices with valence greater than two have $\delta$-type boundary conditions. 
\item $\delta_n(\alpha)$: We have $\delta$-type boundary conditions at all vertices with valency $n$, with parameter $\alpha$. 
\item $\delta_{n_1,n_2,n_3, ...}(\alpha_1,\alpha_2,\alpha_3, ...)$: We have $\delta$-type boundary conditions at all vertices with valency $n_i$, and $\alpha_i$ is the particular parameter used for these vertices. \item $\delta_{v_1, v_2, ...}(\alpha_1, \alpha_2, ...)$: We have $\delta$-type boundary conditions at vertices $v_i$, and $\alpha_i$ is the particular parameter used for these vertices. \\
This situation does not occur often in this paper. Our investigations pertain, with few exceptions, to the situation where all vertices with the same valency have the same boundary conditions. This is a severe limitation, but has been hard enough to implement in software. 
\end{itemize}

Naturally we can also use $\delta'_s$, $\delta'$ or $\delta_p$ in the above. All vertices that are not listed above have standard boundary conditions unless otherwise noted. Terminal vertices are specified separately if they have Dirichlet boundary conditions and the default is Neumann boundary conditions. We usually do not introduce any boundary conditions at vertices with valence two and they can be removed from the graphs. \\

\subsection{Eigenvalue zero}

There are cases where the Laplacian on graphs, $\bf{L}$, have eigenvalue zero. Consider connected graphs with standard boundary conditions, $\delta(0)$, at interior (non-terminal) vertices. If the boundary conditions at terminal vertices are all Neumann, then zero is an eigenvalue of $\bf{L}$ with spectral multiplicity one. If at least one terminal vertex has Dirichlet boundary conditions then zero is not an eigenvalue. See Ref. \cite{kurasov2005inverse} for a proof that the eigenfunction corresponding to the zero eigenvalue is a constant if the graph is connected. The situation is more complicated for more general boundary conditions.\\ 

Let's investigate general connected equilateral graphs with permutation invariant boundary conditions, Eqns. \eqref{eq:delta} - \eqref{eq:delta_p}. We make the coordinate systems more symmetric by shifting the zero on each edge to the center of the edge. The eigenfunctions corresponding to eigenvalue zero can be written $2 a x + b$. The domain of the function is $[-1/2,1/2]$, the range is $[-a +b,a+b]$, and it has normal derivatives $2 a$ and $-2 a$ at its left and right endpoints, respectively.
Let us label the vertices $V_i$ and the edges $e_{ij}$. The edge $e_{ij}$ connects vertices $V_i$ and $V_j$. Let the eigenfunction on $e_{ij}$ be $f_{ij}(x)=2a_{ij}x+b_{ij}$, where $i<j$. That is, $x$ increases from the vertex with the lowest index, $i$, to the vertex with the highest index, $j$. We define $a_{ji}=a_{ij}$ and $b_{ji}=b_{ij}$, in order to simplify the discussion below.\\ \\
$\delta(\alpha)$ boundary conditions, Eqn. \eqref{eq:delta}, imply      

\begin{subequations} \label{eq:deltaalpha}
\begin{empheq}[left=\empheqlbrace]{align}
\displaystyle  - \epsilon_{ij}a_{ij}+b_{ij}&=-\epsilon_{ik}a_{ik}+b_{ik}  \\[3mm]
\displaystyle \sum_{j} \epsilon_{ij}2 a_{ij}&=\alpha (-\epsilon_{ik}a_{ik}+b_{ik})  & \label{eq:deltaalphaderivative}
\end{empheq}
\end{subequations}

at a non-terminal vertex $V_i$, where $j$ and $k$ labels vertices connected to $V_i$. In Eq. \eqref{eq:deltaalphaderivative} $k$ can be the label of any vertex connected to $V_i$.  $\epsilon_{ij}$ is the signature of $i$ and $j$, which is one if $i<j$, and minus one if not. We have val($V_i$) equations for each vertex, where val($V_i$) is the valency of vertex $V_i$. This is true also if the valency of $V_i$ is one where the first equation above is not used. We have $2 E$ parameters (not counting $\alpha$), where $E$ is the number of edges. We also know that $\sum_i \text{val}(V_i) =2 E$ for any graph. This means that the number of parameters is the same as the number of equations.
Rearranging Eqn. \ref{eq:deltaalpha} we get:

\begin{equation} \label{eq:deltaalphamatrix}
\bf {M C = 0}.
\end{equation} 
$\bf{C}$ is a column vector containing the elements $a_{ij}$ and $b_{ij}$. $\bf{M}$ is a square matrix of dimension $2 E$ where most of the elements are small natural numbers. However, each vertex contributes one row containing two elements containing $\alpha$ ($ \pm 2 \pm \alpha$ and $\pm \alpha$, where the signs depend on the vertex). This can be seen from  Eqn. \ref{eq:deltaalphaderivative}).\\

%We give a theorem:
%\begin{theorem} \label{th:determinantzero}
%For all compact graphs with $\delta(\alpha)$-type boundary conditions it is never the case that the determinant of $\bf{M}$ is zero for all $\alpha$. We conjecture that this is true also if the boundary conditions are $\delta'_s(\beta)$-type, $\delta'(\gamma)$-type or $\delta_p(\epsilon)$-type, with the appropriate variable.
%\end{theorem} 
%\begin{proof}
%A determinant of $\bf{M}$ (which is a polynomial in $\alpha$) which is zero for all $\alpha$ is a very strong condition. Let's investigate what happens in the limit $\alpha \to \infty$. This corresponds to Dirichlet boundary conditions and in such a case we do not have eigenvalue zero. Thus the system $\bf {M C = 0}$ has only the trivial solution in the $\alpha \to \infty$ limit and the determinant of the matrix $\bf M$ in this limit must be non-zero. Consequently there must be some finite $\alpha$ such that the determinant of $\bf M$ is not zero.
%\end{proof}
%

%Given Theorem \ref{th:determinantzero} we find:

\begin{theorem} \label{th:delta}
If we have a compact graph with $\delta(\alpha)$-type boundary conditions, there are a only a finite number of $\alpha$, (which cannot be positive), such that the zero eigenvalue exist. If $\alpha=0$ then the graph has eigenvalue zero.
\end{theorem} 
\begin{proof}
In order to have a non-trivial solution to Eqn. \ref{eq:deltaalphamatrix} we require that the determinant of $\bf{M}$ is zero. This determinant is a polynomial, $p(\alpha)$, where the degree is at most the same as the number of vertices. We first exclude the case that $p(\alpha)$ is zero for all $\alpha$. Let us investigate what happens in the limit $\alpha \to \infty$. This corresponds to Dirichlet boundary conditions, $f_0=0$ in Eqn.  \eqref{eq:delta}, and in such a case we do not have eigenvalue zero. Thus the system $\bf {M C = 0}$ has only the trivial solution, $\bf {C = 0}$, in the $\alpha \to \infty$ limit and the determinant of the matrix $\bf M$ in this limit must be non-zero. Consequently there must be some finite $\alpha$ such that the determinant of $\bf M$ is not zero and $p(\alpha)$ cannot be identically zero. $p(\alpha)$ has thus a finite number of roots, equal to the degree, and only in those cases do we have eigenvalue zero. 

Suppose an eigenfunction $f>0$ has a maximum at some vertex. Then all normal derivatives must be non-positive there, thus $\alpha\leq0$. If $f<0$ everywhere, then it has a minimum where all normal derivatives must be positive, thus $\alpha\leq0$. $\alpha=0$ is allowed since $\delta(0)$ has eigenvalue zero. The corresponding eigenfunction is the constant function for connected graphs with spectral multiplicity one. With spectral multiplicity we mean the dimension of the eigenspace of eigenvalue zero.
\end{proof}
Note that the proof is essentially unchanged even if the graph is not equilateral, but has edges with arbitrary, but finite lengths.\\

%This theorem can be generalised to situations where each vertex $V_i$ has their own $\delta(\alpha_i)$-type boundary condition, with $\alpha_i$ depending on the vertex. There will be finite number of sets $\{\alpha_i\}_{\{i = 1 ... n\}}$ such that the zero eigenvalue exist. This number is the number of vertices.\\
 
 Let us now consider equilateral graphs with $\delta'_s(\beta)$-type boundary conditions. We get the following equations instead:\\
 
\begin{subequations} \label{eq:deltaalphaprimes}
\begin{empheq}[left=\empheqlbrace]{align}
\displaystyle  \epsilon_{ij}2a_{ij}&=\epsilon_{ik}2a_{ik}  \\[3mm]
\displaystyle \sum_{j} -\epsilon_{ij}a_{ij}+b_{ij}&=\beta \epsilon_{ik}2a_{ik}  & \label{eq:deltaalphaprimesderivative}
\end{empheq}
\end{subequations}

For $\delta'_s(\beta)$-type boundary conditions we have:

\begin{theorem} \label{th:delta'_s}
%Assume $\delta'_s(\beta)$-type boundary conditions. 
Let $\Gamma=(V,E)$ be a connected equilateral metric graph equipped with the Laplace operator subject to $\delta'_s(\beta)$-type boundary conditions at all vertices.
%A graph that has zero as an eigenvalue with $a\neq 0$ must be a bipartite graph and any bipartite graph has eigenvalue zero. For such a solution $\beta= -E/V$, where E is the number of edges and V is the number of vertices. 

The graph $\Gamma$ admits a non-constant eigenfunction for eigenvalue $\lambda=0$ if and only if $\Gamma$ is a bipartite graph and the coupling parameter satisfies: $\beta = -E/V$.

 If $\Gamma$ is not bipartite, or if $\beta \neq -E/V$, then any existing eigenfunction corresponding to $\lambda=0$ must be piecewise constant on the edges. 

There are graphs without zero as an eigenvalue.
\end{theorem}

%\begin{theorem} \label{th:delta'_s}
%Let $\Gamma=(V,E)$ be a connected equilateral metric graph equipped with the Laplace operator subject to $\delta'_s(\beta)$-type boundary conditions at all vertices. Let eigenfunctions corresponding to the eigenvalue $\lambda=0$ be denoted by $f(x)$.
%
%\begin{enumerate}
%    \item \textbf{Non-constant Eigenfunctions:} The graph $\Gamma$ admits a non-constant eigenfunction for $\lambda=0$ (specifically, one with non-vanishing derivatives on the edges) if and only if $\Gamma$ is a bipartite graph and the coupling parameter satisfies:
%    \begin{equation*}
%        \beta = -\frac{|E|}{|V|}
%    \end{equation*}
%    
%    \item \textbf{Piecewise Constant Eigenfunctions:} If $\Gamma$ is not bipartite, or if $\beta \neq -|E|/|V|$, then any existing eigenfunction corresponding to $\lambda=0$ must be piecewise constant on the edges (i.e., the derivative $f'(x) \equiv 0$ everywhere). Consequently, the kernel of the Laplacian contains only piecewise constant functions or is trivial.
%\end{enumerate}
%\end{theorem}

\begin{proof} 
We first note that the eigenfunctions on adjacent vertices (i. e. sharing an edge) have opposite signs of the normal derivatives at the two vertices, if $a\neq 0$. We can thus color the vertices, say red and blue, corresponding to the sign of the normal derivatives, such that no two adjacent vertices have the same color. This is precisely the definition of a bipartite graph. Let us label the $m$ red vertices (with positive normal derivatives) from $1$ to $m$ and the rest from $m+1$ to $m+n$, where $n$ is the number of blue vertices. The eigenfunctions are $2ax+b_{ij}$ and the normal derivatives are $2 a$ and $-2 a$ at red and blue vertices, respectively. The number of red vertices is $V_{red}$, the number of blue vertices is $V_{blue}$ and the total number of vertices is $V$. \\

We have at a red vertex, say $V_i$, with valence $v_i$ and at a blue vertex, say $V_k$, with valence $v_k$: 

\begin{subequations} \label{eq:deltap}
\begin{empheq}[left=\empheqlbrace]{align}
\displaystyle  - v_i a + \sum_{j=1}^{v_i} b_{ij}=\beta 2 a   \label{eq:deltapderivative} \\[3mm]
\displaystyle v_k a + \sum_{i=1}^{v_k} b_{ik}=-\beta 2 a   \label{eq:deltapderivative}
\end{empheq}
\end{subequations}

Summing over all red and blue vertices we have:

\begin{subequations} \label{eq:deltap}
\begin{empheq}[left=\empheqlbrace]{align}
\displaystyle \sum_{i=1}^{V_{red}}  (- v_i a + \sum_{j} b_{ij})&=V_{red}\beta 2 a   \label{eq:deltapderivative} \\[3mm]
\displaystyle \sum_{k=1}^{V_{blue}}  ( v_k a + \sum_{i} b_{ik})&=-V_{blue}\beta 2 a   \label{eq:deltapderivative}
\end{empheq}
\end{subequations}

for red and blue vertices, respectively, where we note that $i<j$. This can be written:

\begin{subequations} \label{eq:deltap}
\begin{empheq}[left=\empheqlbrace]{align}
\displaystyle -E a +  \sum_{ij} b_{ij}&=V_{red}\beta 2 a  \label{eq:deltapderivative} \\[3mm]
\displaystyle E a +  \sum_{ij} b_{ij}&=-V_{blue}\beta 2 a  \label{eq:deltapderivative}
\end{empheq}
\end{subequations}
Taking the sum and difference of these equations we get:

\begin{subequations} \label{eq:deltap}
\begin{empheq}[left=\empheqlbrace]{align}
\displaystyle 2 \sum_{ij} b_{ij}&=(V_{red}-V_{blue})\beta 2 a  \label{eq:deltapderivative} \\[3mm]
\displaystyle - 2 E a&=V\beta 2 a  \label{eq:deltapderivative}
\end{empheq}
\end{subequations}

and we see that $\beta=-E/V$ if $a \neq 0$. If $\beta=-E/V$ then $a$ can be any number and we have a consistent system with at least one solution (for trees) and many solutions if the graph contains cycles.\\

For an example of a graph which is not bipartite but which has zero as an eigenvalue see Fig. \ref{fig:feigenvaluezero} and $\beta$ can have any value. The eigenfunction is constant on each edge. The triangular graph does not have zero as an eigenvalue.
\end{proof}

For $\delta'(\gamma)$ boundary conditions, we get:

\begin{subequations} \label{eq:deltagamma}
\begin{empheq}[left=\empheqlbrace]{align}
\displaystyle (- \epsilon_{ij}a_{ij}+b_{ij}-(-\epsilon_{ik}a_{ik}+b_{ik}))&= \gamma (\epsilon_{ij}2 a_{ij}- \epsilon_{ik}2 a_{ik}) \\[3mm]
\displaystyle \sum_{j} \epsilon_{ij}2 a_{ij}&=0  & \label{eq:deltagammaderivative}
\end{empheq}
\end{subequations}

\begin{figure}
\centering
\includegraphics[width=1.0\textwidth]{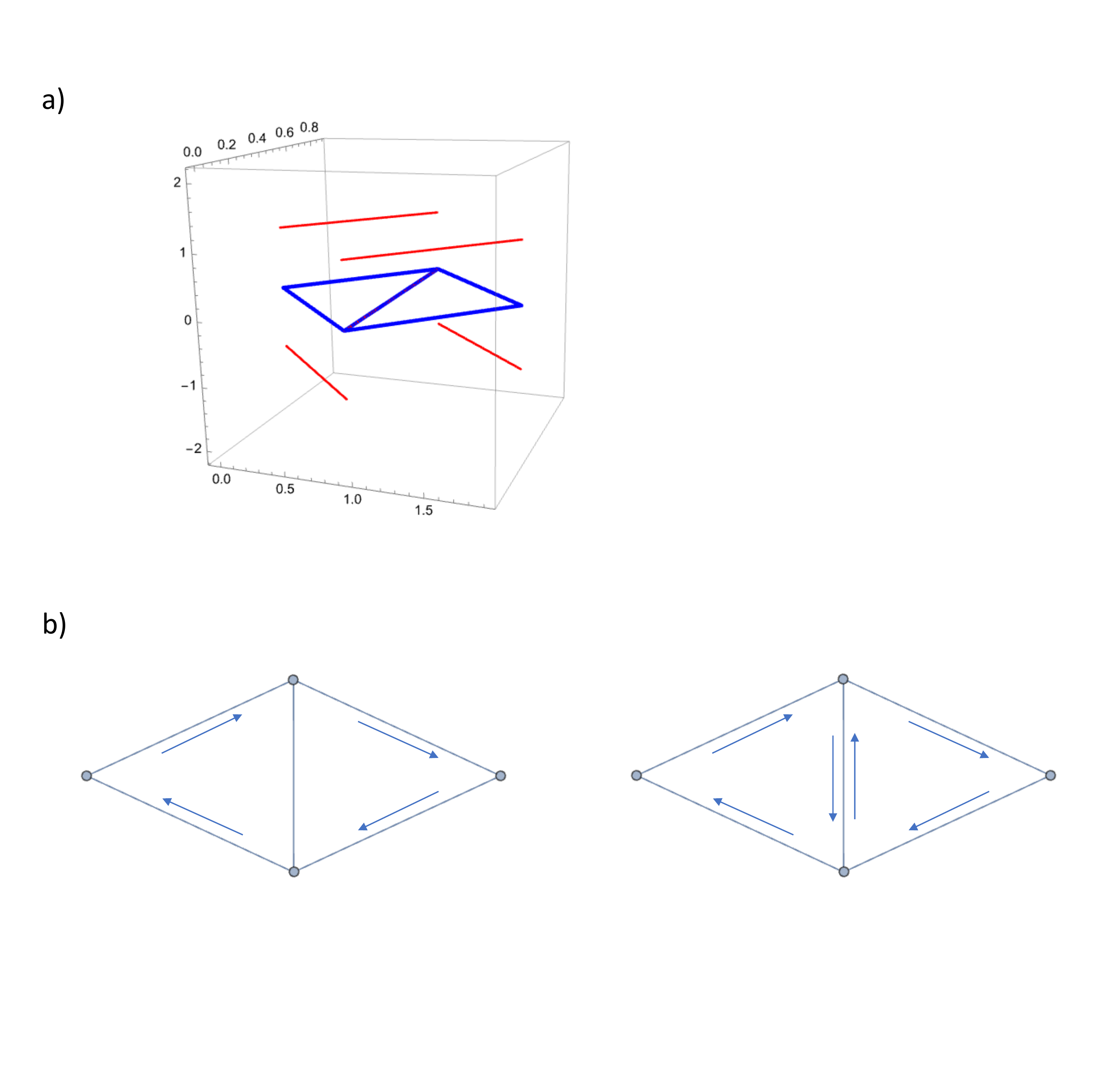}
\caption{a) A graph (blue) which has eigenvalue zero under $\delta'_s(\beta)$-type boundary conditions, along with the eigenfunction, red. This graph is not bipartite. b) $\delta'(\gamma)$-type boundary conditions. An illustration of an eigenfunction, with support on a cycle, that can be written as the sum of two eigenfunctions, with support on two independent cycles. The arrows illustrate the flow of the eigenfunctions.}
\label{fig:feigenvaluezero}
\end{figure}

\begin{theorem} \label{th:delta'}
Let $\Gamma=(V,E)$ be a connected equilateral metric graph equipped with the Laplace operator subject to $\delta'(\gamma)$-type boundary conditions. Let $g = |E| - |V| + 1$ denote the cyclomatic number (the number of independent cycles) of $\Gamma$. The geometric multiplicity of the eigenvalue $\lambda = 0$, denoted $m(0)$, is determined as follows:

\begin{enumerate}
    \item If $g = 0$ (i.e., $\Gamma$ is a tree), then $m(0) = 0$.
    \item If $g \geq 1$ and $\Gamma$ possesses at least one terminal vertex:
    \begin{equation*}
        m(0) = 
        \begin{cases} 
        g & \text{if } \gamma = -1/2, \\
        0 & \text{if } \gamma \neq -1/2.
        \end{cases}
    \end{equation*}
    In the case where $\gamma = -1/2$, the eigenspace is spanned by $g$ linearly independent functions with support on the cycles.
    \item If $g \geq 1$ and $\Gamma$ possesses no terminal vertices:
    \begin{equation*}
        m(0) = 
        \begin{cases} 
        g + 1 & \text{if } \gamma = -1/2, \\
        1 & \text{if } \gamma \neq -1/2.
        \end{cases}
    \end{equation*}
    Here, the constant function is an eigenfunction for all $\gamma$. If $\gamma = -1/2$, the eigenspace is spanned by the constant function and $g$ linearly independent functions with support on the cycles.
\end{enumerate}
\end{theorem}

%\begin{theorem} \label{th:delta'}
%Assume $\delta'(\gamma)$-type boundary conditions and connected equilateral graphs. We  have three situations:\\
%\begin{itemize}
%\item The graph contains a cycle and has a terminal vertex. The graph then has eigenvalue zero and there is an eigenfunction, $ax$, with $\gamma= -1/2$ and $a \neq 0$ which has support only on the cycle. The spectral multiplicity of eigenvalue zero is equal to the number of independent cycles.
%\item The graph contains a cycle and but has no terminal vertex. The graph then has eigenvalue zero and there is an eigenfunction, $ax$, with $\gamma= -1/2$ and $a \neq 0$ which has support only on the cycle. The constant function is an eigenfunction for all $\gamma$. The spectral multiplicity of eigenvalue zero is equal to the number of independent cycles plus one.
%\item The graph does not contain a cycle. The graph does not have eigenvalue zero.
%\end{itemize}
%\end{theorem}

\begin{proof}
When we say eigenfunction we always mean an eigenfunction with eigenvalue zero. We first observe that the eigenspace is linear. We also note that the constant function, $b$, is a solution for any $\gamma$, if there are no terminal vertices. This can be seen from Eqns. \eqref{eq:delta'}. However if we have a terminal vertex, then a possible eigenfunction must be zero on the associated pendant edge and no constant eigenfunction exist for such graphs. We also see that there are no piecewise constant eigenfunctions, $b_i$, that are constant on each edge, but where $b_i \neq b_j$, since this will immediately violate Eqn. \eqref{eq:delta'1}.\\
We will in the rest of the proof let $b=0$ in any eigenfunction $ax+b$, unless otherwise noted, since the constant factor can be added in the end.

Consider a vertex $V$ with valence at least two and let the eigenfunction on edge $e_{j}$ emanating from $V$ be $2 a_j x$, with $x=-1/2$ at $V$. Then from Eqns. \eqref{eq:delta'}:
\begin{subequations} \label{eq:starvertex}
\begin{empheq}[left=\empheqlbrace]{align}
\displaystyle a_{j}-a_{k}&=-\gamma(2a_{j}-2a_{k})  \label{eq:starvertex1} \\[3mm]
\displaystyle 2 \sum_{j} a_{j}&=0  \label{eq:starvertex2}
\end{empheq}
\end{subequations}

In order to have a solution, with at least two $a_j$ different from zero, it is necessary that $\gamma = -1/2$.\\
Let's define the $\mathbf{flow}$ of an eigenfunction on an edge as the direction of the coordinate system of the edge. The magnitude of the flow is defined as $a$ and the direction of the flow is along the direction where the eigenfunction increases. From Eqn. \eqref{eq:starvertex2} we see that for eigenfunctions the flow is conserved at each vertex, that is, Kirchoffs law holds, similar to the case of standard boundary conditions. At a pendant edge we have no flow.

Consider a graph with a cycle with vertices $V_1, V_2, ... , V_n$, such that $V_1$ is connected to $V_2$ and so on until $V_n$ is connected to $V_1$. The edges are $\{e_{12}, e_{23}, ... , e_{n1}\}$. Let the eigenfunction on edge $e_{ij}$ be $2 a x$, when $e_{ij}$ belong to the cycle and zero otherwise. This eigenfunction will satisfy Eqns. \eqref{eq:delta'} with $\gamma= -1/2$ for any $a$. There is thus a non-constant function which is an eigenfunction. Note that $a$ can be any number but it has to be the same for each edge. This can be easily understood by flows. We can have flow around a cycle.

We will now show that an eigenfunction on any cycle which is not independent, can be written as a linear combination of eigenfunction on independent cycles. Suppose we have an eigenfunction on a cycle which is not independent. It can easily be seen that the eigenfunction can be written as the sum of two eigenfunctions on smaller cycles as illustrated in Fig. \ref{fig:feigenvaluezero}. If any of these smaller cycles is not independent, we continue the process until we have a set of independent cycles such that the initial eigenfunction is a sum of the eigenfunctions on the independent cycles. The concept of flows makes this easy to understand.

We find a non-constant eigenfunction can have support only on edges belonging to a cycle. Define a cyclefree subtree of a graph as a subtree where no edge belongs to a cycle of the full graph. Define a maximal cyclefree subtree as a cyclefree subtree that is not a subtree of a bigger cyclefree subtree. The terminal vertices of such a maximal cyclefree subtree must either end at a valence one vertex or end at a vertex where all other edges belong to a cycle. The flow on such a maximal cyclefree subtree must be zero, since neither terminal vertices nor cycles can act as sources or sinks for the flow. Thus non-constant eigenfunctions have support only on cycles.

We find that the number of eigenfunctions (the spectral multiplicity) is the same as the number of independent cycles plus one (due to the constant function) if there are no terminal vertices. If there are terminal vertices we do not have the constant function as an eigenfunction and the spectral multiplicity is the same as the number of independent cycles.
\end{proof}

For $\delta_p(\epsilon)$ boundary conditions, we get:

\begin{subequations} \label{eq:deltap}
\begin{empheq}[left=\empheqlbrace]{align}
\displaystyle  \epsilon (- \epsilon_{ij}a_{ij}+b_{ij}-(-\epsilon_{ik}a_{ik}+b_{ik}))&=  \epsilon_{ij}2 a_{ij}- \epsilon_{ik}2 a_{ik}  \label{eq:deltapder} \\[3mm]
\displaystyle \sum_{j} -\epsilon_{ij}a_{ij}+b_{ij}&=0  & \label{eq:deltapfkn}
\end{empheq}
\end{subequations}

\begin{theorem} \label{th:delta}
If we have a compact graph with $\delta_p(\epsilon)$-type boundary conditions, there are a only a finite number of $\epsilon$ such that the zero eigenvalue exist. 
\end{theorem} 

\begin{proof}
If we let $\epsilon \to \infty$ we find that we have Dirichlet boundary conditions. Using the same argument as in the proof of Theorem \ref{th:delta} we see that the determinant of $M(\epsilon)$ cannot be identically zero. Thus we have only a finite number of $\epsilon$ such that we have a  non-trivial solution of $M(\epsilon) C=0$, and thus a finite number of $\epsilon$ such that the graph has eigenvalue zero. 

%At a pendant edge we see from Eqn. \eqref{eq:deltapfkn} that the value of the eigenfunction at the terminal vertex is zero and from Eqn. \eqref{eq:deltapder} we see that the normal derivative of the eigenfunction is also zero, thus the eigenfunction is zero on this edge.
\end{proof}

%
%We need some notation in the following. Let us denote the $\delta$-type boundary conditions $\delta_n(\alpha)$, where we specify the $\alpha$ and where the subscript denotes the valency of the vertices where these boundary conditions apply. If we have a set of vertices having $\delta$-type boundary conditions we write $\delta_{n_1,n_2,n_3, ...}(\alpha_1,\alpha_2,\alpha_3, ...)$, where $n_i$ corresponds to all vertices having valence $n_i$ and $\alpha_i$ is the particular parameter used for these vertices. Vertices not listed are assumed to have Neumann boundary conditions, unless otherwise noted. If all vertices have the same boundary condition we write $\delta(\alpha)$, sometimes dropping the $\alpha$. We also use the convention that vertices with valence two always have Neumann boundary conditions, and can thus be removed from a graph. For vertices with valence one, the situation is different and we will explicitly specify if the boundary condition is Neumann or Dirichlet. The notation can be generalised to explicitly enumerate the vertices where the boundary conditions apply, $\delta_{v_1, v_2, ...}(\alpha_1, \alpha_2, ...)$, but this situation does not occur in this paper. Our investigations pertain to the situation where all vertices with the same valency have the same boundary conditions. This is a severe limitation, but has been hard enough to implement in software.

\subsection{Systematic search}

We first investigated isospectral pairs having Dirichlet boundary conditions at terminal vertices and standard boundary conditions otherwise. The first isospectral pair with Dirichlet boundary conditions at terminal vertices occurs at seven vertices, and we show this pair in Fig. \ref{fig:fdirichletall7}a). There are five isospectral pairs with Dirichlet boundary conditions at terminal vertices having eight vertices, shown in Fig. \ref{fig:fdirichletall7}b-f). Two of those pairs are also isospectral if they have Neumann boundary conditions at all vertices. This is easy to understand. Fig. \ref{fig:fdirichletall7}g) shows a graph with two vertices having the same $M$-function. These two vertices have the same $M$-function as each other also under $\delta_3(\alpha)$-type boundary conditions, (although the $M$-function depends on $\alpha$). Attaching any graph, including graphs having any boundary conditions giving a self-adjoint Laplacian, to either of the two vertices will result in isospectral graphs. The two graphs in Fig. \ref{fig:fdirichletall7}e) do indeed originate from the graph in g) by attaching an interval at the relevant vertices. They are thus isospectral independently of the boundary conditions of this attached interval. Two of the isospectral pairs in Fig. \ref{fig:fdirichletall7} are isospectral also under $\delta$-type boundary conditions, specified in the figure caption.

\begin{figure}
\centering
\includegraphics[width=1.0\textwidth]{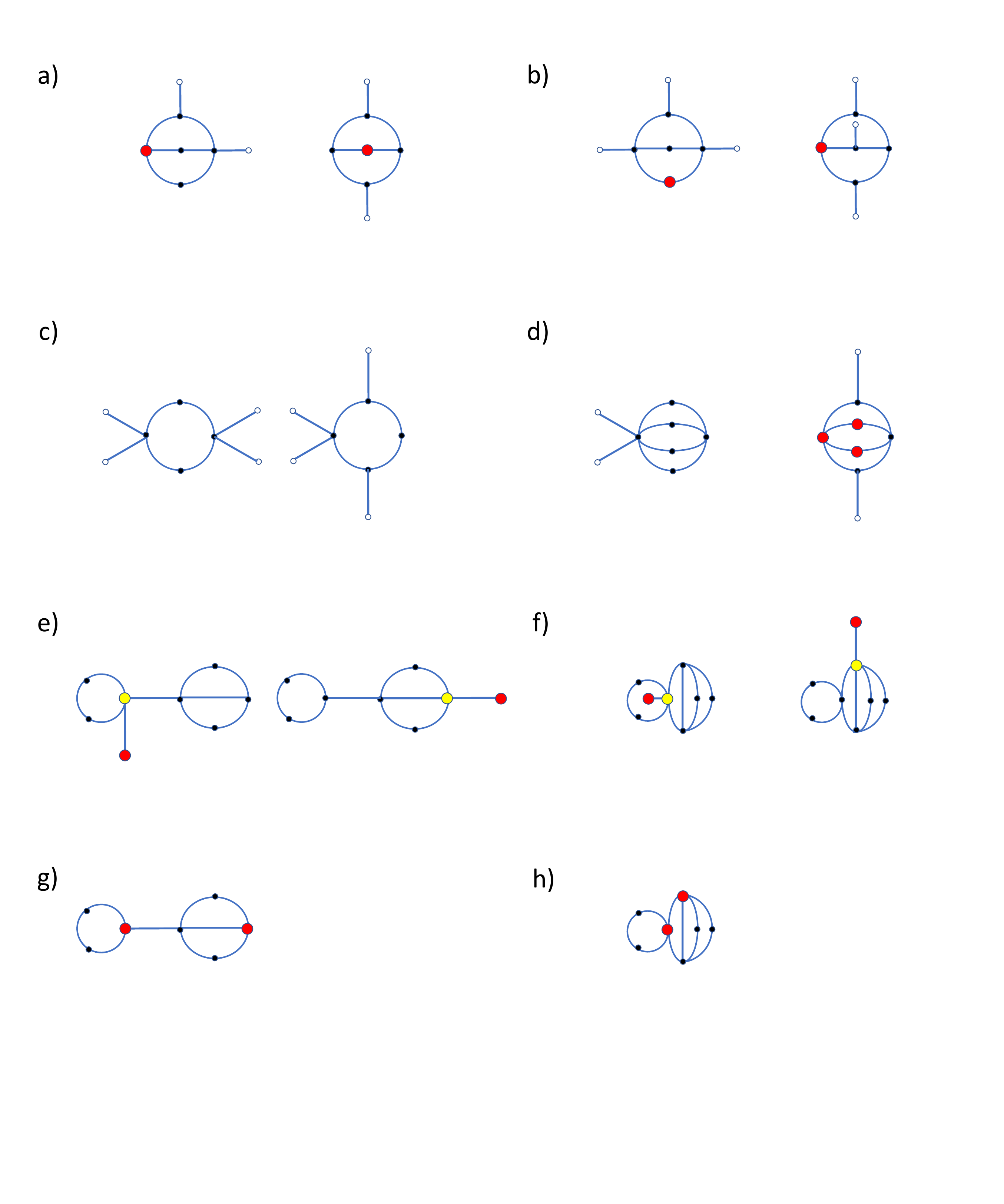}
\caption{Isospectral pairs with Dirichlet boundary conditions at terminal vertices (open circles). Red or yellow vertices have the same $M$-function within each isospectral pair. a) The lone isospectral pair with seven vertices. b-f) The five isospectral pairs with eight vertices. e) Isospectral pairs if the boundary conditions are $\delta_{3}(\alpha)$ and the terminal vertex have either Neumann or Dirichlet boundary conditions. f) Isospectral graphs if the boundary conditions are $\delta_{5}(\alpha)$ and the terminal vertices have either Neumann or Dirichlet boundary condition. g) A graph with two vertices with the same $M$-function under $\delta_{3}(\alpha)$ boundary conditions. Attaching an interval to one of these vertices produces the isospectral pair in e). h) A graph with two vertices with the same $M$-function under $\delta_{5}(\alpha)$ boundary conditions. Attaching an interval to one of these vertices produces the isospectral pair in f). All edges have length one.}
\label{fig:fdirichletall7}
\end{figure}

%\begin{figure}
%\centering
%\includegraphics[width=1.0\textwidth]{fdelta.pdf}
%\caption{a-b) Isospectral pairs under $\delta$-type boundary conditions. c-d) These pairs are isospectral under $\delta_{3,4,5}(\alpha,\beta,\alpha)$ boundary conditions. e-f) These pairs are isospectral under $\delta_{3,4,5,6}(\alpha,\beta,\gamma,\beta)$ boundary conditions. The notation is explained in the main text. Vertices with the same $M$-function have the same colour within each pair. 
%g-i) These graphs have vertices with equal $M$-functions where the boundary conditions are $\delta$-type and Neumann at pendant edges.}
%\label{fig:fdelta}
%\end{figure}

\begin{figure}
\centering
\includegraphics[width=1.0\textwidth]{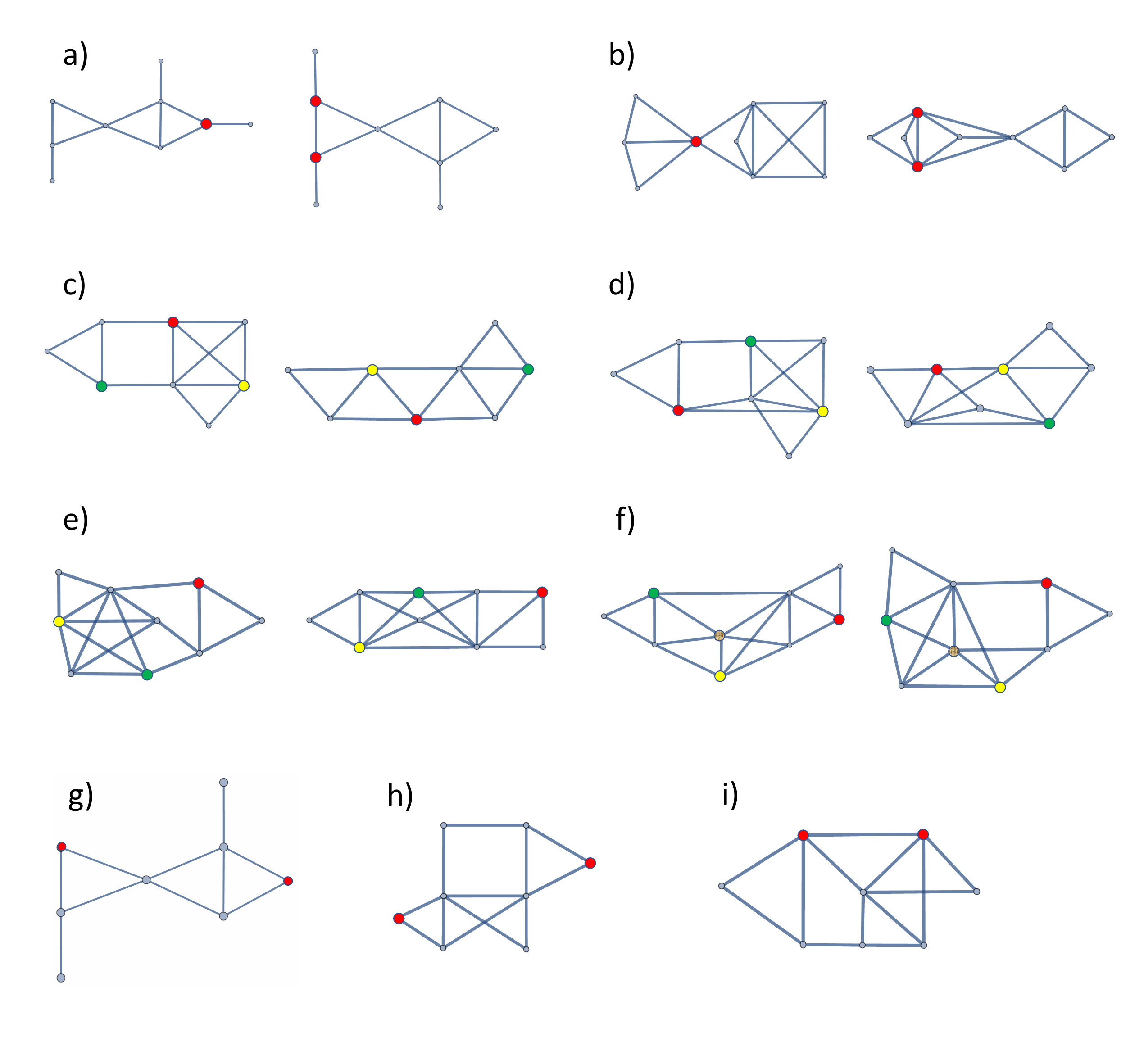}
\caption{a-f) Isospectral pairs under $\delta$-type boundary conditions, given in Table \ref{tab:BoundaryConditions1}. Vertices with the same $M$-function have the same colour within each pair. 
g-i) These graphs have vertices with equal $M$-functions where the boundary conditions are $\delta$-type and Neumann at terminal vertices.}
\label{fig:fdelta}
\end{figure}

\begin{figure}
\centering
\includegraphics[width=1.0\textwidth]{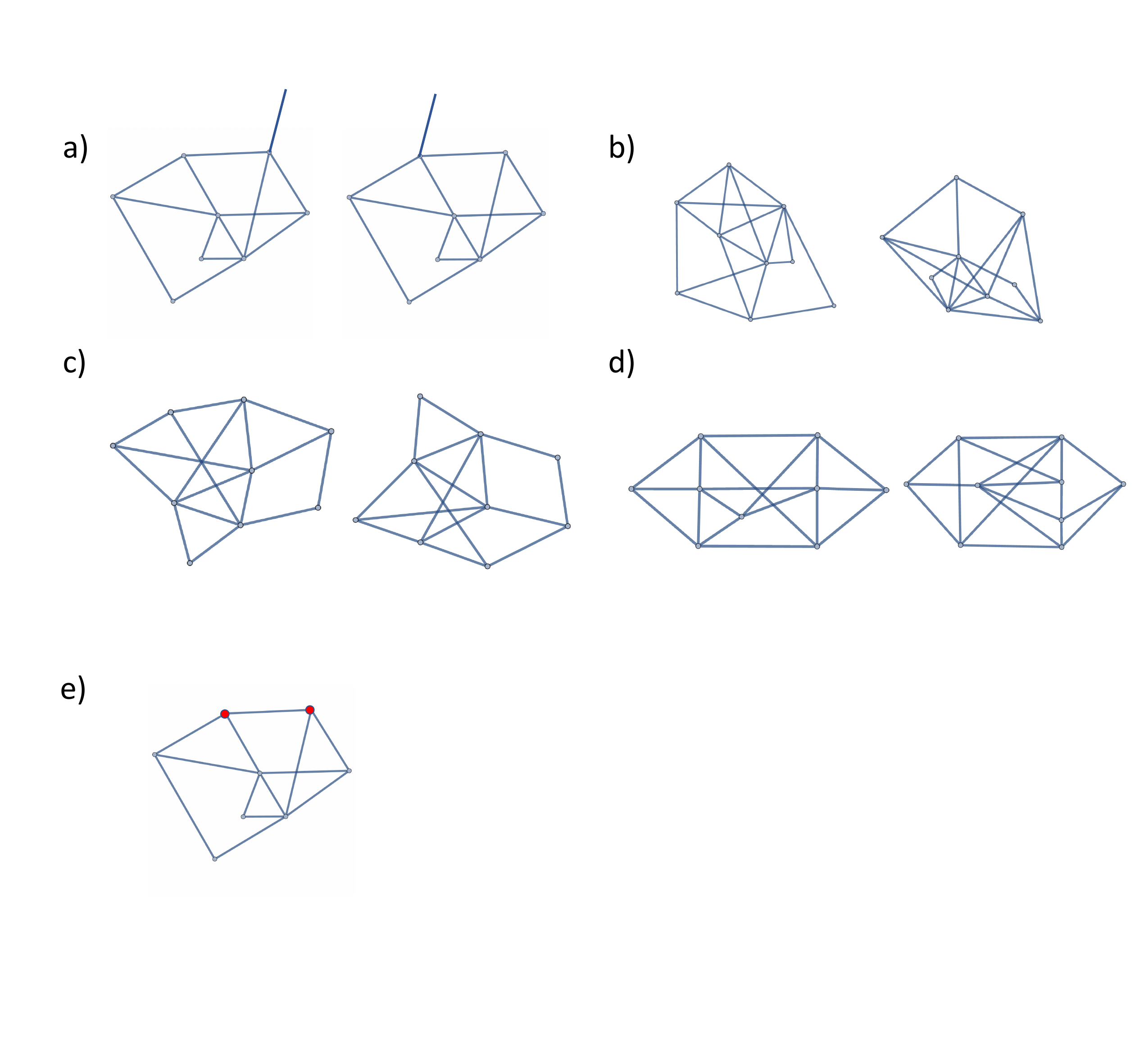}
\caption{a-c) Isospectral pairs under $\delta$-type boundary conditions as well as under $\delta'_s$-type boundary conditions, given in Table \ref{tab:BoundaryConditions1}. These graphs were not tested for vertices with the same $M$-function. e) A graph with two vertices with the same M-function, both under $\delta$-type boundary conditions as well as under $\delta'_s$-type boundary conditions. This graph is the parent graph of the pair in a).}
\label{fig:gDeltaSPrime}
\end{figure}

We then decided to do a systematic search for graphs isospectral under $\delta$-type boundary conditions. Our search finds a set of isospectral pairs under $\delta$-type boundary conditions, shown in Fig. \ref{fig:fdelta} and different sets are given in the Appendix as well as in our software in our GitHub repository \cite{pistol2021graphroots}. The pairs in Fig. \ref{fig:fdelta} a-b) are isospectral if the boundary conditions are $\delta(\alpha)$. Note that $\alpha$ can be any number and the boundary conditions thus include Neumann boundary conditions. These graphs are thus isospectral under a family of boundary conditions. Fig. \ref{fig:fdelta} c-d) shows graphs that are isospectral under $\delta_{3,4,5}(\alpha,\beta,\alpha)$ boundary conditions. Note that vertices with valence three and five need to have the same $\alpha$. Fig. \ref{fig:fdelta} e-f) shows graphs that are isospectral under $\delta_{3,4,5,6}(\alpha,\beta,\gamma,\beta)$ boundary conditions. We find that we can have isospectral graphs where the $\delta$-type boundary conditions can have different number of independent parameters. 

These graphs have vertices with the same $M$-function, indicated in the figure. The $M$-functions of the vertices are equal to each other, as long as the graphs have boundary conditions making them isospectral and this is also the case for the pairs in Fig. \ref{fig:fdirichletall7}e-f) . We here caution the reader. To compute the secular determinant for a single graph takes hours if we have three symbolic parameters, and we cannot do it for four symbolic parameters. Thus we have resorted to numerical experiments to verify that the $M$-functions are equal for all boundary conditions for the graphs in Fig. \ref{fig:fdelta} e-f).

We also found isospectral pairs of graphs under both $\delta$-type boundary conditions as well as under $\delta'_s$-type boundary conditions. These pairs are shown in Fig. \ref{fig:gDeltaSPrime}. There is also a graph which has two vertices with the same $M$-function under both $\delta$-type boundary conditions as well as under $\delta'_s$-type boundary conditions, which is also shown in Fig. \ref{fig:gDeltaSPrime}.

%\begin{table}
%\centering
%\caption{\label{Table 3:all} A summary of the boundary conditions that make various graphs in Fig. \ref{fig:fdelta} isospectral.}
%   \begin{tabular}{lll}
%   &           &             \\ \hline
%    Pendant edges      & Other vertices with valence $>$ 2. & Pairs of graphs     \\ \hline
%    Neumann          & $\delta(\alpha)$      & a-b) \\
%    &  $\delta_{3,4,5}(\alpha,\beta,\alpha)$&  c-d)     \\
%     &  $\delta_{3,4,5,6}(\alpha,\beta,\gamma,\beta)$  & e-f)     \\
%                      \hline
%    \end{tabular}
%\end{table}

\begin{table}
\centering
\caption{\label{Table 3:all} A summary of the boundary conditions that make various graphs in Fig. \ref{fig:fdelta} and \ref{fig:gDeltaSPrime} isospectral.}
    \begin{tabular}{lll}
   &           &             \\ \hline
    Terminal vertices      & Other vertices with valence $>$ 2. & Pairs of graphs   \\ 
       &  & Fig. \ref{fig:fdelta}    \\ \hline
    Neumann          & $\delta(\alpha)$      & a-b) \\
    &  $\delta_{3,4,5}(\alpha,\beta,\alpha)$&  c-d)     \\
     &  $\delta_{3,4,5,6}(\alpha,\beta,\gamma,\beta)$  & e-f)     \\
                      \hline
                          &  &  Fig. \ref{fig:gDeltaSPrime}    \\ \hline
       Neumann, Dirichlet  &  $\delta(\alpha)$, $\delta'_s(\beta)$  & a)     \\
       &  $\delta(\alpha)$, $\delta'_s(\beta)$  & b-c)     \\
                               &  $\delta(\alpha)$, $\delta'_s(\beta)$, $\delta'(\gamma)$, $\delta_p(\epsilon)$  & d)     \\  \hline

    \end{tabular}
    \label{tab:BoundaryConditions1}
\end{table}

It is also the case that we can have individual graphs with two vertices having the same $M$-function as each other even if we change the boundary conditions. This we illustrate in Fig. \ref{fig:fdelta} g-i). 
We can attach graphs to the vertices with the same $M$-function having any boundary conditions and get isospectral graphs, as shown Fig. \ref{fig:ctwohotmirror}, thus generating isospectral graphs that remain isospectral under quite flexible boundary conditions. Even the two attachment vertices themselves can have different boundary conditions as long as the two vertices have the same boundary condition.

\subsection{Trees}

We have also investigated trees with different combinations of Dirichlet, Neumann, standard and $\delta$-type boundary conditions, where we can search all trees having at most thirteen vertices. The first isospectral trees with Dirichlet boundary conditions at terminal vertices and standard otherwise occur at nine vertices and they are shown in Fig. \ref{fig:fdiffboundaryconditionstrees}. These isospectral trees were also found by Pivovarchik \cite{Pivovarchik.Dirichlet.2023} and makes it likely that our program is somewhat correct. Kaliuzhnyi-Verbovetskyi and Pivovarchik have investigated the interplay between Neumann and Dirichlet boundary conditions and shown that caterpillar graphs can be recovered uniquely if their spectra under both Neumann and Dirichlet boundary conditions at certain vertices are known \cite{Kaliuzhnyi-Verbovetskyi}. However these results do not generalise easily. We summarise our results in Table 4. We find it remarkable that the two pairs in Fig. \ref{fig:fdiffboundaryconditionstrees} h-i) are isospectral under any combination of Neumann or Dirichlet at terminal vertices and $\delta(\alpha)$-type or $\delta'_s(\beta)$-type boundary conditions otherwise. These graphs are quite simple.
Fig. \ref{fig:fdeltatrees} shows isospectral trees with a slightly different behaviour. They are isospectral under $\delta_{3,4}(\alpha, 2 \alpha)$-type boundary conditions. That is, the boundary conditions at different vertices are the same but the parameters are not equal but depend on each other. The boundary condition at terminal vertices is Dirichlet. Neumann boundary conditions at terminal vertices destroys the isospectrality.

\begin{figure}
\centering
\includegraphics[width=1.0\textwidth]{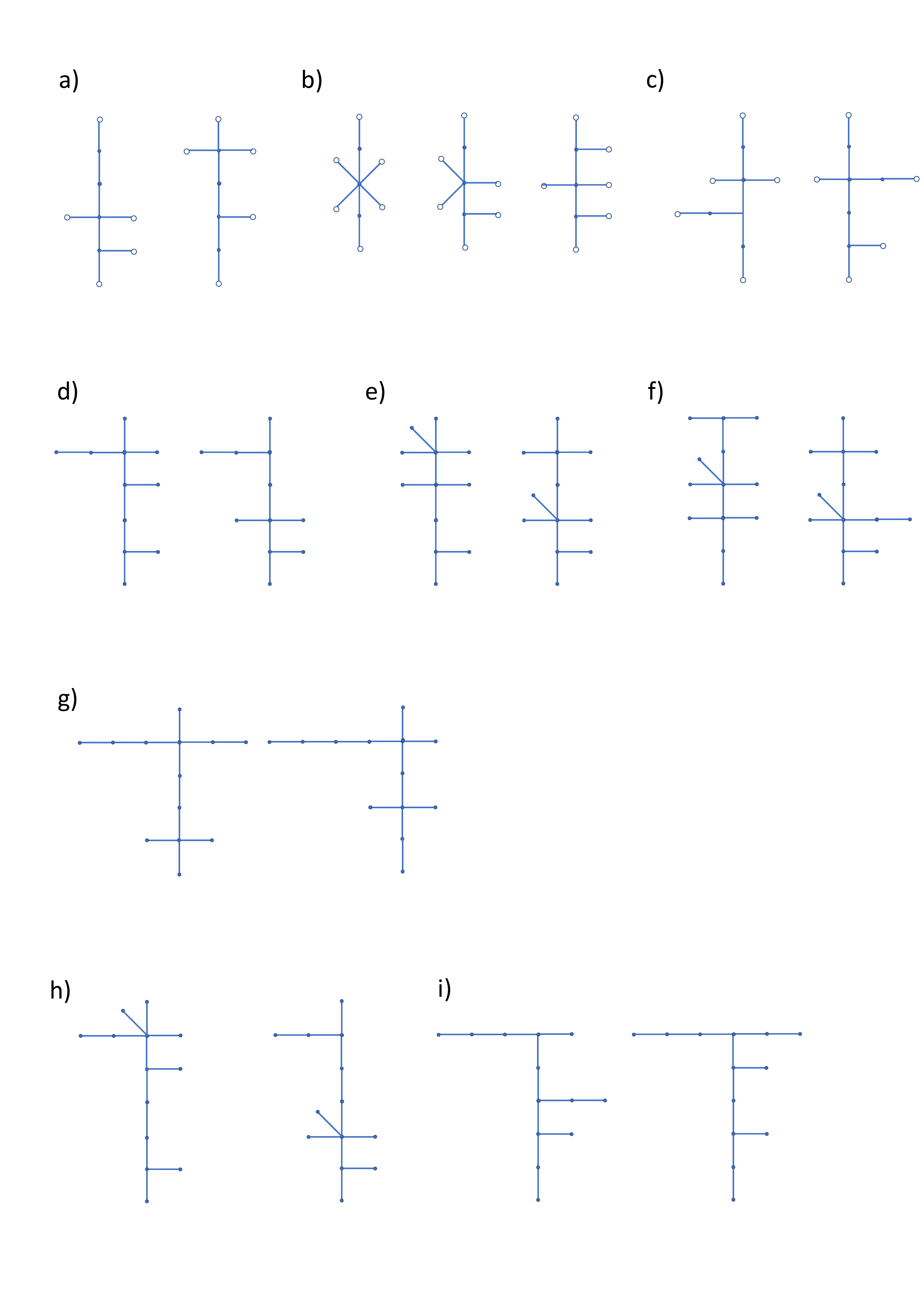}
\caption{Isospectral trees under different boundary conditions, given in Table \ref{tab:isospectraltreesboundaryconditions}.}
% a-c) are isospectral under Dirichlet boundary conditions at pendant edges and Neumann otherwise. d-f) are isospectral under Neumann boundary conditions at pendant edges and $\delta$-type boundary conditions otherwise. g) are isospectral under either Dirchlet or Neumann boundary conditions at pendant edges and Neumann otherwise.
%h-i) are isospectral under either Neumann or Dirichlet boundary conditions at pendant edges and either $\delta$-type or $\delta'_s$-type boundary conditions otherwise. Vertices with valence two always have Neumann boundary conditions.}
\label{fig:fdiffboundaryconditionstrees}
\end{figure}

%\begin{table}[]
%\centering
%\caption{\label{Table 3:all} A summary of the boundary conditions that make various graphs in Fig. \ref{fig:fdiffboundaryconditionstrees} isospectral.}
%    \begin{tabular}{lll}
%   &            &             \\ \hline
%    Pendant edges      & Other vertices with valence $>$ 2. & Sets of graphs     \\ \hline
%    Dirichlet          & Neumann        & a-c) \\
%    Neumann & $\delta(\alpha)$        & d-f)    \\
%    Neumann, Dirichlet            & Neumann      & g)     \\
%    Neumann, Dirichlet & $\delta(\alpha)$, $\delta'_s(\beta)$       & h-i)         \\ \hline
%    \end{tabular}
%\end{table}

\begin{table}[]
\centering
\caption{\label{Table 4:all} A summary of the boundary conditions that make various graphs in Figs. \ref{fig:fdiffboundaryconditionstrees} and \ref{fig:fdeltatrees} isospectral.} 
    \begin{tabular}{lll}
   &            &             \\ \hline
    Terminal vertices      & Other vertices with valence $>$ 2. & Sets of graphs, Fig. \ref{fig:fdiffboundaryconditionstrees}      \\ \hline
    Dirichlet          & standard        & a-c) \\
    Neumann & $\delta(\alpha)$        & d-f)    \\
    Neumann, Dirichlet            & standard      & g)     \\
    Neumann, Dirichlet & $\delta(\alpha)$, $\delta'_s(\beta)$       & h-i)         \\ \hline
                           &  &  Fig. \ref{fig:fdeltatrees}    \\ \hline
       Neumann, Dirichlet  &  $\delta_{3,4}(\alpha, 2 \alpha)$  & a-b)    
                                           \\  \hline

    \end{tabular}
    \label{tab:isospectraltreesboundaryconditions}
\end{table}

The behaviour can be more complex as shown in Fig. \ref{fig:fdirichletdisjointneumanntrees}. There are graphs that have one isospectral partner under Dirichlet boundary conditions and $\delta_3(\alpha)$ elsewhere, which have a different isospectral partner under Neumann boundary conditions everywhere. 

\subsection{Discussion}

It is clear from the above examples that we can have graphs that are isospectral to different degrees, meaning under more and more general boundary conditions. It is interesting to find how far this can be pushed. 
However to do so would necessitate a study under all possible boundary conditions, a task we cannot do yet. 

It is also interesting to find what is common to graphs that are isospectral under $\delta$-type boundary conditions. Inspection of the graphs show that each graph in an isospectral pair have the same number of vertices with the same valency. We propose the following conjecture:

\begin{conjecture}
If we have two graphs that are isospectral under \\ 
$\delta_{n_1,n_2, ... , n_k}(\alpha_1,\alpha_2, ... , \alpha_k)$ boundary conditions, then the two graphs have the same number of vertices with valency $n_1$, the same number of vertices with valency $n_2$, and so on.
\end{conjecture}

This conjecture is not true if the boundary conditions are Neumann, see Fig. \ref{fig:aloop}. It seems that increasing flexibility in boundary conditions constrains the isospectral graphs. If so, there might be sufficiently general boundary conditions that a given graph has no isospectral partner. It is of course interesting to know which boundary conditions uniquely specify graphs. Lawniczak et al. have found a generalised Euler characteristic that detects the number of terminal vertices with Dirichlet boundary conditions given the spectrum (if the rest of the boundary conditions are Neumann) \cite{Lawniczak.2021}. 

\begin{figure}
\centering
\includegraphics[width=1.0\textwidth]{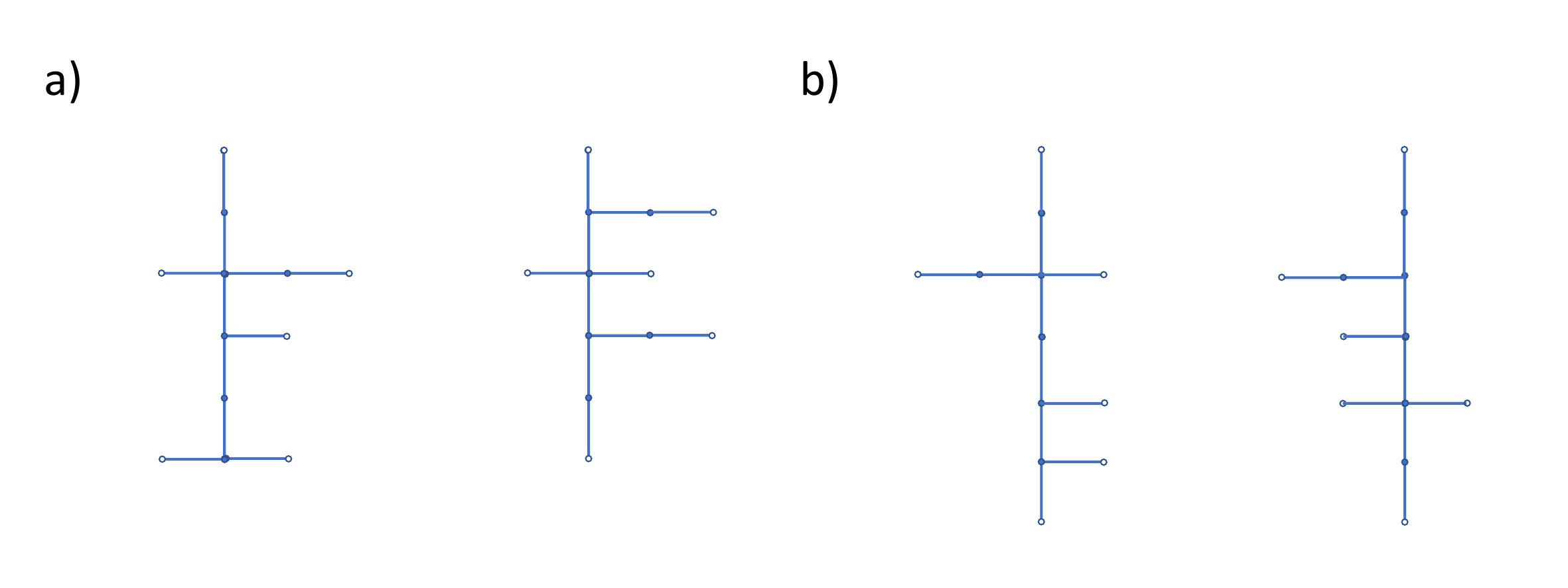}
\caption{Isospectral pairs of trees. These pairs are isospectral under $\delta_{3,4}(\alpha, 2 \alpha)$-type boundary conditions at interior edges and Dirichlet at pendant edges.}
\label{fig:fdeltatrees}
\end{figure}

\begin{figure}
\centering
\includegraphics[width=1.0\textwidth]{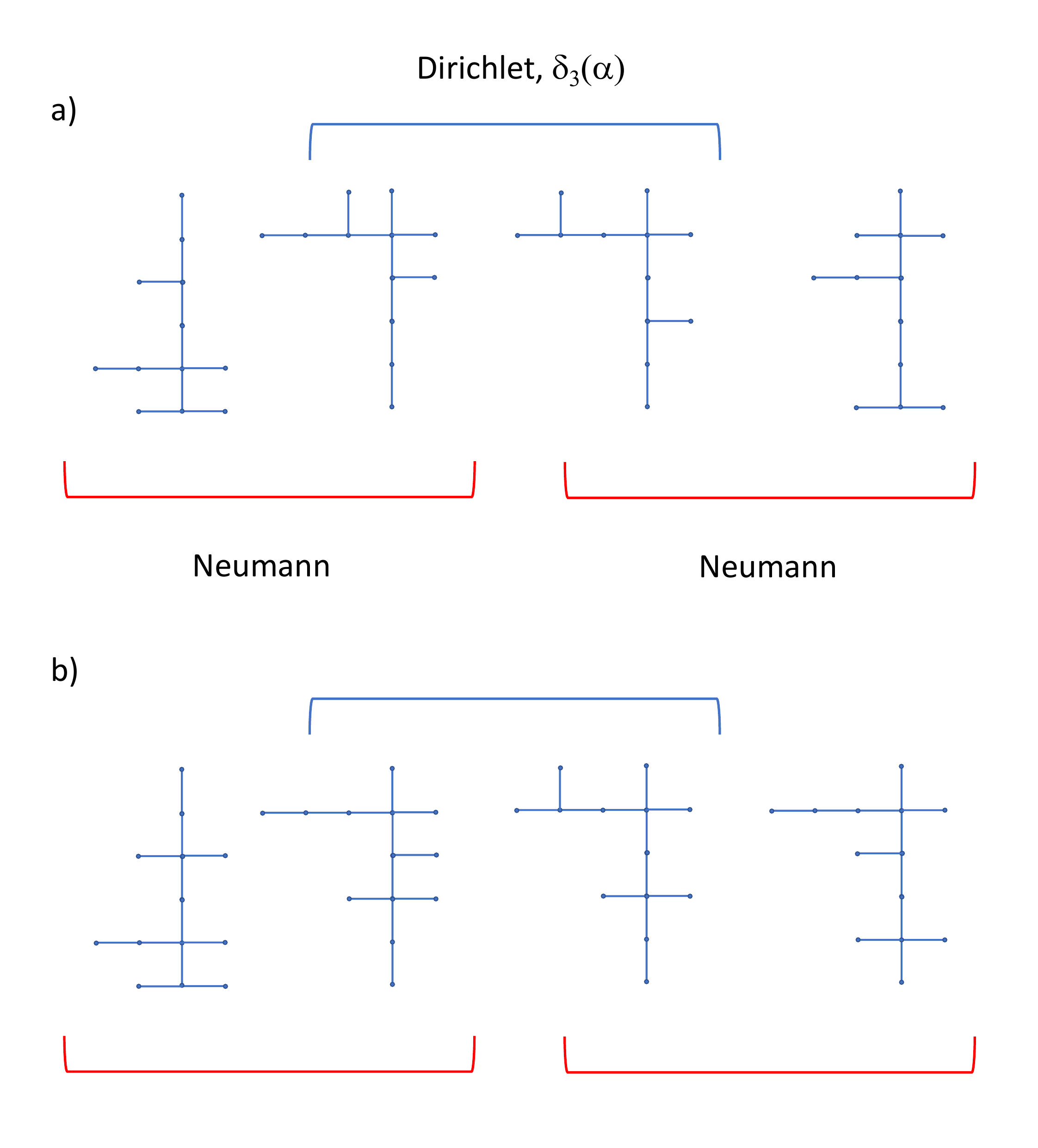}
\caption{Isospectral pairs of trees. Trees grouped by a blue bracket are isospectral under Dirichlet boundary conditions at terminal vertices and $\delta_3(\alpha)$ elsewhere, whereas trees grouped by a red bracket are isospectral under Neumann boundary conditions at terminal vertices, and standard boundary conditions elsewhere. These are all examples of this type for trees with at most 13 vertices.}
\label{fig:fdirichletdisjointneumanntrees}
\end{figure}

\section{Asymmetric boundary conditions and almost isospectral graphs}

We have also started the investigation of isospectral graphs, where one graph has different boundary conditions than the other graph(s). Since each vertex can have its own boundary condition the space of possibilities is very large indeed. Our first examples are shown in Fig. \ref{fig:gmixedisospectral}. The top graph has Dirichlet boundary conditions at terminal vertices and otherwise standard boundary conditions. There are four isospectral partners which have mixed Neumann and Dirichlet boundary conditions at terminal vertices and standard boundary conditions otherwise. This situation is similar to the isospectral pair found by Band et al. \cite{band2009isospectral}. 

Consider connected graphs with standard boundary conditions. If the boundary conditions at terminal vertices are all Neumann, then zero is an eigenvalue. If at least one terminal vertex has Dirichlet boundary conditions then zero is not an eigenvalue. See Ref. \cite{kurasov2005inverse} for a proof that the eigenfunction corresponding to the zero eigenvalue is a constant if the graph is connected. It is often the case that two graphs have the same eigenvalues, apart from eigenvalue zero where they might have different multiplicities or maybe only one of the graphs have eigenvalue zero. We say that graphs which have the same positive eigenvalues with the correct multiplicities to be \textit{almost isospectral}. Examples of almost isospectral graphs are found in Fig. \ref{fig:gmixedbctrees}. We here find trees having Neumann boundary conditions and standard boundary conditions that are almost isospectral to trees having both Neumann and Dirichlet boundary conditions at terminal vertices and standard boundary conditions otherwise. In the figure we also show a tree that is almost isospectral to a double loop.

\begin{figure}
\centering
\includegraphics[width=1.0\textwidth]{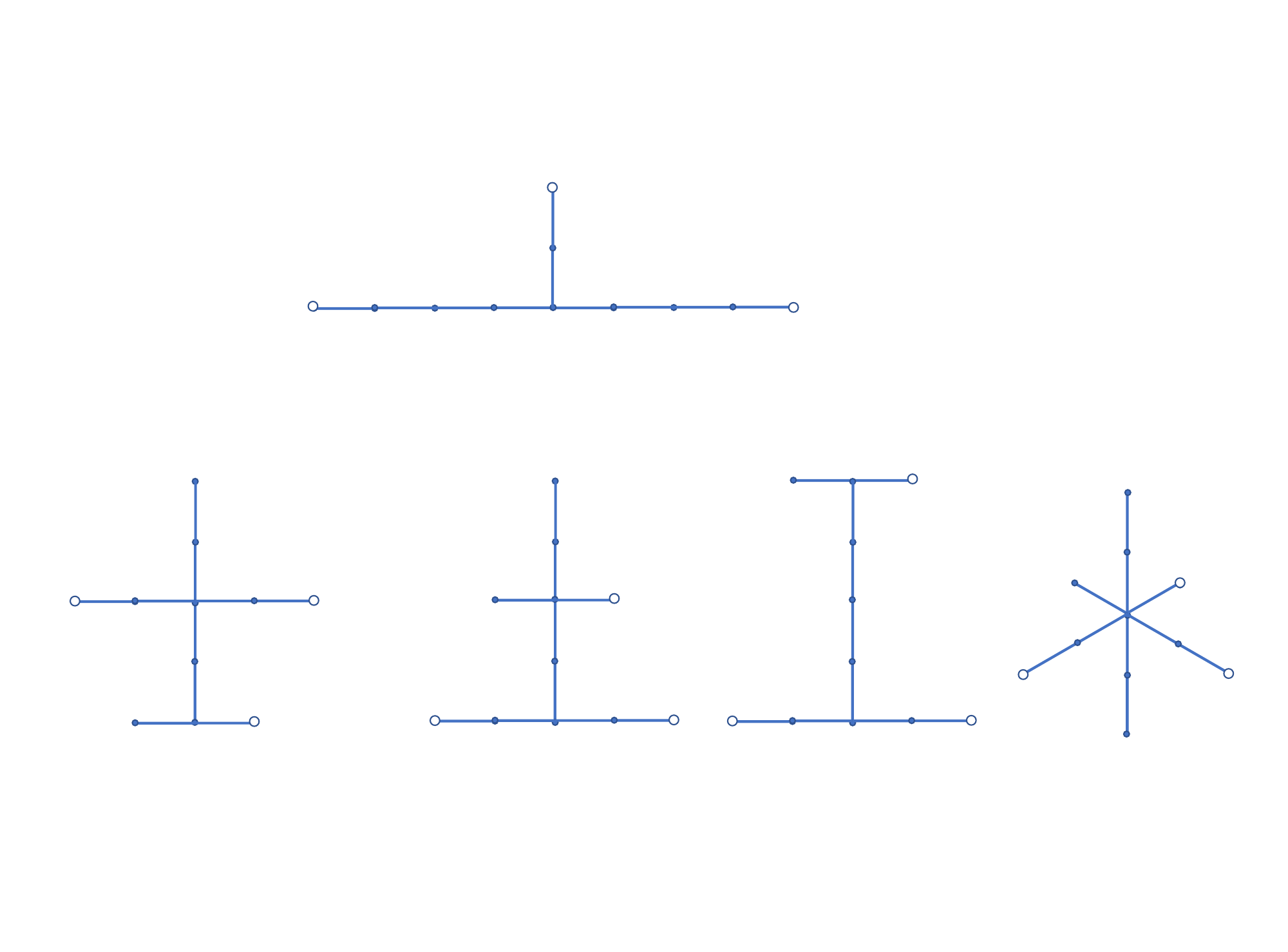}
\caption{Isospectral trees. The top tree has Dirichlet boundary conditions at terminal vertices (open circles). It has four isospectral partners which have both Neumann and Dirichlet boundary conditions at terminal vertices. Vertices with valence three and four have standard boundary conditions.}
\label{fig:gmixedisospectral}
\end{figure}

\begin{figure}
\centering
\includegraphics[width=1.0\textwidth]{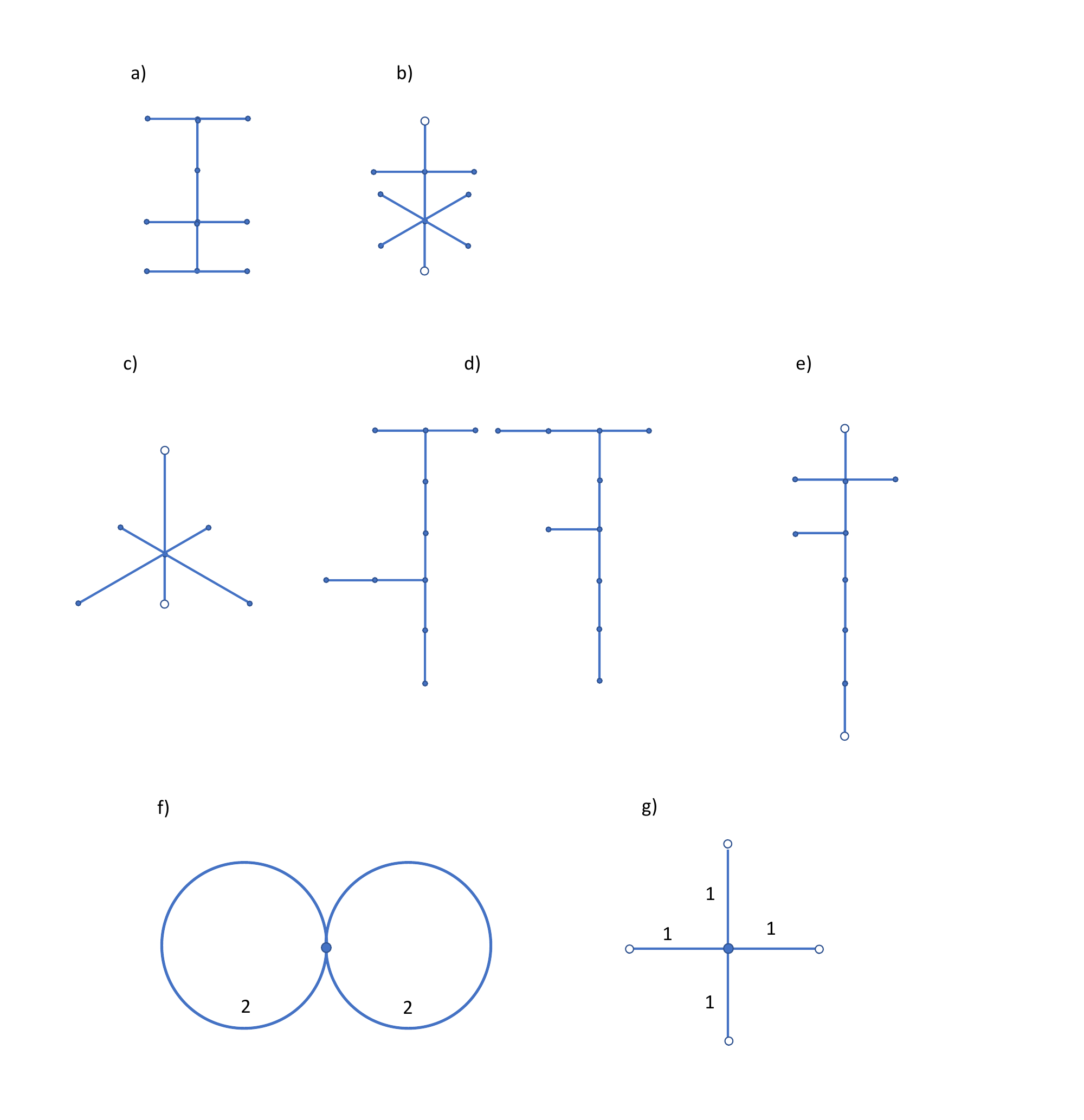}
\caption{Almost isospectral sets of graphs. a) and b) are almost isospectral. d) is an isospectral pair which have two almost isospectral partners c) and e). A double loop, f), has an almost isospectral partner which is a star graph, g). Open circles indicate vertices with Dirichlet boundary conditions, closed circles indicate vertices with standard boundary conditions. These almost isospectral graphs are not isospectral since they differ in eigenvalue zero.}
\label{fig:gmixedbctrees}
\end{figure}

\begin{figure}
\centering
\includegraphics[width=1.0\textwidth]{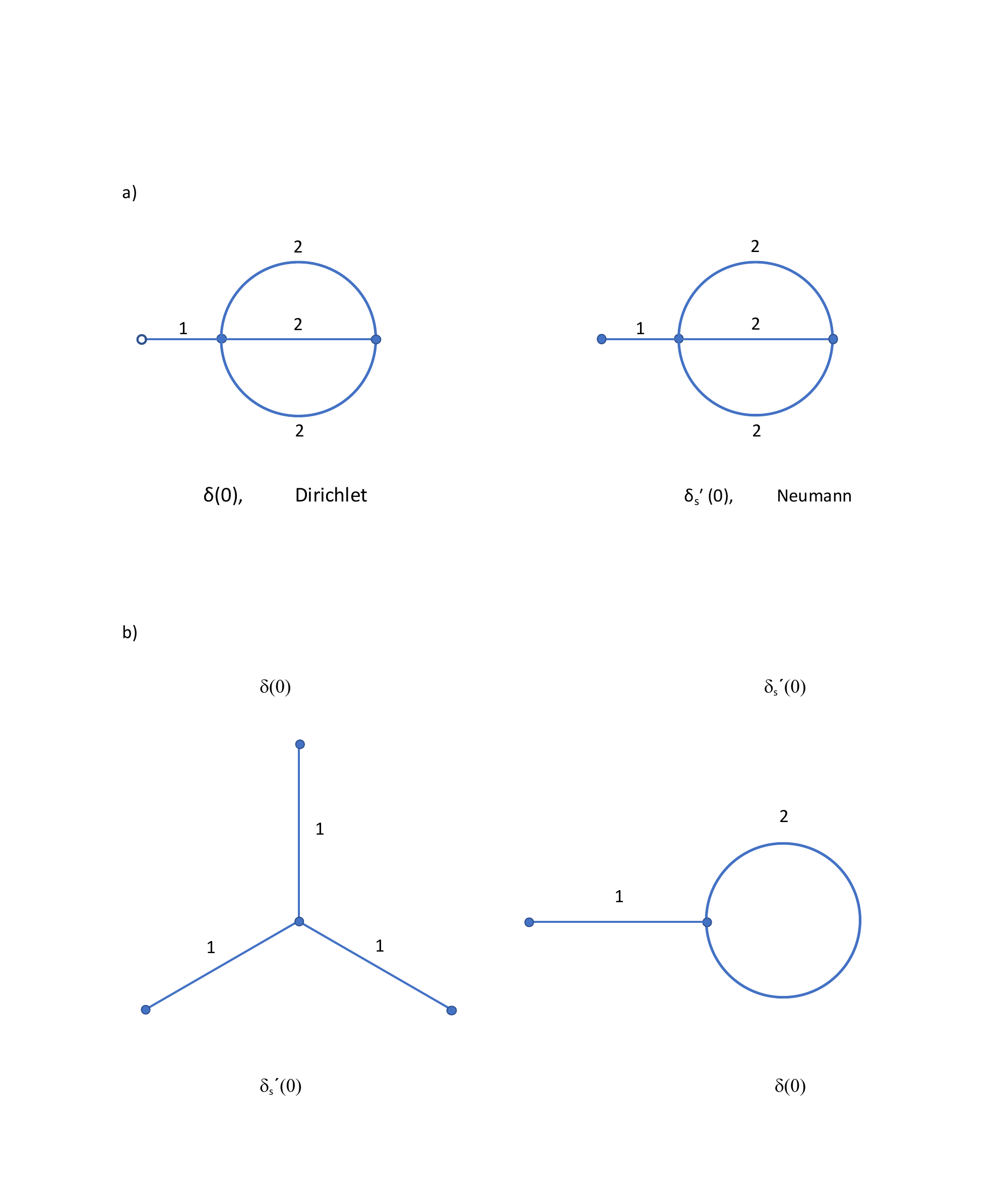}
\caption{a) These two graphs with boundary conditions are almost isospectral. They do differ in eigenvalue zero.
b) One graph can have $\delta(0)$ boundary conditions and the other $\delta_s'(0)$ boundary conditions or the other way around and in both cases they are almost isospectral. We have not investigated eigenvalue zero for this pair. Terminal vertices have Dirichlet boundary conditions if indicated by open circles, otherwise Neumann boundary conditions.}
\label{fig:gspectralduals}
\end{figure}

\section{Different types of isospectrality}

Let us define a new object - a metric graph and its boundary conditions - and denote such an object $\{h, B\}$ and call it a (metric) \textit{graph with boundary conditions}. Here $h$ is the metric graph and B is the boundary conditions for this graph. Usually an isospectral pair is a pair $\{h_1, B\}$ and $\{h_2, B\}$ which are isospectral and where $h_1$ is not isomorphic to $h_2$. That is, the boundary conditions remain fixed. It is however possible to have an (almost) isospectral pair $\{h, B_1\}$ and $\{h, B_2\}$, where the graph remains fixed but the boundary conditions are different. We can also have the situation that $\{h_1, B_1\}$ and $\{h_2, B_2\}$ are (almost) isospectral and $\{h_2, B_1\}$ and $\{h_1, B_2\}$ are (almost) isospectral. In Fig. \ref{fig:gspectralduals} we show examples of both situations.

\section{Outlook and Conclusion}

We have generated all isospectral pairs within a restricted set of quantum graphs, where the length of the edges are all equal and where the number of vertices is limited. We have found several isospectral pairs with different character, such as having different Euler characteristics, having different number of vertices and isospectral pairs that are quite simple. For future work we would like to graphically illustrate the eigenfunctions which might illuminate the relationship between nodal points and isospectrality \cite{band2006nodal} as well as to extend our studies to larger sets of graphs. 

Many of our discoveries have been through inspection of isospectral graphs, and doing experiments with them. This includes the few examples of graphs that are not equilateral that can be used to decorate compact graphs to generate isospectral graphs. We are sure that there are more discoveries to be made by doing more experiments with the graphs and we hope the community will do experiments using our software.

There are many questions left. How to characterize the set of graphs without any isospectral partner is interesting. The interval does not have an isospectral partner so this set is not empty \cite{nicaise1987spectre}, (see \cite{davies2013inverse} for an abstract treatment). The star graphs with $k$ leaves of equal length minimise the $k^{th}$ eigenvalue \cite{friedlander2005extremal} and no other graph does, so they do not have any isospectral partners.

It would be interesting to increase the set of boundary conditions and to find isospectral graphs under even more general boundary conditions than those that we have investigated. The same is the case for $M$-functions. Suppose we have two graphs having vertices $u_i$ and $v_i$, respectively with a one-to-one mapping between $u_i$ and $v_i$. Are there graphs that are isospectral where each pair of vertices , $u_i$ and $v_i$, have $\delta(\alpha_i)$-type boundary conditions and where all $\alpha_i$ can be chosen independently?

We hope that this study will inspire work to classify quantum graphs with respect to isospectrality as well as to study their $M$-functions. 
We also hope that our programs will be subject to strong tests by the community, such that potential errors are found. If there are severe errors in the software, then many results in this paper are incorrect. We have done a lot to make sure the software is correct but there will always be doubts.
Nevertheless, our programs have allowed us to find interesting isospectral graphs which we could prove to be isospectral and it is more difficult to find isospectral graphs than to verify their isospectrality.
 
 \textbf{Limitations}: Our software does not handle graphs which are not equilateral well. It can do it, but the computing time is often excessive. In many cases it cannot give explicit eigenvalues and it cannot give explicit $M$-functions. However numerical eigenvalues and $M$-functions can often be found in limited ranges of the real axis. This can be done with userdefined accuracy.
 
\textbf{Reproducibility Statement}: 
We open-source our code as well as notebooks that allows the reproduction of all our figures. The corresponding graphs are given as Mathematica graph objects allowing for easy manipulation and testing. Our tests of the software is given along with reference results from the literature. We also give notebooks that allows fairly easy testing of all of our isospectrality results for different boundary conditions. In the repository there is software that allows searching for isospectral pairs, given any database of graphs.
 
\clearpage
\newpage
\section{Acknowledgement}
I thank Prof. Pavel Kurasov who introduced me to quantum graphs, encouraged me to do some programming and discussed isospectrality with me. I thank Dr. Gabriela Malenov{\'a} who helped me find quantum graphs with known eigenfrequencies as well as Prof. Ram Band who encouraged me to find a constructive way to generate isospectral graphs, gave me pointers to the literature and who also discussed isospectrality with me. I am grateful to Prof. Vyacheslav Pivovarchik who introduced me to characteristic polynomials which I found very helpful. 

Part of this work was performed in Odesa, Ukraine, and I thank its people for their hospitality also under difficult times.
\clearpage
\newpage
\bibliographystyle{unsrt}
\bibliography{isospectral.16}

\newpage
\section{Appendix}

\renewcommand{\thefigure}{A\arabic{figure}}
\setcounter{figure}{0}

Figs. \ref{fig:A1} and \ref{fig:A2} shows the isospectral pairs we found involving a loop decorated with pendant edges or trees. Fig. \ref{fig:A3} shows all trees we found with att most 12 vertices and Fig. \ref{fig:A4} shows the isospectral triplets with nine edges. Some of the trees are quite simple. \\
We calculated the vertex and edge scattering matrices for the graphs in Figs. \ref{fig:alooptriplets1}c) and \ref{fig:alooptriplets1}d) by hand. For the graph in Fig. \ref{fig:alooptriplets1}c) we get:

$$ S_v =
\left(
\begin{array}{cccc}
 0 & -\frac{1}{3} & \frac{2}{3} & \frac{2}{3} \\
 1 & 0 & 0 & 0 \\
 0 & \frac{2}{3} & \frac{2}{3} & -\frac{1}{3} \\
 0 & \frac{2}{3} & -\frac{1}{3} & \frac{2}{3} \\
\end{array}
\right) $$

and 

$$ S_e =
\left(
\begin{array}{cccc}
 e^{i k L_1} & 0 & 0 & 0 \\
 0 & e^{i k L_1} & 0 & 0 \\
 0 & 0 & e^{i k L_2} & 0 \\
 0 & 0 & 0 & e^{i k L_2} \\
\end{array}
\right)$$
where $L_1 = 1$ and $L_2 = 4$. The edges were ordered - the pendant edge, the loop.

For the graph in Fig. \ref{fig:alooptriplets1}d) we get:

$$ S_v =
\left(
\begin{array}{cccccccccc}
 0 & -\frac{1}{3} & \frac{2}{3} & 0 & \frac{2}{3} & 0 & 0 & 0 & 0 & 0 \\
 1 & 0 & 0 & 0 & 0 & 0 & 0 & 0 & 0 & 0 \\
 0 & 0 & 0 & -\frac{1}{2} & 0 & \frac{1}{2} & \frac{1}{2} & 0 & \frac{1}{2} & 0 \\
 0 & \frac{2}{3} & -\frac{1}{3} & 0 & \frac{2}{3} & 0 & 0 & 0 & 0 & 0 \\
 0 & 0 & 0 & \frac{1}{2} & 0 & -\frac{1}{2} & \frac{1}{2} & 0 & \frac{1}{2} & 0 \\
 0 & \frac{2}{3} & \frac{2}{3} & 0 & -\frac{1}{3} & 0 & 0 & 0 & 0 & 0 \\
 0 & 0 & 0 & 0 & 0 & 0 & 0 & 1 & 0 & 0 \\
 0 & 0 & 0 & \frac{1}{2} & 0 & \frac{1}{2} & -\frac{1}{2} & 0 & \frac{1}{2} & 0 \\
 0 & 0 & 0 & 0 & 0 & 0 & 0 & 0 & 0 & 1 \\
 0 & 0 & 0 & \frac{1}{2} & 0 & \frac{1}{2} & \frac{1}{2} & 0 & -\frac{1}{2} & 0 \\
\end{array}
\right) $$

and

$$ S_e =
\left(
\begin{array}{cccccccccc}
 e^{i k L_1} & 0 & 0 & 0 & 0 & 0 & 0 & 0 & 0 & 0 \\
 0 & e^{i k L_1} & 0 & 0 & 0 & 0 & 0 & 0 & 0 & 0 \\
 0 & 0 & e^{i k L_1} & 0 & 0 & 0 & 0 & 0 & 0 & 0 \\
 0 & 0 & 0 & e^{i k L_1} & 0 & 0 & 0 & 0 & 0 & 0 \\
 0 & 0 & 0 & 0 & e^{i k L_1} & 0 & 0 & 0 & 0 & 0 \\
 0 & 0 & 0 & 0 & 0 & e^{i k L_1} & 0 & 0 & 0 & 0 \\
 0 & 0 & 0 & 0 & 0 & 0 & e^{i k L_1} & 0 & 0 & 0 \\
 0 & 0 & 0 & 0 & 0 & 0 & 0 & e^{i k L_1} & 0 & 0 \\
 0 & 0 & 0 & 0 & 0 & 0 & 0 & 0 & e^{i k L_1} & 0 \\
 0 & 0 & 0 & 0 & 0 & 0 & 0 & 0 & 0 & e^{i k L_1} \\
\end{array}
\right) $$

where $L_1 = 1$. The edges were ordered - left pendant edge, one edge of the loop, the second edge of the loop, one terminal vertex to the right, the other terminal vertexto the right. We find that the secular equation is the same in both cases: $\Sigma (k)=\frac{1}{3} \left(e^{2 i k} - 1\right)^2 \left(7 e^{2 i k}+7 e^{4 i k}+3 e^{6 i k}+3\right)$ and we thus confirm that the graphs are isospectral.

In Fig. \ref{fig:A5} we show isospectral pairs and sets which were derived from Fig. \ref{fig:dhotvertices} by attaching loops to the relevant hot vertices. A variety of pleasing isospectral graphs can be formed. In Fig. \ref{fig:A6} we show examples of graphs which are isospectral under $\delta_{3,4,5,6}(\alpha,\beta,\alpha,\beta)$ boundary conditions.

\begin{figure}[ht]
\centering
\includegraphics[width=1.0\textwidth]{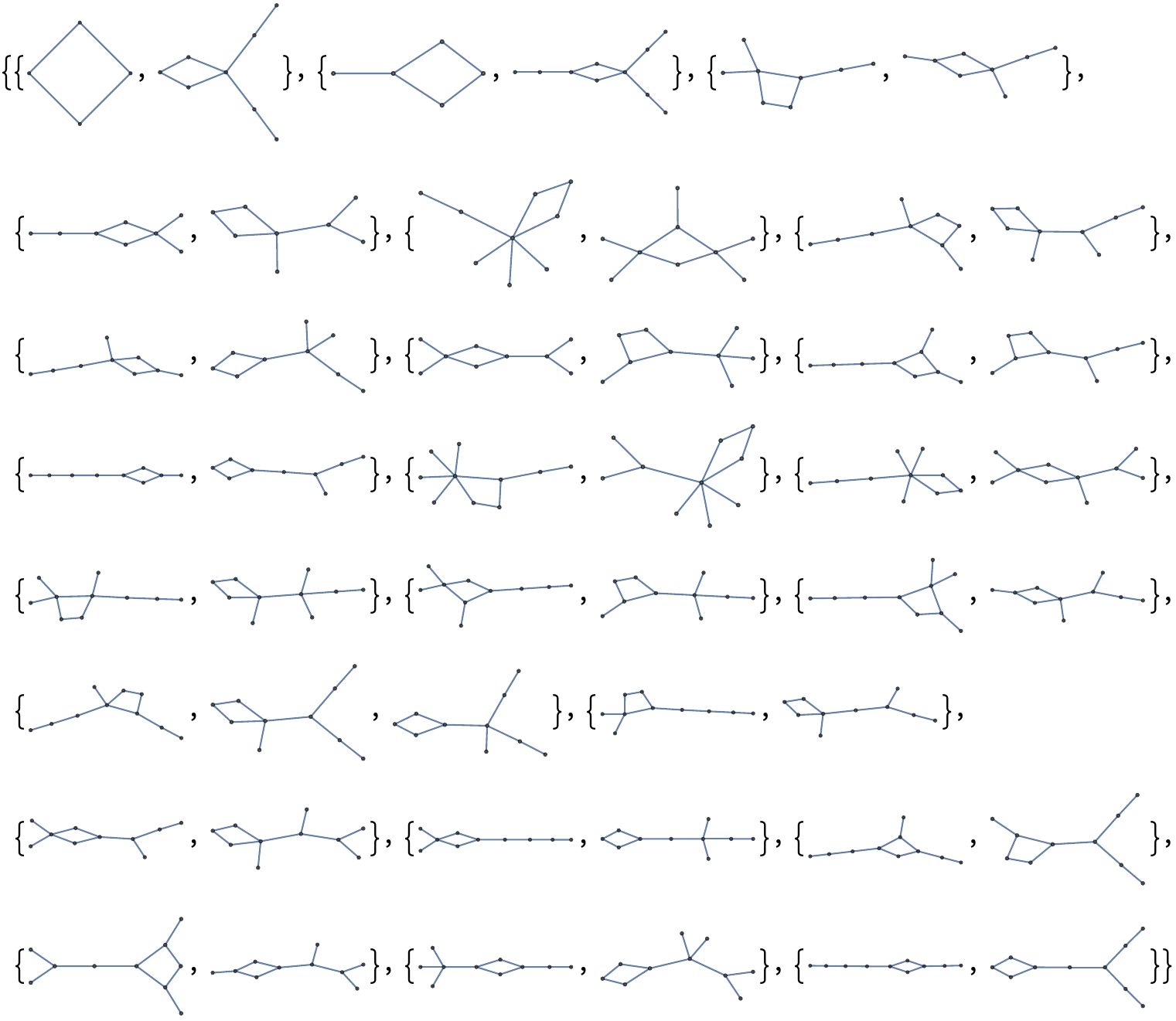}
\caption{All 23 isospectral sets formed from a loop with four vertices, decorated with pendant edges or pendant trees where the full graph has at most ten vertices. All edges have the same length. Some of these isospectral pairs have been used in the figures in the main text, with vertices of valence two removed. Note that the graphs in each set have to be normalised to the same length in order to be isospectral. These graphs were not tested for hot vertices.}
\label{fig:A1}
\end{figure}

\begin{figure}
\centering
\includegraphics[width=1.0\textwidth]{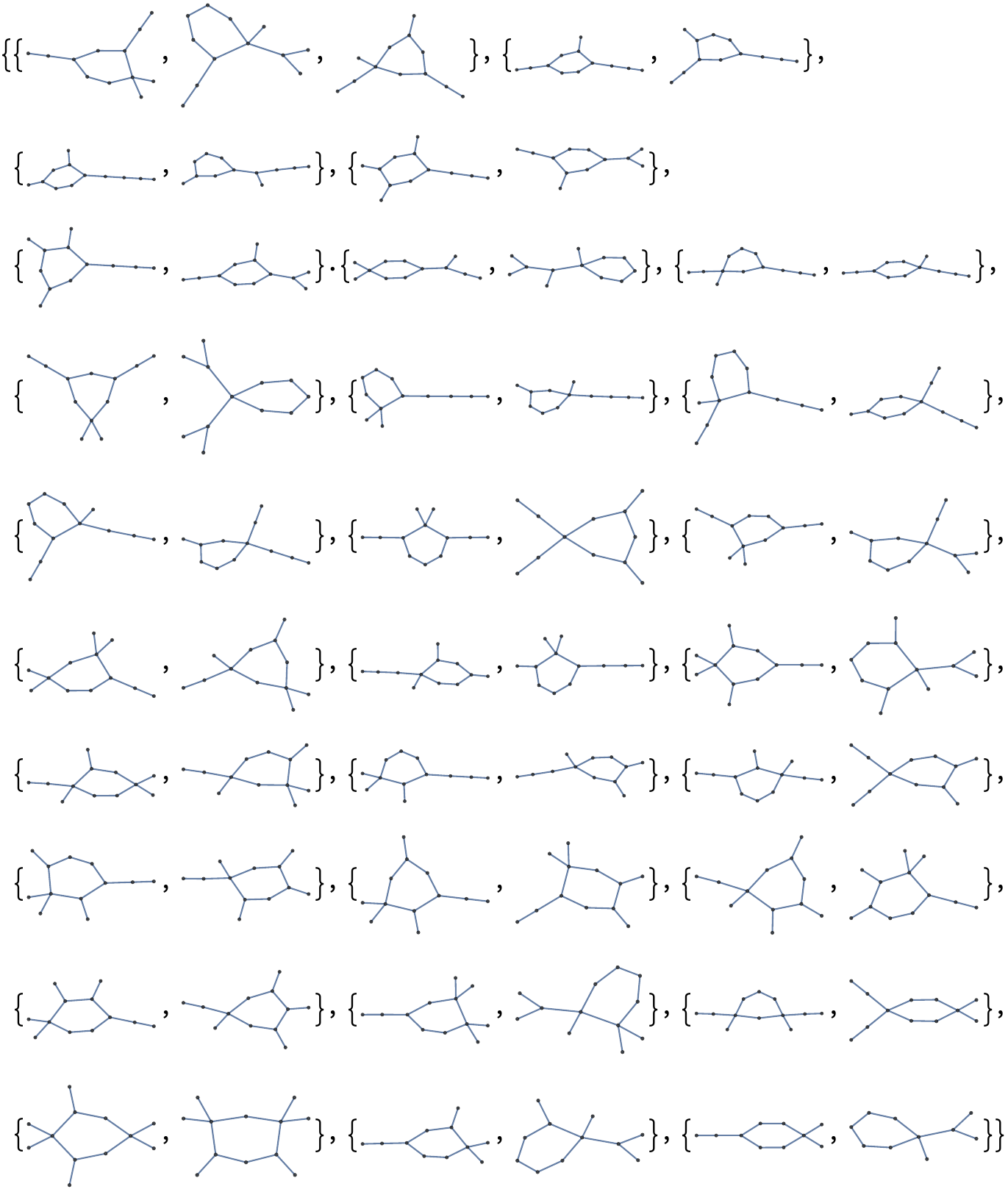}
\caption{All 28 isospectral sets formed from a loop with six vertices, decorated with pendant edges or pendant trees where the full graph has at most twelve vertices. All edges have the same length. Some of these isospectral pairs have been used in the figures in the main text, with vertices of valence two removed. One more isospectral pair was found but which is isomorphic with the first pair in Fig. \ref{fig:alooptriplets3} after rescaling to the same length. Note that the graphs in each set have to be normalised to the same length in order to be isospectral. These graphs were not tested for hot vertices.}
\label{fig:A2}
\end{figure}

\begin{figure}[!h]
\centering
\begin{center}
\includegraphics[scale=0.7,center]{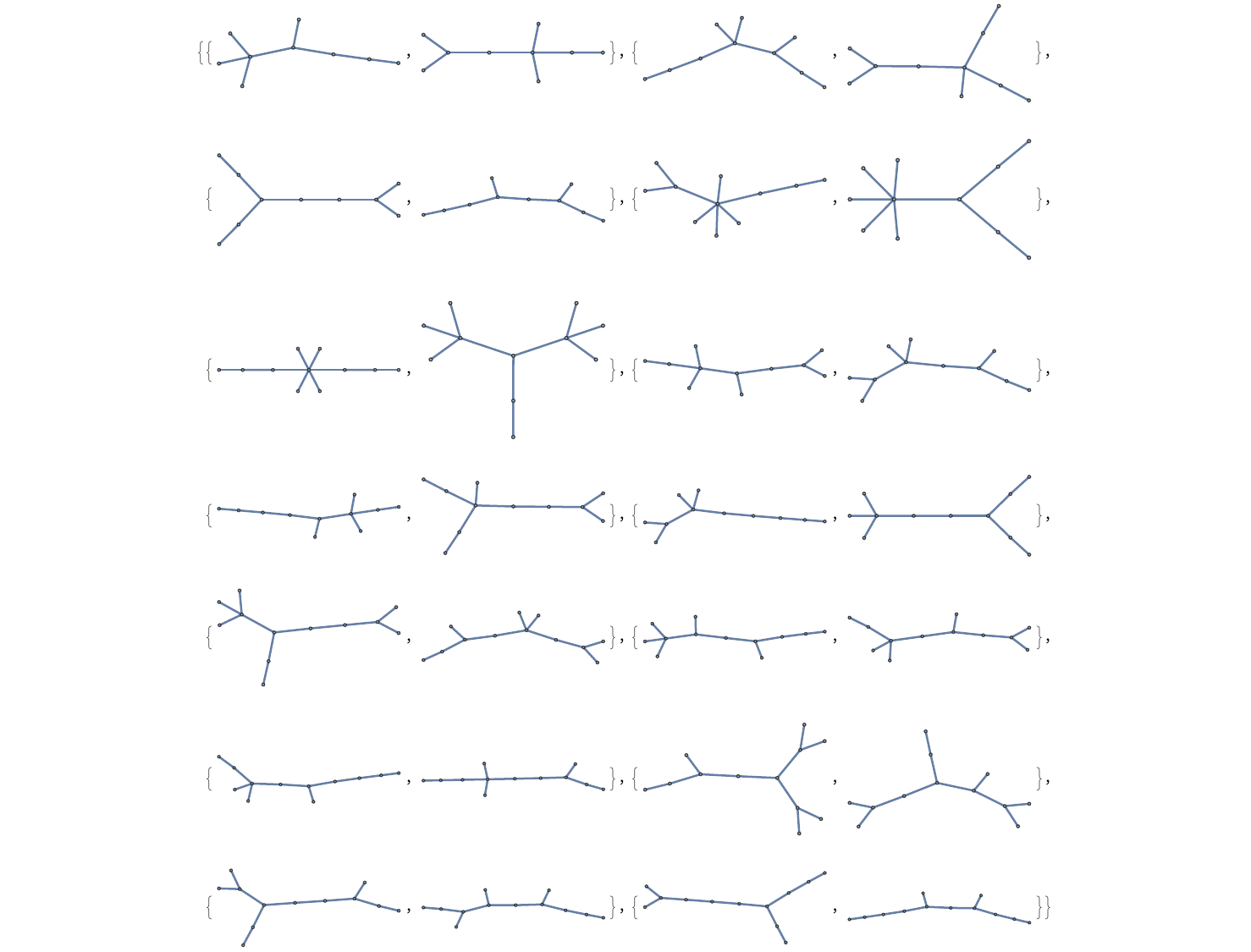}
\end{center}
\caption{ All isospectral pairs of trees with at most 12 vertices. All edges have the same length. The three first pairs have been used in figures in the main text, plotted differently. There is one isospectral pair with nine vertices, two isospectral pairs with 10 vertices, five isospectral pairs with 11 vertices and six isospectral pairs with 12 vertices.These graphs were tested for hot vertices within each isospectral pair and none was found.}
\label{fig:A3}
\end{figure}

\begin{figure}[!h]
\centering
\begin{center}
\includegraphics[scale=0.30,center]{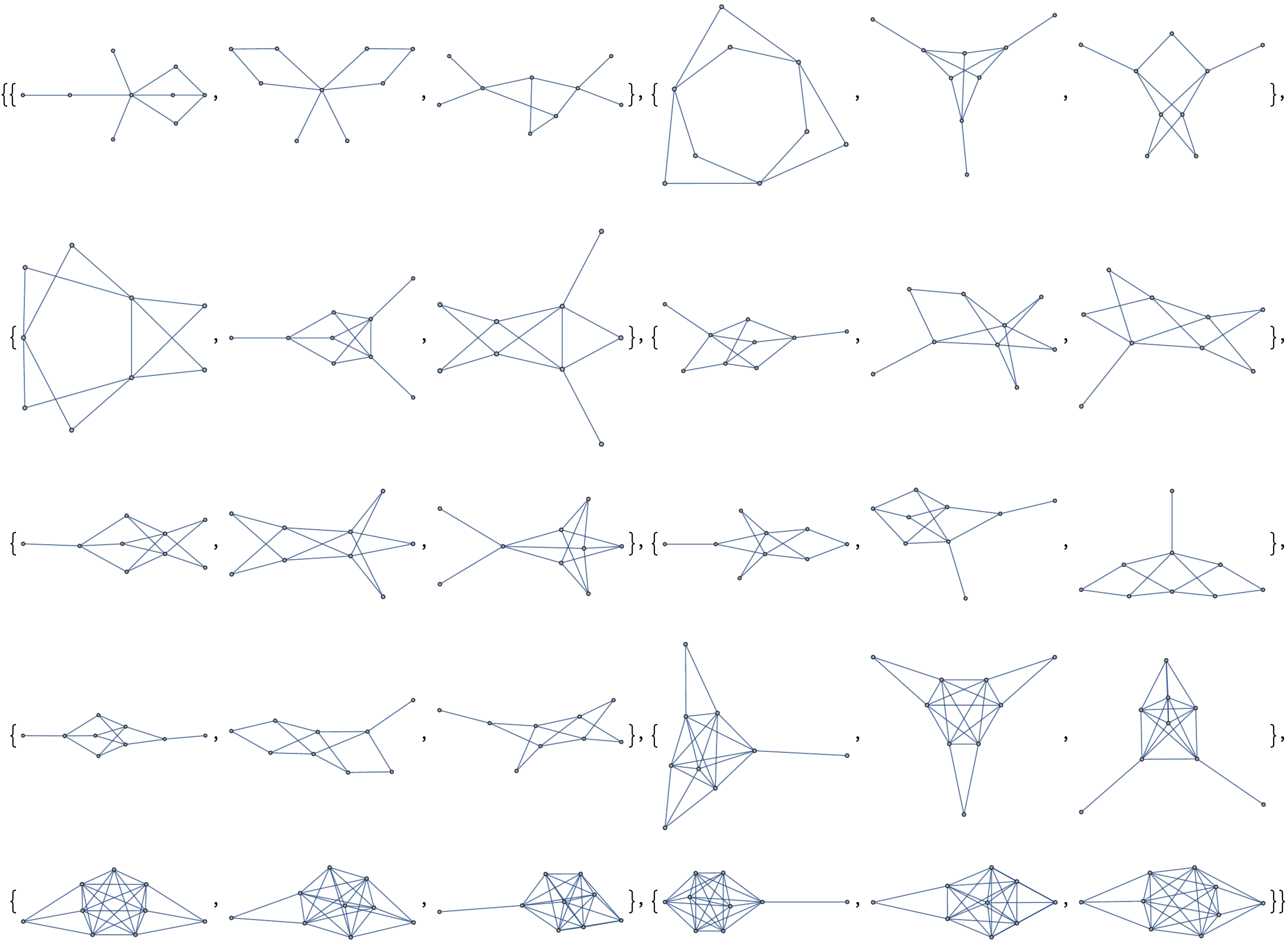}
\end{center}
\caption{All ten isospectral triplets with nine vertices. All edges have the same length. The three isospectral triplets with eight vertices are shown in the main text. These graphs were not tested for hot vertices.}
\label{fig:A4}
\end{figure}

\begin{figure}
\centering
\includegraphics[width=1.0\textwidth]{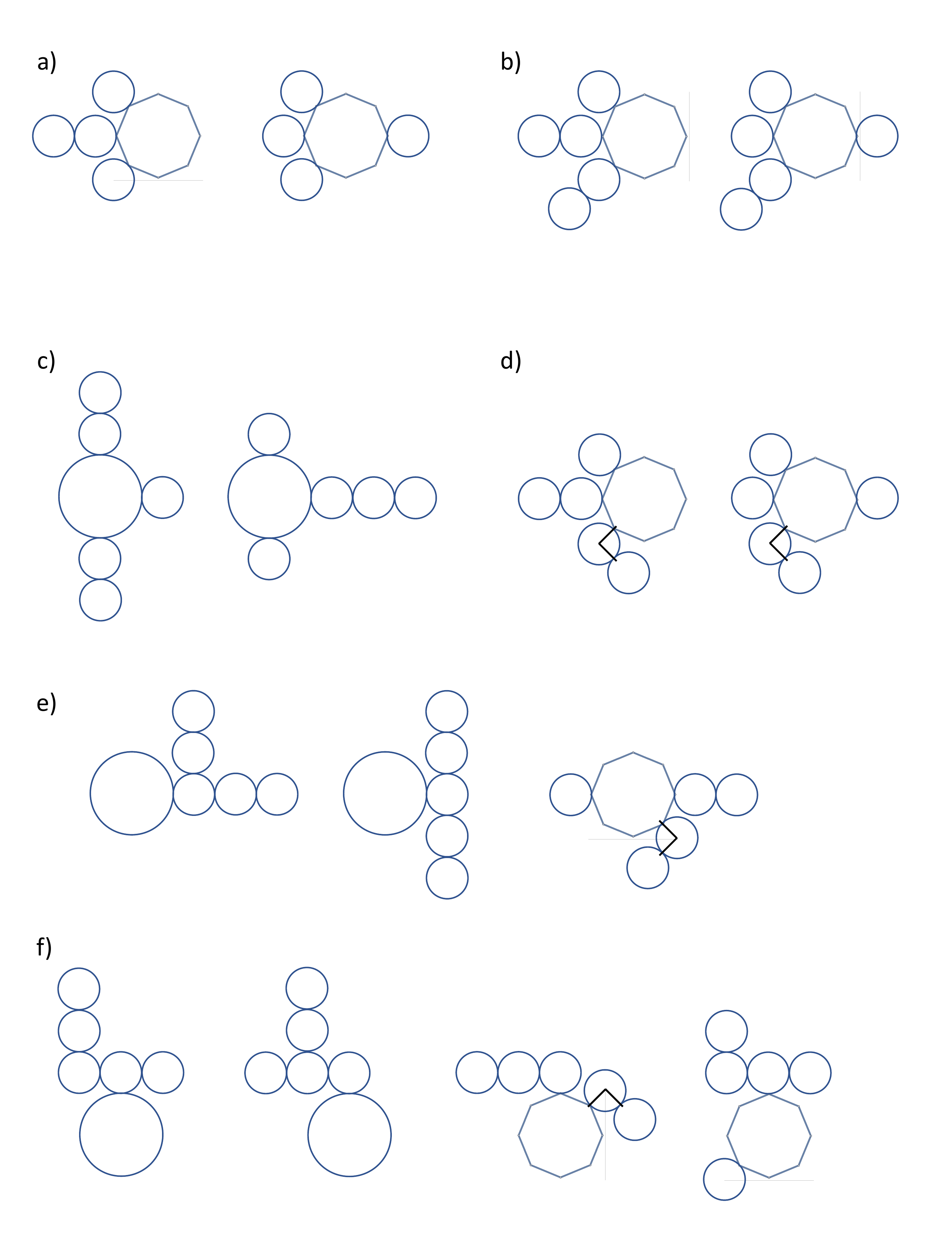}
\caption{ Examples of isospectral sets consisting of loops attached to loops. The length of the loops is four or eight and some loops are drawn as octagons to clarify the geometry. The angle between the black lines is $\pi/2$. a-d) are isospectral pairs, e) an isospectral triplet and f) an isospectral set of four. They inspired Fig. \ref{fig:dhotvertices}.}
\label{fig:A5}
\end{figure}

\begin{figure}
\centering
\includegraphics[width=1.0\textwidth]{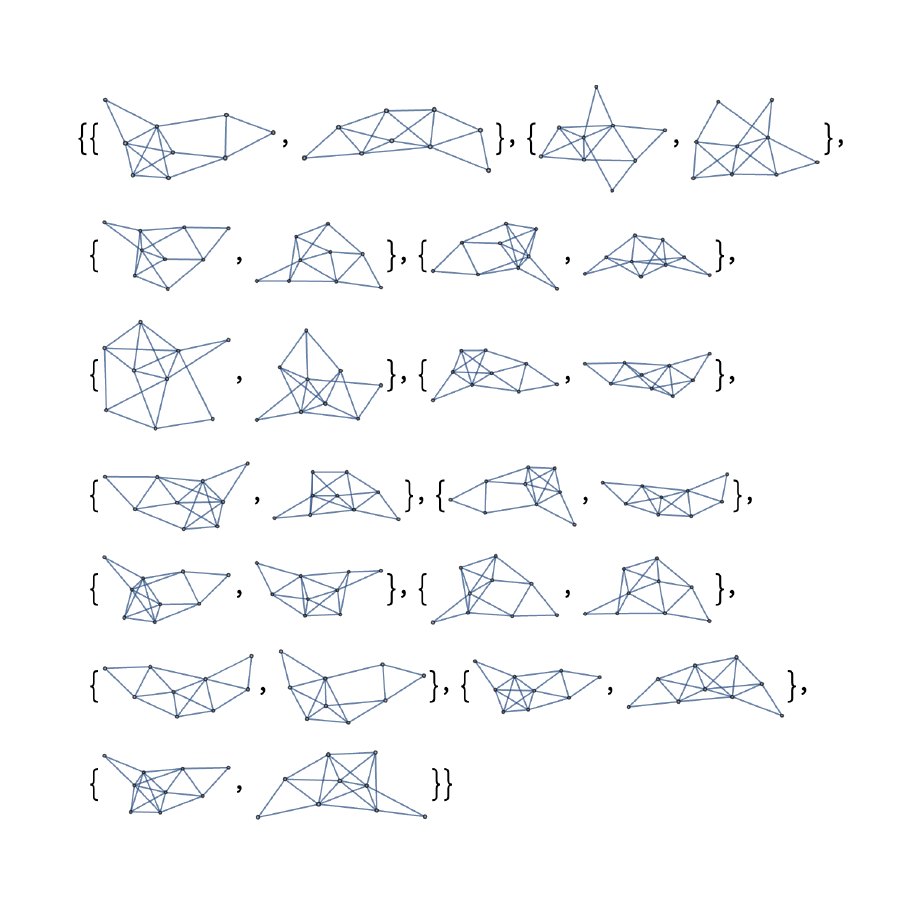}
\caption{These pairs of graphs are isospectral under $\delta_{3,4,5,6}(\alpha,\beta,\alpha,\beta)$ boundary conditions.}
\label{fig:A6}
\end{figure}

\end{document}